\documentclass[11pt]{amsart}
\usepackage{amsmath}
\usepackage{amssymb}
\usepackage{amscd}

\def\NZQ{\mathbb}               
\def\NN{{\NZQ N}}
\def\QQ{{\NZQ Q}}
\def\ZZ{{\NZQ Z}}
\def\RR{{\NZQ R}}
\def\CC{{\NZQ C}}
\def\AA{{\NZQ A}}

\newtheorem{Theorem}{Theorem}[section]
\newtheorem{Lemma}[Theorem]{Lemma}
\newtheorem{Corollary}[Theorem]{Corollary}
\newtheorem{Proposition}[Theorem]{Proposition}
\newtheorem{Remark}[Theorem]{Remark}

\newtheorem{Definition}[Theorem]{Definition}

\let\epsilon\varepsilon
\let\phi=\varphi
\let\kappa=\varkappa

\begin{document}

\title{local monomialization of   analytic maps}
\author{Steven Dale Cutkosky}
\thanks{partially supported by NSF grant DMS 1360564}

\address{Steven Dale Cutkosky, Department of Mathematics,
University of Missouri, Columbia, MO 65211, USA}
\email{cutkoskys@missouri.edu}

\keywords{Local Monomialization, Analytic Maps, \'Etoile}

\begin{abstract}
In this paper local monomialization theorems are proven for   morphisms of  complex and real  analytic spaces. This gives the generalization of the local monomialization theorem for morphisms of algebraic varieties over a field of characteristic zero proven in \cite{Ast} and \cite{LMTE} to 
analytic spaces.
\end{abstract}

\maketitle

\section{Introduction}

In this paper we prove local monomialization theorems for complex and real analytic morphisms.

A local blow up of an  analytic space $X$  is a morphism $\pi:X'\rightarrow X$ determined by a triple $(U,E,\pi)$ where $U$ is an open subset of $X$, $E$ is a closed  analytic subspace  of $U$ and $\pi$ is the composition of the inclusion of $U$ into $X$ with the blowup of $E$.

Hironaka introduced  in his work on analytic sets and maps (\cite{Hcar} and \cite{H3})  the notion of an \'etoile over a complex analytic
space $X$ to generalize a valuation of a function field of an algebraic
variety. An \'etoile $e$ over an analytic space $X$ is a subcategory of sequences of local blowups over $X$ which satisfy good properties. If $\pi:X'\rightarrow X$ belongs to $e$,  a point $e_{X'}\in X'$, called the center of $e$ on $X'$ is associated to $e$. The set $\mathcal E_X$ of all \'etoiles over  $X$, with the collection of sets $\mathcal E_{\pi}=\{e\in \mathcal E_X\mid \pi\in e\}$ for all $\pi:X'\rightarrow X$ which are products of local blow ups as a basis of a topology  is the vo\^ute \'etoil\'ee over $X$. Hironaka proved  that the map $P_X:\mathcal E_X\rightarrow X$,  defined by $P_X(e)=e_X$ is continuous, surjective and proper. The Vo\^ute \'etoil\'ee can be seen as a generalization of the Zariski Riemann manifold of an algebraic function field, but the comparison is limited. A valuation of a giant field can be associated to an \'etoile, but this valuation does not enjoy many of the good properties realized by valuations on algebraic function fields (\cite{WLM}).
The basic properties of \'etoiles are reviewed in Section \ref{Pre}.

Suppose that $\phi:Y\rightarrow X$  is a morphism of 
reduced complex analytic spaces and  that $e$ is an \'etoile over $Y$. We prove that $\phi$ can be made into a
monomial mapping at the center of $e$ after performing sequences of
local blowups of nonsingular analytic  subvarieties above $Y$ and $X$. 
We derive some consequences for complex and real analytic geometry.

\begin{Definition} Suppose that $\phi:Y\rightarrow X$ is a  morphism of  complex or real analytic manifolds, and $p\in Y$. We will say that the map $\phi$ is monomial at $p$ if there exist regular parameters $x_1,\ldots,x_m,x_{m+1},\ldots,x_t$ in $\mathcal O_{X,\phi(p)}^{\rm an}$ and $y_1,\ldots,y_n$ in $\mathcal O_{Y,p}^{\rm an}$ and $c_{ij}\in \NN$ such that
$$
\phi^*(x_i)=\prod_{j=1}^ny_j^{c_{ij}}\mbox{ for }1\le i\le m
$$
with $\mathop{rank}(c_{ij})=m$ and $\phi^*(x_i)=0$ for $m<i\le t$. 
\end{Definition}

There is a related notion of an analytic morphism $\phi:Y\rightarrow X$ being monomial  on $Y$ (Definition \ref{DefMon}). 

Our principal result is the following theorem. 

\begin{Theorem}\label{TheoremB*} Suppose that $\phi:Y\rightarrow X$ is a morphism of reduced complex analytic spaces  and $e$ is an \'etoile over $Y$. Then there exists a commutative diagram of complex analytic morphisms
$$
\begin{array}{ccc}
Y_{e}&\stackrel{\phi_{e}}{\rightarrow}&X_e\\
\beta\downarrow &&\downarrow \alpha\\
Y&\stackrel{\phi}{\rightarrow}& X
\end{array}
$$
such that $\beta\in e$, the morphisms $\alpha$ and $\beta$ are  finite products of local blow ups of nonsingular analytic sub varieties, $Y_e$ and $ X_e$ are nonsingular analytic  spaces and $\phi_e$ is a  monomial  analytic morphism at the center of $e$.

There exists a nowhere dense closed analytic subspace $F_e$ of $X_e$ such that $X_e\setminus F_e\rightarrow X$ is an open embedding and $\phi_e^{-1}(F_e)$ is nowhere dense in $Y_e$.
\end{Theorem}

The last condition on $F_e$ is always true if $\alpha, \beta$ are sequences of local blow ups and $\phi$ is regular (this concept is  defined in equation (\ref{eqout3})). A regular morphism is the analog in analytic geometry of a dominant morphism in algebraic geometry.

A stronger version of Theorem \ref{TheoremB*} is proven in Theorem \ref{TheoremB}.
The analogue of Theorem \ref{TheoremB*} for dominant morphisms of algebraic varieties (over a field of characteristic 0) dominated by a valuation was proven earlier in \cite{Ast} and \cite{LMTE}. The fact that the theorem is not true in positive characteristic was proven in \cite{C4}. It is not difficult to extend the proof of local monomialization along a valuation for dominant morphisms of characteristic zero algebraic varieties to arbitrary (not necessarily dominant) morphisms, using standard theorems from resolution of singularities.

We deduce the following Theorem \ref{TheoremC*} from Theorem \ref{TheoremB*}, using the fact that the set of \'etoiles (La Vo\^ute \'Etoil\'ee) on a complex analytic space has  some good topological properties (\cite{Hcar} and \cite{H3}).
We use in this and the following theorems stated in this introduction  the  notion of an analytic morphism $\phi:Y\rightarrow X$ of manifolds being monomial on $Y$ which is defined in Definition  \ref{DefMon}. 
The proof of Theorem \ref{TheoremC*} is obtained from Theorem \ref{TheoremB*} by utilizing techniques from \cite{HLT} and \cite{H3}. Let $K$ be a compact neighborhood of the point $p\in Y$. Theorem \ref{TheoremB*} produces for each \'etoile $e\in\mathcal E_X$ a morphism $\pi_e:Y_e\rightarrow Y$ which  lifts the initial morphism $\phi:Y\rightarrow X$ to a morphism $\phi_e:Y_e\rightarrow X_e$ which is monomial at the point $e_Y$. Since $P_Y:\mathcal E_Y\rightarrow Y$ is proper, the set $K'=P_Y^{-1}(K')$ is compact. Theorem \ref{TheoremC*} follows by extracting a finite sub cover from an open cover of $K'$ by the preimages of  open sets obtained  from the $Y_e$.

\begin{Theorem}\label{TheoremC*} Suppose that $\phi:Y\rightarrow X$ is a morphism of reduced  complex  analytic spaces and $p\in Y$. Then there exists a finite number $t$ of commutative diagrams of complex  analytic morphisms
$$
\begin{array}{ccc}
Y_{i}&\stackrel{\phi_{i}}{\rightarrow}&X_i\\
\beta_i\downarrow &&\downarrow \alpha_i\\
Y&\stackrel{\phi}{\rightarrow}& X
\end{array}
$$
for $1\le i \le t$ 
such that each $\beta_i$ and $\alpha_i$ are finite products of local blow ups of nonsingular analytic sub varieties, $Y_i$ and $ X_i$ are smooth analytic  spaces and $\phi_i$ is a  monomial  analytic morphism. Further, there exist compact subsets $K_i$ of $Y_i$ such that $\cup_{i=1}^t\beta_i(K_i)$ is a compact neighborhood of $p$ in $Y$.

There exist nowhere dense closed analytic subspaces $F_i$ of $X_i$ such that $X_i\setminus F_i\rightarrow X$ are open embeddings and $\phi_i^{-1}(F_i)$ is nowhere dense in $Y_i$.
\end{Theorem}

A stronger version of Theorem \ref{TheoremC*} is proven in Theorem \ref{TheoremC} below.


We obtain corresponding theorems for real analytic morphisms.

\begin{Theorem}\label{TheoremC'*} Suppose that $Y$ is a real analytic manifold, $X$ is a reduced real analytic space and $\phi:Y\rightarrow X$ is a real analytic morphism.
Then there exists a finite number $t$ of commutative diagrams of complex  analytic morphisms
$$
\begin{array}{ccc}
Y_{i}&\stackrel{\phi_{i}}{\rightarrow}&X_i\\
\beta_i\downarrow &&\downarrow \alpha_i\\
Y&\stackrel{\phi}{\rightarrow}& X
\end{array}
$$
for $1\le i \le t$ 
such that each $\beta_i$ and $\alpha_i$ are finite products of local blow ups of nonsingular analytic sub varieties, $Y_i$ and $ X_i$ are smooth analytic  spaces and $\phi_i$ is a  monomial  analytic morphism. Further, there exist compact subsets $K_i$ of $Y_i$ such that $\cup_{i=1}^t\beta_i(K_i)$ is a compact neighborhood of $p$ in $Y$.

There exist nowhere dense closed analytic subspaces $F_i$ of $X_i$ such that $X_i\setminus F_i\rightarrow X$ are open embeddings and $\phi_i^{-1}(F_i)$ is nowhere dense in $Y_i$.

\end{Theorem}

A stronger version  of Theorem \ref{TheoremC'*} is proven in Theorem \ref{TheoremC'}.

An application of Theorem \ref{TheoremC'*}, showing that Hironaka's rectilinearization theorem can be deduced from local monomialization, is given in \cite{C11}. The rectilinearization theorem was first proven by Hironaka in \cite{H3}. Different proofs have been given by Denef and Van Den Dries \cite{DV} and  Bierstone and Milman \cite{BM3}.

Because of the existence of examples such as the Whitney Umbrella, $x^2-zy^2=0$,  it is not possible for Theorem \ref{TheoremC'*} to hold when $Y$ is only assumed to be a reduced real analytic space. However, using a generalization of the notion of resolution of singularities by Hironaka for real analytic spaces we can generalize Theorem \ref{TheoremC'*} to arbitrary reduced analytic spaces.

We  recall the definition of a smooth real analytic filtration of a real analytic space.

\begin{Definition}(Definition 5.8.2 \cite{H3}) Let $X$ be a real analytic space. A smooth real analytic filtration of $X$ is a sequence of closed real analytic subspaces $\{X^{i}\}_{0\le i<\infty}$ of $X$ such that 
\begin{enumerate}
\item[1)] $X^{(0)}=|X|$ and $X^{(i)}\supset X^{(i+1)}$ for all $i\ge 0$.
\item[2)] $\{X^{(i)}\}$ is locally finite at every point $p\in X$.
\item[3)] $X^{(i)}\setminus X^{(i+1)}$ is smooth.
\end{enumerate}
\end{Definition}

If $X$ is a reduced real analytic space which is countable at infinity, then $X$ has a smooth real analytic filtration (Proposition 5.8 \cite{H3}).

Using resolution of singularities, Hironaka deduces the following result.

\begin{Proposition}\label{TheoremE}(Desingularization I. (5.10) \cite{H3})
Suppose that $X$ is a real analytic space and $p\in X$. Then there exists an open neighborhood $U$ of $p$ in $X$, a finite smooth real analytic filtration $\{U^{(i)}\}$ on $U$ and real analytic morphisms $\pi^{(i)}:\overline U^{(i)}\rightarrow U^{(i)}$ such that 
\begin{enumerate}
\item[1)] Each $\overline U^{(i)}$ is smooth and $\pi^{(i)}$ is a sequence of blowups of smooth sub varieties. 
\item[2)] $(\pi^{(i)})^{-1}(U^{(i+1)})$ is nowhere dense in $\overline U^{(i)}$ and
\item[3)] $\pi^{(i)}$ induces an isomorphism $\overline U^{(i)}\setminus (\pi^{(i)})^{-1}(U^{(i+1)})\rightarrow U^{(i)}\setminus U^{(i+1)}$.
\end{enumerate}
In particular, $U=\cup_{i\ge 0}\pi^{(i)}(\overline U^{(i)})$.
\end{Proposition}

We deduce the following theorem from Theorem \ref{TheoremC'*} and Proposition \ref{TheoremE}.

\begin{Theorem}\label{TheoremD*} Suppose that $\phi:Y\rightarrow X$ is a real analytic morphism of reduced real analytic spaces  and $p\in Y$. Then there exists a finite number $t$ of commutative diagrams of real  analytic morphisms
$$
\begin{array}{ccc}
Y_{i}&&\\
\beta_i\downarrow &\phi_i\searrow&\\
Y_i^*&&X_i\\
\gamma_i\downarrow&&\downarrow\alpha_i\\
Y&\stackrel{\phi}{\rightarrow}& X
\end{array}
$$
for $1\le i \le t$ 
such that each $\gamma_i:Y_i^*\rightarrow Y$ is a resolution of singularities of a component of a smooth real analytic filtration of a neighborhood of $p$ in $Y$, $\gamma_i$, $\beta_i$ and $\alpha_i$ are  finite products of local blow ups of nonsingular analytic sub varieties,  $Y_i$ and $ X_i$ are smooth analytic  spaces and $\phi_i$ is a  monomial  analytic morphism. Further, there exist compact subsets $K_i$ of $Y_i$ such that $\cup_{i=1}^t\gamma_i\beta_i(K_i)$ is a compact neighborhood of $p$ in $Y$.

There exist nowhere dense closed analytic subspaces $F_i$ of $X_i$ such that $X_i\setminus F_i\rightarrow X$ are open embeddings and $\phi_i^{-1}(F_i)$ is nowhere dense in $Y_i$.
\end{Theorem}

There are a number of local theorems in  analytic geometry, including by Hironaka on the local structure of subanalytic sets
(\cite{Hcar} and \cite{H3}), especially the rectilinearization theorem, by  Hironaka,  the theorem by Lejeune and Teissier \cite{HLT} and by Hironaka \cite{H3} on local flattening, by Cano on local resolution of  3-dimensional vector fields (\cite{Ca}), by Denef and van den Dries \cite{DV} and Bierstone and Milman (\cite{BM3}) 
 on the structure of semianalytic and subanalytic sets, by Lichtin (\cite{Li}, \cite{Li2} )to construct local monomial forms of analytic mappings in low dimensions to prove convergence of series
and  by Belotto on local resolution and monomialization of foliations (\cite{B1}). A global form of the result of \cite{Ca} holds on an algebraic three fold (over an algebraically closed field of characteristic zero) by combining the theorem of \cite{Ca} with the patching theorem of Piltant in \cite{Pi}.

For dominant morphisms of algebraic varieties of characteristic zero, local monomialization along an arbitrary valuation is proven in 
\cite{Ast} and \cite{LMTE}.  It is shown in \cite{C4} that local monomialization (and even ``weak'' local monomialization where the vertical arrows are only required to be birational maps) is not true along an arbitrary valuation in positive characteristic, even for varieties of dimension two.

Global monomialization (toroidalization) has been proven for varieties over algebraically closed fields of characteristic zero for dominant morphisms from 
a projective 3-fold (\cite{C5}, \cite{C6} and \cite{C7}). 
Weak toroidalization (weak global monomialization), where the vertical arrows giving a toroidal map are only required to be birational is proven globally for algebraic varieties of   characteristic zero by  Abramovich and Karu \cite{AK} and Abramovich, Denef and Karu \cite{ADK}. 
Applications of this theorem  to  quantifier elimination  and other important problems in logic are given by Denef in \cite{De} and \cite{De1}.

The proof of local monomialization in characteristic zero function fields given in \cite{Ast} and \cite{LMTE} does not readily extend to the case of analytic morphisms. This is because the methods from valuation theory that are used there do not behave well under the infinite extensions of quotient fields of local rings which take place under local blow ups associated to an \'etoile. The behavior of a valuation associated to an \'etoile  which has rank larger than 1 is particularly wild (examples are given in \cite{WLM}), and the reduction to rank 1 valuations (the value group is an ordered subgroup of $\RR$) 
in the proofs of \cite{Ast} and \cite{LMTE} does not extend to a higher rank valuation which is associated to an \'etoile.  New techniques are developed in this paper which are not sensitive to the rank of a valuation.  The notion of ``independence of variables''
for an \'etoile, Definition  \ref{Def53}, replaces the notion of the rational rank of a (rank 1) valuation which is used in \cite{Ast} and \cite{LMTE}. If $e$ is an \'etoile over an irreducible complex analytic space $X$,  then we have (as in the classical case of function fields) by Lemma 5.3 \cite{WLM}
the inequalities 
$$
{\rm rank }V_e\le {\rm ratrank }V_e\le \dim X
$$
where $V_e$ is the valuation ring associated to $e$.

The proofs of this paper can be adapted to give simpler proofs of the local monomialization theorem for characteristic zero algebraic function fields of \cite{Ast} and \cite{LMTE}. However, two sources of complexity in the proofs of \cite{Ast} and \cite{LMTE} do not exist in the case of complex analytic morphisms, and cannot (readily) be eliminated. They are the problem of residue field extension of local rings, and the problem of approximation of formal (analytic) constructions to become algebraic.

The proofs of this paper, and the difficulties which must be overcome are related to the problems which arise in resolution of vector fields and differential forms (\cite{S}, \cite{Ca} , \cite{P}, \cite{BBGM}) and in resolution of singularities in positive characteristic (some papers illustrating this are \cite{Ab},  \cite{Ab2}, \cite{C8}, \cite{C9}, \cite{Ha1}, \cite{Ha2}, \cite{BV}, \cite{CP1}, \cite{CP2}, \cite{CP3}). A common difficulty to monomialization of morphisms, resolution of singularities in positive characteristic and resolution of vector fields is the possibility of a natural order going up after the blow up of an apparently suitable nonsingular sub variety.

We thank Jan Denef for suggesting the local monomialization problem for analytic morphisms, and for discussion, encouragement and explanation of possible applications. We also thank Bernard Teissier for discussions on this and related problems.
We thank  reviewers for their helpful comments and careful reading.

\section{A brief overview of the proof} 

 In this section we give an outline of the proof of Theorem \ref{TheoremB*} (and a stronger version, Theorem \ref{TheoremB}).  Suppose that $\phi:Y\rightarrow X$ is a complex analytic morphism of complex manifolds and $e$ is an \'etoile over $Y$. 
 The first step is to reduce, using  Proposition \ref{TheoremA} in Section \ref{Pre}, to the assumption that $\phi$ is quasi regular; that is, if we have a commutative diagram
 $$
 \begin{array}{rcl}
 Y_1&\stackrel{\phi_1}{\rightarrow}&X_1\\
 \beta\downarrow&&\downarrow\alpha\\
 Y&\stackrel{\phi}{\rightarrow}&X
 \end{array}
 $$
 with $\beta\in e$, and $\alpha$, $\beta$ products of local blowups of nonsingular analytic sub varieties then
  $\phi_1^*:\mathcal O_{X_1,\phi_1(e_{Y_1})}^{\rm an}\rightarrow \mathcal O_{Y_1,e_{Y_1}}^{\rm an}$ is injective.  This proof only uses the statement of the theorem of resolution of singularities.  
  In fact, it is true that if $\phi$ is quasi regular then $\phi$ is regular (so $\hat{\phi}_1^*:\hat{\mathcal O}_{X_1,\phi_1(e_{Y_1})}^{\rm an}\rightarrow \hat{\mathcal O}_{Y_1,e_{Y_1}}^{\rm an}$ is also injective), as can be deduced from the sophisticated local flattening theorem of Hironaka, Lejeune and Teissier \cite{HLT} and with a different proof by Hironaka in \cite{H3}. This deduction is shown in \cite{WLM}. However, we do not need this for our proof, and in fact deduce it in Corollary \ref{RegQR} from our proof. The fact that we only assume quasi regularity, and not regularity, is addressed in the proof of Theorem \ref{TheoremB*} in Proposition \ref{Theorem502}.
  
  With the assumption that $\phi$ is quasi regular, we have reduced to the proof of Theorem \ref{Theorem3}, and  we have that $e$ induces a restricted \'etoile on $X$ (as explained at the end of Section \ref{Pre}).  We will also need the fact, explained in Section  \ref{Pre}, that there is a valuation $\nu_e$ with valuation ring $V_e$ on the union of quotient fields of local rings at the center of $e$ of sequences of local blowups by nonsingular sub varieties above $Y$ which are in $e$.

    The most important types of transformations (sequences of local blow ups or change of variables) used in the proof are the generalized monoidal transformation, GMT and the simple GMT (SGMT), which are  defined in Section \ref{SecGMT}. The full set of transformations used are defined after the proof of Lemma \ref{Lemma35} in Section \ref{SecGMT}.
    A GMT associates to a given set $x_1,\ldots,x_n$ of variables another set $\overline x_1,\ldots,\overline x_n$ (which are parameters at the point on the corresponding birational extension determined by the \'etoile $e$), defined by 
  $$
  \overline x_i=\prod_{j=1}^n(x_j+\alpha_j)^{a_{ij}}
  $$
  where $A=(a_{ij})$ is a matrix of natural numbers with $\mbox{Det }A=\pm1$ and $\alpha_j\in \CC$. 
  
  A collection of variables $x_1,\ldots,x_n$ is called {\it independent} if every GMT in $x_1,\ldots, x_n$ is monomial (all $\alpha_j=0$). This is a crucial concept in the proof.   A critical fact is that a  GMT  preserves independence of variables.

      In the proof, we inductively construct commutative diagrams
 $$
 \begin{array}{rcl}
 \tilde Y&\stackrel{\tilde\phi}{\rightarrow}&\tilde X\\
 \downarrow&&\downarrow\\
 Y&\stackrel{\phi}{\rightarrow}&X
 \end{array}
 $$  
  where the vertical morphisms are products of local  blow ups of nonsingular analytic sub varieties which are in $e$ such that there 
  exist regular parameters $x_1,\ldots,x_m$ in $\mathcal O_{\tilde X,e_{\tilde X}}^{\rm an}$ and $y_1,\ldots,y_n$ in 
$\mathcal O_{\tilde Y,e_{\tilde Y}}^{\rm an}$ such that $y_1,\ldots,y_s$ are independent but $y_1,\ldots, y_s,y_i$ are dependent for all $i$ with $s+1\le i\le n$, $x_1,\ldots,x_r$ are independent, and identifying $x_i$ with $\tilde\phi^*(x_i)$, there is an expression for some $l$
\begin{equation}\label{Outline1}
\begin{array}{lll}
x_1&=&y_1^{c_{11}}\cdots y_s^{c_{1s}}\\
&&\vdots\\
x_r&=& y_1^{c_{r1}}\cdots y_s^{c_{rs}}\\
x_{r+1}&=& y_{s+1}\\
&&\vdots\\
x_{r+l}&=& y_{s+l}.
\end{array}
\end{equation}

We necessarily have that $C=(c_{ij})$ has rank $r$ (by Lemma \ref{Lemma25}) with our assumptions. 
We will say that the variables $(x,y)=(x_1,\ldots,x_m; y_1,\ldots,y_n)$ are {\it prepared} of type $(s,r,l)$ if all of the above conditions hold.
The above diagram (\ref{Outline1}) is labeled as equation (\ref{eq5})  in Section \ref{SecTrans}, where it is introduced in the proof.
We  say that $(s_1,r_1,l_1)\ge (s,r,l)$ if $s_1\ge s$, $r_1\ge r$ and $r_1+l_1\ge r+l$, and that $(s_1,r_1,l_1)>(s,r,l)$ if $(s_1,r_1,l_1)\ge (s,r,l)$ and $s_1>s$ or $r_1>r$ or $r_1+l_1>r+l$.

Theorem \ref{Theorem3}, and thus 
Theorem \ref{TheoremB*}, is a consequence of induction using Proposition \ref{Theorem2}, which shows that if $\tilde \phi$ is not monomial, and an  expression (\ref{Outline1}) holds, then we can construct some more local blow ups $Y_1\rightarrow \tilde Y$ and $X_1\rightarrow \tilde X$ of nonsingular sub varieties, with $Y_1\rightarrow Y\in e$  such that we have a resulting morphism  $\phi_1:Y_1\rightarrow X_1$ giving equations (\ref{Outline1}) with an increase $(s_1,r_1,l_1)>(s,r,l)$.   

We will now say a little bit about the proof of Proposition \ref{Theorem2}, and the necessary results preceding it.  This is accomplished in Section \ref{Mon}. 
We start with an expression (\ref{Outline1}), and then we perform a sequence of transformations which maintain the form (\ref{Outline1}) to also put $x_{r+l+1}$  into a monomial form consistent with (\ref{Outline1}). We may assume that there is no change in $(r,s,l)$ under these transformations (until the very last step), since otherwise we have already obtained a proof of the induction statement.

 We make use of the following method to reduce the order of a function along a valuation, taking a Tschirnhaus transformation (Lemma \ref{Lemma9}) and then performing sequences of blow ups to make the coefficients monomials (times units), and then performing a transformation of type 4) (defined after the proof of Lemma \ref{Lemma35} in Section \ref{SecGMT}) to get a reduction in multiplicity. This is a variation on the reduction method of Zariski in \cite{LU}, except we consider  valuations of arbitrary rank, and  use the  Tschirnhaus transformation which was introduced by Abhyankar and developed by Hironaka. 
This method is used repeatedly through out the proofs. 

Another important method is developed in Section \ref{DecSeries}. We  define the notion of  a formal series $g$ in $\CC[[y_1,\ldots,y_{s+l}]]$ to be {\it algebraic} over $x_1,\ldots,x_{r+l}$
in Definition \ref{Def500}.  We consider this notion through the decomposition of a series  $g$ expressed in (\ref{eq60}) and (\ref{eq60}). This decomposition was introduced in \cite{LMTE}.

 We perform 10 types of transformations to achieve the proof of Proposition \ref{Theorem2}, which are listed after Lemma \ref{Lemma35}. The basic transformations are 1), 2), 4) and 9) which are generalized monoidal transforms, and 3) and 10), which are generally used to make a Tschirnhaus transformation. 
 
  In Lemma \ref{Lemma2}, it is shown that we can perform transformations which preserve the form (\ref{Outline1}) to transform a given  element $g\in \CC\{\{y_1,\ldots,y_{s+l}\}\}$ into a monomial in $y_1,\ldots,y_s$ times a unit. The decomposition of Section \ref{DecSeries} is essential in the proof of this lemma. From this lemma, we obtain in Lemma \ref{Lemma3} that if $g\in \CC\{\{y_1,\ldots,y_{s+l}\}\}$ is not algebraic over $x_1,\ldots,x_r$,  then we can perform transformations which preserve the form \ref{Outline1} to obtain that
 \begin{equation}\label{Outline2}
 g=P+y_1(1)^{d_1}\cdots y_s(1)^{d_s}
 \end{equation}
 where $P$ is algebraic over $x_1(1),\ldots,x_{r+l}(1)$ and $y_1(1)^{d_1}\cdots y_s(1)^{d_s}$ is not algebraic over $x_1(1),\ldots, x_r(1)$.
 
 In Proposition \ref{Theorem502}, we  show that the natural map of formal power series 
 $$
 \CC[[x_1,\ldots,x_{r+l+1}]]\rightarrow \CC[[y_1,\ldots,y_n]]
 $$
  is an inclusion. (Since $\phi$ is quasi regular, we must have that the map 
  $$
  \CC\{\{x_1,\ldots,x_m\}\}\rightarrow \CC\{\{y_1,\ldots,y_n\}\}
  $$
   is injective.)
 
 Lemmas \ref{Lemma4} and \ref{Lemma5} generalize Lemmas \ref{Lemma2} and \ref{Lemma3} to the case when $g\in \CC\{\{y_1,\ldots,y_n\}\}$.
 
 In Proposition  \ref{Theorem1}, we now deduce that there is a sequence of transforms preserving the form (\ref{Outline1}) such that 
$$
x_{r+l+1}(1)=P+y_1(1)^{d_1}\cdots y_s(1)^{d_s}
$$
with $P\in \CC\{\{y_1(1),\ldots,y_{s+l}(1)\}\}$ algebraic over $x_1(1),\ldots,x_{r+l}(1)$ and 
$y_1(1)^{d_1}\cdots y_s(1)^{d_s}$ not algebraic over $x_1(1),\ldots,x_r(1)$
or we have an expression
$$
x_{r+l+1}(1)=P+y_1(1)^{d_1}\cdots y_s(1)^{d_s}y_{s+l+1}(1)
$$
with $P \in \CC\{\{y_1(1),\ldots,y_{s+l}(1)\}\}$ algebraic over $x_1(1),\ldots,x_{r+l}(1)$. It now remains to perform a sequence of transformations which remove the $P$ term. This is accomplished in Proposition \ref{Theorem2}.

\section{Preliminaries on analytic maps and \'etoiles}\label{Pre}

We require that an analytic space be Hausdorff.

Suppose that $X$ is a complex or real analytic manifold and $p\in X$. Let $K=\CC$ or $\RR$. Suppose that $x_1,\ldots,x_m$ are regular parameters in 
$\mathcal O_{X,p}^{\rm an}$.
Then the completion  $\hat{\mathcal O}_{X,p}^{\rm an}$ of $\mathcal O_{X,p}^{\rm an}$ with respect to its maximal ideal is  the ring of formal power series  $K[[x_1,x_2,\ldots,x_m]]$. The ring $\mathcal O_{X,p}^{\rm an}$ is then identified with the  subring $K\{\{ x_1,\ldots,x_m\}\}$ of convergent power series. By Abel's theorem, the formal series
$$
f=\sum a_{i_1,\ldots,i_m}x_1^{i_1}\cdots x_m^{i_m}\in K[[x_1,\ldots,x_m]]
$$
is a convergent power series if and only if there exist positive real numbers $r_1,\ldots,r_m, M$ such that 
\begin{equation}\label{eq42}
||a_{i_1,\ldots,i_m}||r_1^{i_1}\cdots r_m^{i_m}\le M
\end{equation}
for every $i_1,\ldots,i_m$.

The local ring $\mathcal O_{X,p}^{\rm an}$ of a point $p$ on a complex or real analytic space $X$ is noetherian and henselian by Theorem 45.5 and fact 43.4 \cite{N}. The local ring $\mathcal O_{X,p}^{\rm an}$ is excellent by Section 18 \cite{EGAIV} (or Theorem 102, page 291 \cite{Ma}
and by (ii) of Scholie 7.8.3 \cite{EGAIV}).

A local blow up of an  analytic space $X$ (page 418 \cite{Hcar} or Section 1 \cite{H3}) is a morphism $\pi:X'\rightarrow X$ determined by a triple $(U,E,\pi)$ where $U$ is an open subset of $X$, $E$ is a closed  analytic subspace  of $U$ and $\pi$ is the composition of  the inclusion of $U$ into $X$ with
 the blowup of $E$. If $\pi:X^*\rightarrow X$ is a sequence of local blowups, then taking $F$ to be the union of the preimages on $X^*$ of the closed subspaces that are blown up in constructing $\pi$, we have that $F$ is a closed analytic subspace of $X^*$ such that the induced morphism $X^*\setminus F\rightarrow X$ is an open embedding.

Suppose that $X$ is a real or complex analytic manifold. A divisor $E$ on $X$ is a simple normal crossings (SNC) divisor if the support of $E$ is a union of irreducible smooth codimension 1 sub varieties of $X$ which intersect transversally.

Suppose that $\phi:Y\rightarrow X$ is a morphism of complex or real analytic manifolds. Gabrielov \cite{Gab} (also \cite{BM2} for a survey of this and related topics) has defined three ranks of $\phi$ at a point $q$ of $Y$.
Let $p=\phi(q)$. We have induced local homomorphisms
$$
\phi^*:\mathcal O_{X,p}^{\rm an}\rightarrow \mathcal O_{Y,q}^{\rm an}
$$
and 
$$
\hat \phi^*:\hat{\mathcal O}_{X,p}^{\rm an}\rightarrow \hat{\mathcal O}_{Y,q}^{\rm an}
$$
on the completions. We define
$$
\begin{array}{lll}
r_q(\phi)&=&
\mbox{ generic rank}\\
&=&\mbox{ largest rank of the tangent mapping of $\phi$ in a small open neighborhood of $q$,}
\end{array}
$$
$$
r_q^{\mathcal F}(\phi)=\dim \hat{\mathcal O}_{X,p}^{\rm an}/{\rm Kernel\,}\hat\phi^*
$$
$$
r_q^{\mathcal A}(\phi)=\dim \mathcal O_{X,p}^{\rm an}/{\rm Kernel\,}\phi^*.
$$
We have
\begin{equation}\label{Gineq}
r_q(\phi)\le r_q^{\mathcal F}(q)\le r_q^{\mathcal A}(\phi)\le \dim X.
\end{equation}
We will say that $\phi$ is regular at $q$ if all three of these ranks are equal to the dimension of $X$,
\begin{equation}\label{eqout3}
r_q(\phi)=r_q^{\mathcal F}(\phi)=r_q^{\mathcal A}(\phi)=\dim X.
\end{equation}
If $Y$ is a connected manifold and $\phi$ is regular at a point $q\in Y$ then $\phi$ is regular everywhere on $Y$. In this case we will say that $\phi$ is regular.

The dimension of a subset $E$ of a complex manifold $X$ at a point $p\in X$ is (page 152 \cite{L})
$$
\dim_pE=\sup\{\dim \Gamma\mid \Gamma\mbox{ is a sub manifold of $U$ contained in $E\cap U$ }\}
$$
where $U$ is a small neighborhood of $p$ in $X$.

If $\phi:Y\rightarrow X$ is a complex analytic morphism of complex manifolds, $q\in Y$ and $p=\phi(q)$, then $\dim_p\phi(U)=r_q(\phi)$ if $U$ is a sufficiently small neighborhood of $q$ in $Y$. 

If $E$ is a closed analytic subset of the complex manifold $X$ and $p\in E$, then
 $$
 \dim_pE=\dim\mathcal O^{\rm an}_{E,p}
 $$  
where $\dim\mathcal O^{\rm an}_{E,p}$ is the Krull dimension of the local ring $\mathcal O^{\rm an}_{E,p}$.

For real analytic spaces, we use  the topological dimension ${\rm T}$-$\dim_p$, which is defined analogously (Section 5 of \cite{H3}). Rank and dimension are also discussed in \cite{BM2}, along with some illustrative examples.

An \'etoile is defined in Definition 2.1 \cite{Hcar}. An \'etoile $e$ over a complex analytic space $X$ is defined as a subcategory of the category of sequences of local blow ups over $X$.

A sequence of local blow ups of $X$ is the composite of a finite sequence of local blow ups $(U_i,E_i,\pi_i)$. 

Let $X$ be a complex analytic space. $\mathcal E(X)$ will denote the category of morphisms $\pi:X'\rightarrow X$ which are a sequence of local blow ups.
For $\pi_1:X_1\rightarrow X\in \mathcal E(X)$ and $\pi_2: X_2\rightarrow X\in \mathcal E(X)$, ${\rm Hom}(\pi_1,\pi_2)$ denotes the $X$-morphisms
$X_2\rightarrow X_1$ (morphisms which factor $\pi_1$ and $\pi_2$). The set ${\rm Hom}(\pi_1,\pi_2)$ has at most one element.

\begin{Definition}\label{EtoileDef} (Definition 2.1 \cite{Hcar}) Let $X$ be a complex analytic space. An \'etoile over $X$ is a subcategory $e$ of $\mathcal E(X)$ having the following properties:
\begin{enumerate}
\item[1)] If $\pi:X'\rightarrow X\in e$ then $X'\ne \emptyset$.
\item[2)] If $\pi_i\in e$ for $i=1,2$, then there exists $\pi_3\in e$ which dominates $\pi_1$ and $\pi_2$; that is, ${\rm Hom}(\pi_3,\pi_i)\ne 0$
for $i=1,2$.
\item[3)] For all $\pi_1:X_1\rightarrow X\in e$, there exists $\pi_2:X_2\rightarrow X\in e$ such that there exists $q\in {\rm Hom}(\pi_2,\pi_1)$, and the image $q(X_2)$ is relatively compact in $X_1$.
\item[4)] (maximality) If $e'$ is a subcategory of $\mathcal E(X)$ that contains $e$ and satisfies the above conditions 1) - 3), then $e'=e$.
\end{enumerate}
\end{Definition}

The set of all \'etoiles over $X$ is denoted by $\mathcal E_X$.

Using property 3), Hironaka shows that
for $e\in \mathcal E_X$, and $\pi:X'\rightarrow X\in e$, there exists a uniquely determined point $p_{\pi}(e)\in X'$ (which we will also denote by $e_{X'}$) which has the property that
if $\alpha:Z\rightarrow X\in e$ factors as
$$
Z\stackrel{\beta}{\rightarrow} X' \stackrel{\pi}{\rightarrow} X,
$$
then $\beta(p_{\alpha}(e))=p_{\pi}(e)$.  We will also call $e_{X'}$ the center of $e$ on $X'$.

  The \'etoile associates a point $e_X\in X$ to $X$ and if $\pi_1:X_1\rightarrow U$ is a local blow up of $X$ such that $e_X\in U$ then $\pi_1\in e$ and $e_{X_1}\in X_1$ satisfies $\pi_1(e_{X_1})=e_X$. If $\pi_2:X_2\rightarrow U_1$ is a local blow up of $X_1$ such that $e_{X_1}\in U_1$ then $\pi_1\pi_2\in e$ and $e_{X_2}\in X_2$ satisfies $\pi_2(e_{X_2})=e_{X_1}$. Continuing in this way, we can construct sequences of local blow ups
$$
X_n\stackrel{\pi_n}{\rightarrow} X_{n-1}\rightarrow \cdots\rightarrow X_1\stackrel{\pi_1}{\rightarrow} X
$$
such that $\pi_1\cdots\pi_i\in e$, with associated points $e_{X_i}\in X_i$ such that $\pi_i(e_{X_i})=e_{X_{i-1}}$ for all $i$.

In Section 5 of  \cite{WLM} it is shown that a valuation can be naturally associated to an \'etoile. We will summarize this construction here.

Suppose that $X$ is a reduced complex analytic space and  $e$ is an \'etoile over $X$. We will say that $\pi:X_n\rightarrow X\in e$ is nonsingular if $\pi$ factors as a sequence of local blowups
$$
X_n\rightarrow X_{n-1}\rightarrow \cdots \rightarrow X_1\rightarrow X
$$
 such that $X_i$ is nonsingular for $i\ge 1$. The set of local rings  $A_{\pi}:=\mathcal O_{X_n,e_{X_n}}^{\rm an}$ such that $\pi$ is nonsingular is a directed set, as is the set of quotient fields $K_{\pi}$ of the $A_{\pi}$ (Lemma 4.3 and Definition 3.2 \cite{WLM}). Let 
$$
\Omega_e=\lim_{\rightarrow}K_{\pi}\mbox{ and }V_e=\lim_{\rightarrow}A_{\pi}.
$$

Then $V_e$ is a valuation ring of the field $\Omega_e$ whose residue field is $\CC$ (Lemma 6.1 \cite{WLM}).

We now summarize some further results from \cite{Hcar}. Let $X$ be a complex analytic space. Let $\mathcal E_X$ be the set of all \'etoiles over $X$ and for
$\pi:X_1\rightarrow X$ a product of local blow ups, let 
\begin{equation}\label{eqopen}
\mathcal E_{\pi}=\{e\in \mathcal E_X\mid \pi\in e\}.
\end{equation}
Then the $\mathcal E_{\pi}$ form a basis for a topology on $\mathcal E_X$. The space $\mathcal E_X$ with this topology is called the vo\^ ute \'etoil\'ee over $X$ (Definition 3.1 \cite{Hcar}). The vo\^ ute \'etoil\'ee is a generalization to complex analytic spaces of the Zariski Riemann manifold of a variety $Z$ in algebraic geometry (Section 17, Chapter VI \cite{ZS2}).  

The fields $\Omega_e$ are gigantic, while the points of the Zariski Riemann manifold of a variety $Z$ are just (equivalence classes) of valuations of the function field $k(Z)$ of $Z$, so many of the good properties of valuations of the function field do not hold for the valuation induced by an \'etoile.

We have a canonical map $P_X:\mathcal E_X\rightarrow X$ defined by $P_X(e)=e_X$ which is continuous, surjective and proper (Theorem 3.4 \cite{Hcar}).
It is shown in Section 2 of \cite{Hcar} that given a product of local blow ups $\pi:X_1\rightarrow X$, there is a natural homeomorphism
$j_{\pi}:\mathcal E_{X_1}\rightarrow \mathcal E_{\pi}$ giving a commutative diagram
$$
\begin{array}{rccl}
\mathcal E_{X_1}&\cong \mathcal E_{\pi}&\subset &\mathcal E_X\\
P_{X_1}\downarrow&&&\downarrow P_X\\
X_1&&\stackrel{\pi}{\rightarrow}&X.
\end{array}
$$

\begin{Definition}\label{DefMon} Suppose that $\phi:Y\rightarrow X$ is a  morphism of  complex or real analytic manifolds, and $p\in Y$. We will say that the map $\phi$ is monomial at $p$ if there exist regular parameters $x_1,\ldots,x_m, x_{m+1},\ldots,x_t$ in $\mathcal O_{X,\phi(p)}^{\rm an}$ and $y_1,\ldots,y_n$ in $\mathcal O_{Y,p}^{\rm an}$ and $c_{ij}\in \NN$ such that
$$
\phi^*(x_i)=\prod_{j=1}^ny_j^{c_{ij}}\mbox{ for }1\le i\le m
$$
with $\mathop{rank}(c_{ij})=m$ and $\phi^*(x_i)=0$ for $m< i\le t$. We will say that $y_1y_2\cdots y_n=0$ is a local toroidal structure $O$ at $p$ and that  $\phi$ is a monomial morphism for the toroidal structure $ O$ at $p$.

We will say that $\phi$ is monomial on $Y$ (or simply that $\phi$ is monomial) if there exists an open cover of $Y$ by open sets $U_k$ which are isomorphic to open subsets of $\CC^n$ (or $\RR^n$) and an open cover of $X$ by open sets $V_k$ which are isomorphic to open subsets of $\CC^t$ (or $\RR^t$) such that $\phi(U_k)\subset V_k$ for all $i$ and there exist
$c_{ij}(k)\in\NN$ such that
$$
\phi^*(x_i)=\prod_{j=1}^ny_j^{c_{ij}(k)}\mbox{ for }1\le i\le m
$$
with $\mathop{rank}(c_{ij})=m$ and $\phi^*(x_i)=0$ for $m< i\le t$, and 
where $x_i$ and $y_j$ are the respective coordinates on $\CC^t$ and $\CC^n$ (or $\RR^t$ and $\RR^n$). 

We will say that $y_1y_2\cdots y_n=0$ is a local toroidal structure $O$ on $U_k$ and that  $\phi|U_k$ is a monomial morphism for the toroidal structure $ O$ on $U_k$.

\end{Definition}

\begin{Definition} Suppose that $\phi:Y\rightarrow X$ is an analytic morphism of connected complex analytic manifolds and $e$ is an \'etoile over $Y$. Define
$$
d_e(\phi)=\min\{r_{e_{Y_1}}^{\mathcal A}(\phi_1)\}
$$
where the minimum is over commutative diagrams of analytic morphisms
\begin{equation}\label{eqQR1}
\begin{array}{rcl}
Y_1&\stackrel{\phi_1}{\rightarrow}&X_1\\
\beta\downarrow&&\downarrow \alpha\\
Y&\stackrel{\phi}{\rightarrow}&X
\end{array}
\end{equation}
such that $Y_1$ and $X_1$ are connected complex analytic manifolds, $\beta\in e$, $\alpha$ and $\beta$ are products of local blowups of nonsingular closed analytic sub varieties
and there exists a nowhere dense closed analytic subspace $F_1$ of $X_1$ such that $X_1\setminus F_1\rightarrow X$ is an open embedding and $\phi_1^{-1}(F_1)$ is nowhere dense in $Y_1$.

We will say that $\phi$ is quasi regular with respect to an \'etoile $e$ on $Y$ if
$$
d_e(\phi)=r_{e_Y}^{\mathcal A}(\phi)=\dim X.
$$
\end{Definition}

\begin{Lemma}\label{LemmaQR} Suppose that $\phi:Y\rightarrow X$ is a morphism of connected complex analytic manifolds and $e$ is an \'etoile over $Y$. Suppose that we have a commutative diagram
$$
\begin{array}{rlc}
Y_2&\stackrel{\phi_2}{\rightarrow}&X_2\\
\alpha_2\downarrow&&\downarrow\beta_2\\
Y_1&\stackrel{\phi_1}{\rightarrow}&X_1\\
\alpha\downarrow&&\downarrow\beta\\
Y&\stackrel{\phi}{\rightarrow}&X
\end{array}
$$
such that $Y_2$, $X_2$, $Y_1$ and  $X_1$ are connected complex analytic manifolds, $\alpha\in e$, $\alpha\alpha_2\in e$ and $\alpha,\alpha_2,\beta,\beta_2$ are products of local blow ups of nonsingular closed analytic sub varieties such that there exists a nowhere dense closed analytic subspace $F_2$ of $X_2$ such that $X_2\setminus F_2\rightarrow X$ is an open embedding and $\phi_2^{-1}(F_2)$ is nowhere dense in $Y_2$.
Then
$$
r_{e_{Y_1}}^{\mathcal A}(\phi_1)\ge r_{e_{Y_2}}^{\mathcal A}(\phi_2).
$$
\end{Lemma}

\begin{proof} Let $\mathcal K_1$ be the kernel of the homomorphism
$$
\phi_1^*:\mathcal O_{X_1,\phi_1(e_{Y_1})}^{\rm an}\rightarrow \mathcal O_{Y_1,e_{Y_1}}^{\rm an}.
$$
The kernel $\mathcal K_1$ is a prime ideal.  There exists an open neighborhood $V$ of $\phi_1(e_{Y_1})$ in $X_1$ such that $\mathcal K_1$ is generated by analytic functions $f_1,\ldots,f_r$ on $V$ and $Z_1=Z(f_1,\ldots,f_r)\subset V$ is analytically irreducible with $\dim_{\phi_1(e_{Y_1})} Z_1=r_{e_{Y_1}}^{\mathcal A}(\phi_1)$. We have $e_{Y_1}\in\phi_1^{-1}(V)$. Let $Z_2$ be the strict transform of $Z_1$ in $\beta_2^{-1}(V)$. The open set $\phi_2^{-1}(\beta_2^{-1}(V))\not\subset \phi_2^{-1}(F_2)$ since $\phi_2^{-1}(F_2)$ is nowhere dense in $Y_2$ and so $\phi_2(\alpha_2^{-1}(\phi_1^{-1}(V)))\not\subset F_2$. But
$$
\phi_2(\phi_2^{-1}(\beta_2^{-1}(V)))=\phi_2(\alpha_2^{-1}(\phi_1^{-1}(V)))\subset \beta_2^{-1}(Z_1)
$$
and so $\beta_2^{-1}(Z_1)\not\subset F_2$ and thus $Z_2\ne \emptyset$, $\phi_2(\alpha_2^{-1}(\phi_1^{-1}(V)))\subset Z_2$
  and  the ideal of the germ of $Z_2$ at $\phi_2(e_{Y_2})$ is contained in the kernel $\mathcal K_2$ of $\phi_2^*:\mathcal O_{X_2,\phi_2(e_{Y_2})}^{\rm an}\rightarrow \mathcal O_{Y_2,e_{Y_2}}^{\rm an}$. Thus 
$$
r_{e_{Y_2}}^{\mathcal A}(\phi_2)\le \dim_{\phi_2(e_{Y_2} )}Z_2=\dim_{\phi_1(e_{Y_1})}Z_1=r_{e_{Y_1}}^{\mathcal A}(\phi_1).
$$
\end{proof}

\begin{Proposition}\label{TheoremA} Suppose that $\phi:Y\rightarrow X$ is a morphism of reduced  complex analytic spaces and $e\in \mathcal E_Y$
is an \'etoile over $Y$. Then there exists a commutative diagram of morphisms
\begin{equation}\label{eq1}
\begin{array}{ccc}
Y_e&\stackrel{\phi_e}{\rightarrow}& X_e\\
\delta\downarrow &&\downarrow \gamma\\
Y&\stackrel{\phi}{\rightarrow}& X
\end{array}
\end{equation}
such that $\delta\in e$, the morphisms $\gamma$ and $\delta$ are  finite products of local blow ups of nonsingular analytic sub varieties, $Y_e$ and $X_e$ are smooth analytic spaces, there exists a closed analytic sub manifold $Z_e$ of $X_e$ such that $\phi_e(Y_e)\subset Z_e$ and the induced analytic map $\phi_e:Y_e\rightarrow Z_e$ is quasi regular with respect to $e$. 
Further, there exists a nowhere dense closed analytic subspace $F_e$ of $X_e$ such that $X_e\setminus F_e\rightarrow X$ is an open embeddding and $\phi_e^{-1}(F_e)$ is nowhere dense in $Y_e$.
\end{Proposition}

\begin{proof} Let 
\begin{equation}\label{eqQR2}
\begin{array}{rcl}
Y_1&\stackrel{\phi_1}{\rightarrow}&X_1\\
\alpha\downarrow&&\downarrow\beta\\
Y&\stackrel{\phi}{\rightarrow}&X
\end{array}
\end{equation}
be a diagram as in (\ref{eqQR1}) such that 
$$
d_e(\phi)=r_{e_{Y_1}}^{\mathcal A}(\phi_1).
$$
Let $\mathcal K$ be the prime ideal which is the kernel of 
$$
\phi_1^*:\mathcal O_{X_1,\phi_1(e_{Y_1})}^{\rm an}\rightarrow \mathcal O_{Y_1,e_{Y_1}}^{\rm an}.
$$
We can replace $X_1$ with an open neighborhood $V$ of $\phi_1(e_{Y_1})$ on which a set of generators of $\mathcal K$ are analytic and determine a locally irreducible closed analytic subset $Z$ of $V$ and replace $Y_1$ with $\phi_1^{-1}(V)$.  After performing an embedded resolution of singularities $X_2\rightarrow X_1$ of $Z$ and a resolution of indeterminacy of the rational map $Y_1\dashrightarrow X_2$, we may assume that $Z$ is nonsingular. Then we have achieved the conclusions of Proposition \ref{TheoremA} by Lemma \ref{LemmaQR}.
\end{proof}

Suppose that $\phi:Y\rightarrow X$ is a regular morphism of nonsingular complex analytic spaces and that $e$ is an \'etoile over $Y$. Then
$e$ naturally induces  an \'etoile $f$ over $X$; we have that $\Omega_f\subset \Omega_e$ and $V_f=V_e\cap \Omega_f$ by Proposition 6.2 \cite{WLM}.

If we do not assume that $\phi:Y\rightarrow X$ is regular, but only that $\phi$ is quasi regular with respect to $e$, then the same construction of an induced \'etoile on $X$ is valid (by Lemma \ref{LemmaQR} and Proposition \ref{TheoremA}).

We in fact have that a quasi regular morphism is regular, as we deduce in Corollary \ref{RegQR}.  This fact can also be deduced from the local flattening theorem of Hironaka, Lejeune and Teissier \cite{HLT} and Hironaka \cite{H3}, as is shown in \cite{WLM}.

\section{Valuations on algebraic function fields}

We begin this section by reviewing some material from Sections 8,9,10 of \cite{RTM} and Chapter VI, Section 10 \cite{ZS2}.

Let $K$ be an algebraic function field over a field $\overline k$, and let $\nu$ be a valuation of $K$ which is trivial on $\overline k$. Let $V_{\nu}$ be the valuation ring of $\nu$ and $\Gamma_{\nu}$ be the value group of $\nu$. Let 
$$
0=\mathfrak p_0\subset \cdots \subset \mathfrak p_d\subset V_{\nu}
$$
be the chain of prime ideals in $V_{\nu}$.  Let $U_i=\{\nu(a)\mid a\in \mathfrak p_i\setminus \{0\}\}$. Let $\Gamma_i$ be the complement of $U_i$ and $-U_i$ in $\Gamma_{\nu}$. The chain of isolated subgroups in $\Gamma_{\nu}$ is
$$
0=\Gamma_d\subset \cdots\subset \Gamma_0=\Gamma_{\nu}.
$$
The valuations composite with $\nu$ have the valuation rings $V_{\mathfrak p_i}$ with value groups $\Gamma_{\nu}/\Gamma_i$. Let $\nu_i$ be the induced valuation ($\nu_i(f)$ is the class of $\nu(f)$ in $\Gamma_{\nu}/\Gamma_i$ for $f\in K\setminus \{0\}$). 
The valuation $\nu$ is called zero dimensional if the residue field $V_{\nu}/\mathfrak p_d$ is an algebraic extension of $\overline k$. In this section we prove the following lemma. In the case when $\nu$ has rank 1 (so there is an order preserving embedding of $\Gamma_{\nu}$ in $\RR$),  Lemma \ref{Lemma25} is proven in Section 9 of \cite{LU}. We extend this proof to the case when $\nu$ has  arbitrary rank $d$.
Related constructions of Perron transforms along a valuation of rank greater than $1$ are given by ElHitti in \cite{ElH}.

\begin{Lemma}\label{Lemma25} Suppose that $\overline k$ is a field and $\nu$ is a valuation of the quotient field of the polynomial ring $\overline k[x_1,\ldots,x_{s+1}]$ such that $\nu(x_i)>0$ for  $1\le i\le s$, $\nu(x_{s+1})\ge 0$, $\nu(x_1),\ldots,\nu(x_s)$ are rationally independent and $\nu(x_{s+1})$  is rationally dependent on $\nu(x_1),\ldots,\nu(x_s)$. Then there exists a composition of monoidal transforms (a sequence of blow ups of nonsingular subvarieties) of the form
$$
x_i=\left(\prod_{j=1}^s\overline x_j^{a_{ij}}\right)\overline x_{s+1}^{a_{i,s+1}}\mbox{ for }1\le i\le s\mbox{ and}
$$
$$
x_{s+1}= \left(\prod_{j=1}^s\overline x_j^{a_{s+1,j}}\right)\overline x_{s+1}^{a_{s+1,s+1}}
$$
such that $\nu(\overline x_i)>0$ for $1\le i\le s$ and $\nu(\overline x_{s+1})=0$. 

If $\nu$ is zero dimensional and $\overline k$ is algebraically closed, then there exists $0\ne \alpha\in \overline k$ such that $\nu(\overline x_{s+1}-\alpha)>0$.
\end{Lemma}

\begin{proof} The proof is by decreasing   induction on the largest $k\le d$ such that there exist $x_{i_1},\ldots,x_{i_a}$ (with $1\le i_1\le \cdots\le i_a\le s$) such that $\nu(x_{s+1})$ is rationally dependent on $\nu(x_{i_1}),\ldots,\nu(x_{i_a})$ and $\nu(x_{i_1}),\ldots,\nu(x_{i_a})\in \Gamma_k$. If $k=d$ then $\nu(x_{s+1})=0$, and the lemma is trivially satisfied, with $(a_{ij})$ being the identity matrix. 

Suppose that this condition is satisfied for $k$, and the lemma is true for $k+1$. Without loss of generality, since with this condition we can ignore the variables such that $\nu(x_i)\not\in \Gamma_k$, we may assume that $\nu(x_1),\ldots,\nu(x_s)\in \Gamma_k$. After reindexing the $x_i$, there exists $r$ such that $1\le r\le s$ and $\nu_{k+1}(x_1),\ldots,\nu_{k+1}(x_r)$ is a basis of the span as a rational vector space of $\nu_{k+1}(x_1),\ldots,\nu_{k+1}(x_s)$ in $(\Gamma_k/\Gamma_{k+1})\otimes \QQ$. 

Suppose that there exists $t$ with $r<t\le s$ and $\nu_{k+1}(x_t)\ne 0$. After possibly reindexing $x_{r+1},\ldots,x_s$ we may assume that
$\nu_{k+1}(x_{r+1})\ne 0$. We necessarily have that $\nu_{k+1}(x_{r+1})>0$ since $\nu(x_{r+1})>0$. Since $\Gamma_k/\Gamma_{k+1}$ is a rank 1 ordered group, we can apply the algorithm of Section 2 on pages 861 - 863 of \cite{LU} and Section 9 on page 871 of \cite{LU} to construct a sequence of monoidal transforms along $\nu$,
$$
x_i=\left(\prod_{j=1}^rx_j(1)^{a_{ij}(1)}\right)x_{r+1}(1)^{a_{i,r+1}(1)}\mbox{ for }1\le i\le r\mbox{ and}
$$
$$
x_{r+1}=\left(\prod_{j=1}^rx_j(1)^{a_{r+1,j}(1)}\right)x_{r+1}(1)^{a_{r+1,r+1}(1)}
$$
and $x_i=x_i(1)$ for $r+1\le i\le s$ such that 
$\nu_{k+1}(x_i(1))>0$ for $1\le i\le r+1$ and 
$$
\nu_{k+1}(x_{r+1}(1))=\lambda_1\nu_{k+1}(x_1(1))+\cdots+\lambda_r\nu_{k+1}(x_r(1))
$$
for some $\lambda_1,\ldots,\lambda_r\in \NN$ (by equation (11') on page 863 \cite{LU}). We necessarily have that some $\lambda_i>0$, so we may assume that $\lambda_1>0$. Then perform the sequence of monoidal transforms along $\nu$
$$
x_{r+1}(1)=x_1(2)^{\lambda_1-1}x_2(2)^{\lambda_2}\cdots x_r(2)^{\lambda_r}x_{r+1}(2)
$$
and $x_i(1)=x_i(2)$ for $i\ne r+1$. Then $\nu_{k+1}(x_i(2))>0$ for all $i$ with $1\le i\le r+1$ and $\nu_{k+1}(x_{r+1}(2))=\nu_{k+1}(x_1(2))$.
We necessarily have that 
$$
\nu\left(\frac{x_1(2)}{x_{r+1}(2)}\right)>0\mbox{ or }\nu\left(\frac{x_{r+1}(2)}{x_1(2)}\right)>0
$$
as $\nu(x_1(2)),\ldots,\nu(x_s(2))$ are rationally independent. In the first case, perform the monoidal transform along $\nu$
$$
x_1(2)=x_1(3)x_{r+1}(3),\,\, x_{r+1}(2)=x_1(3)\mbox{ and }x_i(2)=x_i(3)\mbox{ for }i\ne 1\mbox{ or }r+1.
$$
Otherwise, perform the monoidal transform along $\nu$
$$
x_1(2)=x_1(3),\,\, x_{r+1}(2)=x_1(3)x_{r+1}(3)\mbox{ and }x_i(2)=x_i(3)\mbox{ for }i\ne 1\mbox{ or }r+1.
$$
We then have that $\nu(x_i(3))>0$ for $1\le i\le s+1$, $\nu_{k+1}(x_1(3)),\ldots,\nu_{k+1}(x_r(3))$ is a rational basis of the span of 
$\nu_{k+1}(x_1(3)),\ldots,\nu_{k+1}(x_s(3))$ as a rational vector space in $(\Gamma_k/\Gamma_{k+1})\otimes\QQ$,
$\nu(x_1(3)),\ldots,\nu(x_s(3))$ are rationally independent, and $\nu(x_{s+1}(3))$ is rationally dependent on $\nu(x_1(3)),\ldots,\nu(x_s(3))$. We further have that $\nu_{k+1}(x_{r+1}(3))=0$. We repeat this algorithm, reducing to the case that $\nu_{k+1}(x_i)=0$ if $r+1\le i\le s$. 

Suppose that $\nu_{k+1}(x_{s+1})>0$ (and $\nu_{k+1}(x_i)=0$ for $r+1\le i\le s$). Then we apply the algorithm that we used above to construct a monoidal transform along $\nu$
\begin{equation}\label{eq101}
\begin{array}{lll}
x_i&=&\left(\prod_{j=1}^rx_j(1)^{a_{ij}(1)}\right)x_{s+1}(1)^{a_{i,r+1}(1)}\mbox{ for }1\le i\le r\mbox{ and }\\
x_{s+1}&=&\left(\prod_{j=1}^rx_j(1)^{a_{r+1,j}(1)}\right)x_{s+1}(1)^{a_{r+1,r+1}(1)}
\end{array}
\end{equation}
to achieve $\nu_{k+1}(x_i(1))>0$ for $1\le i\le r$, $\nu_{k+1}(x_{s+1}(1))=0$ and $\nu(x_{s+1}(1))\ge 0$. Since $\nu_{k+1}(x_1),\ldots,\nu_{k+1}(x_r)$ are rationally independent, (\ref{eq101}) implies that $\nu_{k+1}(x_1(1)),\ldots,\nu_{k+1}(x_r(1))$ are rationally independent. Since $\nu_{k+1}(x_i)=0$ for $r<i\le s$ and $\nu(x_{r+1}),\ldots,\nu(x_s)\in \Gamma_{k+1}$ are rationally independent we have that
 $$
\nu(x_1(1)),\ldots,\nu(x_r(1)),\nu(x_{r+1}),\ldots,\nu(x_s)
$$
 are rationally independent. Since
$$
\nu(x_1(1)),\ldots,\nu(x_r(1)),\nu(x_{r+1}),\ldots,\nu(x_s),\nu(x_{s+1}(1))
$$
and $\nu(x_1),\ldots,\nu(x_s)$ span the same rational subspace $V$ of $\Gamma_{\nu}\otimes \QQ$, which has dimension $s$, we have that 
$$
\nu(x_1(1)),\ldots,\nu(x_r(1)),\nu(x_{r+1}),\ldots,\nu(x_s)
$$
is a rational basis of $V$, so $\nu(x_{s+1}(1))$ is a rational linear combination of
$$
\nu(x_1(1)),\ldots,\nu(x_r(1)),\nu(x_{r+1}),\ldots,\nu(x_s).
$$
Since $\nu_{k+1}(x_{s+1}(1))=0$ and $\nu(x_{k+1}(1)),\ldots,\nu_{k+1}(x_r(1))$ are rationally independent, we have that 
$\nu(x_{s+1}(1))$ is a rational linear combination of $\nu(x_{r+1}),\ldots,\nu(x_s)\in \Gamma_{k+1}$. We thus attain the conclusions of the lemma by decreasing induction on $k$.

Finally, if $\nu$ is zero dimensional and $\overline k$ is algebraically closed, then the class $\alpha$ of $\overline x_{s+1}$ in the residue field $\overline k$ of $V_{\nu}$ is nonzero. Then necessarily $\nu(\overline x_{s+1}-\alpha)>0$.

\end{proof}

\section{Generalized Monoidal Transforms}\label{SecGMT}
Suppose that $X$ is a nonsingular complex analytic space and $e$ is an \'etoile over $X$. Let $\nu_e$ be a valuation of $\Omega_e$ whose valuation ring is $V_e$ (Section \ref{Pre}).
Suppose that $\tilde X\rightarrow X\in e$ and
$x_1,\ldots,x_n$ is a regular system of parameters in $\mathcal O_{\tilde X,e_{\tilde X}}^{\rm an}$. 
Suppose that $\overline X\rightarrow \tilde X$ is such that $\overline X\rightarrow \tilde X\rightarrow X\in e$. The germ of the local homomorphism $\mathcal O_{\tilde X,e_{\tilde X}}^{\rm an}\rightarrow \mathcal O_{\overline X,e_{\overline X}}^{\rm an}$ is a 
Generalized Monoidal Transform (GMT) along the \'etoile $e$ if $\mathcal O_{\overline X,e_{\overline X}}^{\rm an}$ has regular parameters
 $\overline x_1,\ldots,\overline x_n$ such that there exists an $n\times n$ matrix $A=(a_{ij})$ with $a_{ij}\in \NN$ and $\mbox{Det}(A)=\pm 1$ such that
\begin{equation}\label{eq2}
x_i=\prod_{j=1}^n(\overline x_j+\alpha_j)^{a_{ij}}
\end{equation}
for $1\le j\le n$ and  $\alpha_j\in \CC$ (at least one of which must be zero since $\mathcal O_{\tilde X,e_{\tilde X}}^{\rm an}\rightarrow \mathcal O_{\overline X,e_{\overline X}}^{\rm an}$ is a local homomorphism). We will say that the GMT is in the variables $x_{i_1},\ldots,x_{i_m}$ if the GMT has the special form
$$
x_i=\prod_{j\in S}(\overline x_j+\alpha_j)^{a_{ij}}
$$
for $i\in S$ and 
$$
x_i=\overline x_i
$$
for $i\not\in S$
where $S=\{i_1,\ldots,i_m\}$. We will say that the GMT is monomial if all $\alpha_j$ are zero.
We observe that a GMT is a regular morphism.

It will be assumed through out this paper that all GMT are along a fixed \'etoile $e$.

\begin{Definition}\label{Def53}
The variables $x_1,\ldots,x_s$ are said to be dependent if there exists a GMT (\ref{eq2}) in $x_1,\ldots,x_s$
which is not  monomial.
\end{Definition}

The variables $x_1,\ldots,x_s$ are said to be independent if they are not dependent.

\begin{Lemma}\label{Lemma1} Suppose that $x_1,\ldots,x_s$ are independent and (\ref{eq2}) is a GMT in $x_1,\ldots,x_s$. Then $\overline x_1,\ldots, \overline x_s$ are independent.
\end{Lemma}

\begin{proof} This follows since a composition of a GMT in $x_1,\ldots,x_s$ and in $\overline x_1,\ldots,\overline x_s$ is a GMT in $x_1,\ldots,x_s$.
\end{proof}

\begin{Definition} A GMT is a simple GMT (SGMT) if it can be factored by a sequence of blow ups of nonsingular subvarieties.
\end{Definition}

\begin{Lemma}\label{Lemma50}  The variables $x_1,\ldots,x_s$ are independent if and only if every SGMT in $x_1,\ldots,x_s$ is monomial.
\end{Lemma}

\begin{proof} Suppose that every SGMT in $x_1,\ldots,x_s$ is monomial and (\ref{eq2}) is a GMT in $x_1,\ldots,x_s$. We must show that all $\alpha_i=0$. Let $\nu$ be the valuation of the quotient field $K$ of $\CC[x_1,\ldots,x_s]$ which gives the restriction of $\nu_e$ to $K$. Let $\pi:Z\rightarrow \AA^s$ be a projective morphism of nonsingular toric varieties such that $\overline x_1,\ldots,\overline x_s$ are regular parameters in $\mathcal O_{Z,p}$, where $p$ is the center of $\nu$ on $Z$. Let $J$ be a (monomial) ideal in $\CC[x_1,\ldots,x_s]$ whose blow up in $\AA^s$ is $Z$.  By principalization of ideals (a particularly simple algorithm which is adequate for our purposes is given in \cite{Go}), there exists a projective morphism of nonsingular toric varieties $\Lambda:Z_1\rightarrow \AA^s$ which is a product of blow ups of nonsingular varieties such that $J\mathcal O_{Z_1}$ is locally principal, and so $\Lambda$ factors through $\pi$. Let $I$ be a monomial ideal such that $Z_1$ is the blow up of $I$.

Let $X_1$ be obtained by blowing up $I$ in a neighborhood of $e_{\tilde X}$ in $\tilde X$.  Then $\mathcal O_{\tilde X,e_{\tilde X}}^{\rm an}\rightarrow \mathcal O_{X_1,e_{X_1}}^{\rm an}$ is a SGMT (since $Z_1\rightarrow \AA^s$ is a morphism of toric varieties which is a product of blow ups of nonsingular varieties). Thus  $\mathcal O_{Z_1,p_1}$ has regular parameters $\tilde x_1,\ldots,\tilde x_s$ (where $p_1$ is the center of $\nu$ on $Z_1$) and  $\tilde x_1,\ldots,\tilde x_s,x_{s+1},\ldots,x_n$ are regular parameters in $\mathcal O_{X_1,e_{X_1}}^{\rm an}$
such that $x_i=\prod_{j=1}^s\tilde x_j^{b_{ij}}$ are monomials for $1\le i\le s$. Since $\Lambda$ factors through $\pi$, and so there is a factorization
$$
\mathcal O_{\tilde X,e_{\tilde X}}\rightarrow \mathcal O_{\overline X,\mathcal O_{\overline X}}\rightarrow \mathcal O_{X_1,e_{X_1}}
$$
we must also have that the given GMT (\ref{eq2}) is monomial.
\end{proof}

\begin{Lemma}\label{Lemma21}
Suppose that $x_1,\ldots,x_s$ are independent and 
$$
M_1=x_1^{d_1(1)}\cdots x_s^{d_s(1)},\,\, M_2=x_1^{d_1(2)}\cdots x_s^{d_s(2)}
$$
are monomials with $d_i(j)\in \NN$. Then there exists a (monomial) SGMT in $x_1,\ldots,x_s$ such that the ideal generated by $M_1$ and $M_2$ is principal in $\mathcal O_{X_1,e_{X_1}}^{\rm an}$.
\end{Lemma}

\begin{proof} Let $\nu$ be the valuation of the quotient field $K$ of $\CC[x_1,\ldots,x_s]$ which gives the restriction of $\nu_e$ to $K$.
Since $x_1,\ldots,x_s$ are independent, $\nu(x_1),\ldots,\nu(x_s)$ are rationally independent by Lemma \ref{Lemma25}. Let $I$ be the ideal generated by $M_1$ and $M_2$ in $\CC[x_1,\ldots,x_s]$ There exists a birational morphism of nonsingular toric varieties which is a product of blow ups of nonsingular subvarieties $\pi:Z\rightarrow \AA^s$ such that $I\mathcal O_Z$ is an invertible ideal sheaf. Let $p_1$ be the center of $\nu$ on $Z$. Since $\pi$ is toric and $\nu(x_1),\ldots,\nu(x_s)$ are rationally independent, there exist regular parameters $\overline x_1,\ldots,\overline x_s$ in $\mathcal O_{Z,p_1}$ such that 
\begin{equation}\label{eq50}
x_i=\prod_{j=1}^s\overline x_j^{a_{ij}}
\end{equation}
for $1\le i\le s$ are monomials in $\overline x_1,\ldots,\overline x_s$. Let $J$ be the monomial ideal in $\CC[x_1,\ldots,x_s]$ whose blow up is $Z$.  Let $X_1$ be the blow up of $J$ in a neighborhood of $e_{\tilde X}$ in $\tilde X$. Then $\overline x_1,\ldots,\overline x_s,x_{s+1},\ldots,x_m$ are regular parameters in $\mathcal O_{X_1,e_{X_1}}^{\rm an}$ and $I\mathcal O_{X_1,e_{X_1}}^{\rm an}$ is a principal ideal.
\end{proof}

\begin{Lemma}\label{Lemma10} Suppose that $x_1,\ldots,x_s\in \mathcal O_{\tilde X,e_{\tilde X}}^{\rm an}$ are independent, $\gamma\in \mathcal O_{\tilde X,e_{\tilde X}}^{\rm an}$ is a unit and $d_1,\ldots,d_s\in \QQ$. Then $\tilde x_1=\gamma^{d_1}x_1,\ldots,\tilde x_s=\gamma^{d_s}x_s$ are independent.
\end{Lemma}

\begin{proof} Suppose that $\tilde x_1,\ldots,\tilde x_s$ are not independent. Then there exists $\overline X\rightarrow \tilde X$ giving a GMT 
$\tilde x_i=\prod_{j=1}^s(\hat x_j+\hat\alpha_j)^{a_{ij}}$ for $1\le j\le s$ with some $\hat\alpha_j\ne 0$. After reindexing the $\tilde x_i$, we may assume that $\hat\alpha_j=0$ for $1\le j\le a<s$ and $\hat\alpha_j\ne 0$ for $a< j\le s$.
Define $c_1,\ldots,c_s\in\QQ$ by
$$
\left(\begin{array}{c}
c_1\\ \vdots\\c_s\end{array}\right)=A^{-1}\left(\begin{array}{c}
d_1\\ \vdots\\ d_s\end{array}\right)
$$
where $A=(a_{ij})$. Then $\prod_{j=1}^s(\gamma^{c_{ij}})^{a_{ij}}=\gamma^{d_i}$ for $1\le i\le s$. We have
\begin{equation}\label{eq100}
\gamma^{c_j}\equiv \gamma(0)^{c_j}\mbox{ mod }(\hat x_1,\ldots,\hat x_a)\mathcal O_{\overline x,e_{\overline x}}^{\rm an}\mbox{ for all }j.
\end{equation}
Set $\overline x_j=\gamma^{c_j}\hat x_j$ for $1\le j\le a$, and define $\alpha_j=\gamma(0)^{c_j}\hat\alpha_j$,
$\overline x_j=\gamma^{c_j}(\hat x_j+\hat\alpha_j)-\alpha_j$ for $a\le j\le s$. Then $\overline x_1,\ldots,\overline x_s$ are regular parameters in $\mathcal O_{\overline X,e_{\overline X}}^{\rm an}$ by (\ref{eq100}). Thus we have a GMT
$$
x_i=\prod_{j=1}^s(\overline x_j+\alpha_j)^{a_{ij}}\mbox{ for }1\le j\le s
$$
 in $x_1,\ldots,x_s$, contradicting the independence of $x_1,\ldots,x_s$ since some $\alpha_j\ne 0$.
\end{proof}

\begin{Lemma}\label{Lemma6} Suppose that $x_1,\ldots,x_s$ are independent and $x_1,\ldots,x_s,x_{s+1}$ are dependent. Suppose that
(\ref{eq2}) is A GMT in $x_1,\ldots,x_{s+1}$ such that some $\alpha_j\ne 0$. Then there are  $x_1(1),\ldots,x_{s+1}(1)$ in $\mathcal O_{\overline X,e_{\overline X}}^{\rm an}$ such that $x_1(1),\ldots,x_{s+1}(1),x_{s+2},\ldots,x_n$ are a regular system of parameters in $\mathcal O_{\overline X,e_{\overline X}}^{\rm an}$ and there is an expression
$$
x_i=\prod_{j=1}^sx_j(1)^{b_{ij}}\mbox{ for }1\le i\le s
$$
and 
$$
x_{s+1}=\prod_{j=1}^sx_j(1)^{b_j}(x_{s+1}(1)+\alpha)
$$
where $0\ne\alpha \in \CC$, $b_{ij},b_j\in \NN$ and the $s\times s$ matrix $(b_{ij})$ has nonzero determinant. 
Further, the variables $x_1(1),\ldots, x_s(1)$ are independent.
\end{Lemma}

\begin{proof} 
Let $R= \CC[x_1,\ldots,x_{s+1}]_{(x_1,\ldots,x_{s+1})}$ and  $K$ be the quotient field of $R$.  Let   (\ref{eq2}) be a  GMT in  $x_1,\ldots,x_s,x_{s+1}$ which is not monomial and $R_1=\CC[\overline x_1,\ldots, \overline x_{s+1}]_{(\overline x_1,\ldots, \overline x_{s+1})}$. We have a commutative diagram of injective local homomorphisms
$$
\begin{array}{ccc}
R&\rightarrow & \mathcal O_{\tilde X,e_{\tilde X}}^{\rm an}\\
\downarrow&&\downarrow\\
R_1&\rightarrow &\mathcal O_{\overline X,e_{\overline X}}^{\rm an}.
\end{array}
$$
The field  $K$ is also the quotient field of $R_1$ and $R\rightarrow R_1$ is birational. Let $\nu$ be  the restriction of $\nu_e$ to $K$. We have that $\nu$ dominates $R$ and $\nu$ dominates $R_1$. Since all GMT in $x_1,\ldots,x_s$ are monomial, we must have that   $\nu(x_1),\ldots,\nu(x_s)$ are  rationally independent by Lemma \ref{Lemma25}.
We have that 
$$
\nu(x_i)=\sum_{j=1}^{s+1} a_{ij}\nu(\overline x_j+\alpha_j)\mbox{ for }1\le i\le s+1.
$$
Thus after possibly interchanging the  variables $\overline x_1,\ldots, \overline x_{s+1}$, we have that $\alpha_1=\ldots=\alpha_s=0$.
Further, since our GMT (\ref{eq2}) is not monomial, we must have that $\alpha_{s+1}\ne 0$. Thus the $s\times s$ matrix consisting of the first $s$ rows and columns of $A=(a_{ij})$ has rank $s$ and $\nu(\overline x_1),\ldots,\nu(\overline x_s)$ are rationally independent. There exists $\lambda_i\in \QQ$ such that after replacing $\overline x_i$ with
$x_i(1):= (\overline x_{s+1}+\alpha_{s+1})^{\lambda_i}\overline x_i$ for $1\le i\le s$, we have that 
$x_i=\prod_{j=1}^sx_j(1)^{a_{ij}}$ for $1\le i\le s$ and $x_{s+1}=\prod_{j=1}^sx_j(1)^{a_{s+1,j}}(\overline x_{s+1}+\alpha_{s+1})^{\lambda}$
where $\lambda\in\QQ$ is non zero since $\mbox{Det}(A)\ne 0$. Setting $x_{s+1}(1):= (\overline x_{s+1}+\alpha_{s+1})^{\lambda}-\alpha_{s+1}^{\lambda}$ and $\alpha = \alpha_{s+1}^{\lambda}$, we obtain the expression of the GMT asserted in the lemma. 

The values $\nu_e(\overline x_1),\ldots,\nu_e(\overline x_s)$ are rationally independent, and $\nu_e(\overline x_{s+1}+\alpha_{s+1})=0$, so 
$\nu_e(x_1(1)),\ldots,\nu_e(x_s(1))$ are rationally independent. Thus $x_1(1),\ldots,x_s(1)$ are independent.

\end{proof}

The following lemma giving a Tschirnhaus transformation will be useful.

\begin{Lemma}\label{Lemma9} Suppose that $F\in \CC\{\{ x_1,\ldots,x_n\}\}$ and $\mbox{ord }F(0,\ldots,0,x_n)=t\ge 1$. Then there exists $\Phi\in \CC\{\{ x_1,\ldots,x_{n-1}\}\}$ such that setting $\overline x_n=x_n-\Phi$, we have that
$$
F=\tau_0\overline x_n^t+\tau_2\overline x_n^{t-2}+\cdots+\tau_t
$$
where $\tau_0\in\CC\{\{ x_1,\ldots,\overline x_n\}\}$ is a unit and $\tau_i\in \CC\{\{ x_1,\ldots,x_{n-1}\}\}$ for $2\le i\le t$.
\end{Lemma}

\begin{proof} By the implicit function theorem (cf. Section C.2.4 \cite{L}),
$$
\frac{\partial^{t-1}F}{\partial x_n^{t-1}}=u(x_n-\Phi)
$$
where $u\in \CC\{\{ x_1,\ldots,x_n\}\}$ is a unit series and $\Phi\in \CC\{\{ x_1,\ldots,x_{n-1}\}\}$. Let $\overline x_n=x_n-\Phi$. Let
$G(x_1,\ldots,x_{n-1},\overline x_n)=F(x_1,\ldots,x_n)$. We expand
$$
\begin{array}{lll}
G&=&G(x_1,\ldots,x_{n-1},0)+\frac{\partial G}{\partial \overline x_n}(x_1,\ldots,x_{n-1},0)\overline x_n+\cdots +\frac{1}{(t-1)!}\frac{\partial^{t-1}G}{\partial \overline x_n^{t-1}}(x_1,\ldots,x_{n-1},0)\overline x_n^{t-1}\\
&&+\frac{1}{t!}\frac{\partial^tG}{\partial \overline x_n^t}(x_1,\ldots,x_{n-1},0)\overline x_n^t+\cdots
\end{array}
$$
We have
$$
\frac{\partial^{t-1}G}{\partial \overline x_n^{t-1}}(x_1,\ldots,x_{n-1},0)=\frac{\partial^{t-1}F}{\partial x_n^{t-1}}(x_1,\ldots,x_{n-1},\Phi)=0
$$
and
$$
\frac{\partial^tG}{\partial\overline x_n^t}(x_1,\ldots,x_{n-1},0)=\frac{\partial^{t}F}{\partial x_n^{t}}(x_1,\ldots,x_{n-1},\Phi)
$$
is a unit in $\CC\{\{ x_1,\ldots,x_n\}\}$, giving (by (\ref{eq42})) the conclusions of the lemma.
\end{proof}

\section{Transformations}\label{SecTrans}
Suppose that $\phi:Y\rightarrow X$ is an analytic morphism of complex analytic manifolds and $e$ is an \'etoile over $Y$ such that $\phi$ is quasi regular with respect to $e$ (Section \ref{Pre}).  We will also denote the induced \'etoile on $X$ (Section \ref{Pre}) by $e$.  

Suppose that $\tilde Y\rightarrow Y\in e$ and $\tilde X\rightarrow X\in e$ give a morphism $\tilde\phi:\tilde Y\rightarrow \tilde X$. Then
$$
\tilde\phi^*:\mathcal O_{\tilde X,e_{\tilde X}}^{\rm an}\rightarrow \mathcal O_{\tilde Y,e_{\tilde Y}}^{\rm an}
$$
 is injective, so we may regard $\mathcal O_{\tilde X,e_{\tilde X}}^{\rm an}$ as a subring of $\mathcal O_{\tilde Y,e_{\tilde Y}}^{\rm an}$.
Assume that there exist regular parameters $x_1,\ldots,x_m$ in $\mathcal O_{\tilde X,e_{\tilde X}}^{\rm an}$ and $y_1,\ldots,y_n$ in 
$\mathcal O_{\tilde Y,e_{\tilde Y}}^{\rm an}$ such that $y_1,\ldots,y_s$ are independent but $y_1,\ldots, y_s,y_i$ are dependent for all $i$ with $s+1\le i\le n$, $x_1,\ldots,x_r$ are independent, and identifying $x_i$ with $\tilde\phi^*(x_i)$, there is an expression for some $l$
\begin{equation}\label{eq5}
\begin{array}{lll}
x_1&=&y_1^{c_{11}}\cdots y_s^{c_{1s}}\\
&&\vdots\\
x_r&=& y_1^{c_{r1}}\cdots y_s^{c_{rs}}\\
x_{r+1}&=& y_{s+1}\\
&&\vdots\\
x_{r+l}&=& y_{s+l}.
\end{array}
\end{equation}

We necessarily have that $C=(c_{ij})$ has rank $r$ (by Lemma \ref{Lemma25}) with our assumptions, and so by the rank theorem (page 134 \cite{L}) and the inequality (\ref{Gineq}) there is an induced inclusion 
$$
\CC[[x_1,\ldots,x_{r+l}]]\rightarrow \CC[[y_1,\ldots,y_n]].
$$

Assume that $E_Y$ is a SNC divisor on $Y$ supported on $Z(y_1y_2\cdots y_s)$ (in a neighborhood of $e_Y$) in $Y$.

\begin{Definition}
We will say that the variables $(x,y)=(x_1,\ldots,x_m; y_1,\ldots,y_n)$ are prepared of type $(s,r,l)$ if all of the above conditions hold.
\end{Definition}

We will say that $(s_1,r_1,l_1)\ge (s,r,l)$ if $s_1\ge s$, $r_1\ge r$ and $r_1+l_1\ge r+l$, and that $(s_1,r_1,l_1)>(s,r,l)$ if $(s_1,r_1,l_1)\ge (s,r,l)$ and $s_1>s$ or $r_1>r$ or $r_1+l_1>r+l$.

We will perform  transformations of the types 1) - 10) below, which preserve the form (\ref{eq5}) (and the quasi regularity of the morphism of germs), giving an expression
\begin{equation}\label{eq7}
\begin{array}{lll}
x_1(1)&=& y_1(1)^{c_{11}(1)}\cdots y_s(1)^{c_{1s}(1)}\\
&&\vdots\\
x_r(1)&=& y_1(1)^{c_{r1}(1)}\cdots y_s(1)^{c_{rs}(1)}\\
x_{r+1}(1)&=& y_{s+1}(1)\\
&&\vdots\\
x_{r+l}(1)&=& y_{s+l}(1)
\end{array}
\end{equation}
where $x_1(1),\ldots,x_m(1)$ and $y_1(1),\ldots,x_n(1)$ are respective regular parameters in $\mathcal O_{\overline X,e_{\overline X}}^{\rm an}$ and
$\mathcal O_{\overline Y,e_{\overline Y}}^{\rm an}$ in the induced commutative diagram  of quasi regular analytic morphisms
$$
\begin{array}{ccc}
\overline Y&\stackrel{\overline \phi}{\rightarrow}&\overline X\\
\downarrow&&\downarrow\\
\tilde Y&\stackrel{\tilde \phi}{\rightarrow}&\tilde X.
\end{array}
$$
where $\overline Y\rightarrow \tilde Y\rightarrow Y\in e$ and $\overline X\rightarrow \tilde X\rightarrow X\in e$.

 Further, we will have that $x_1(1),\ldots,x_r(1)$ are independent and $y_1(1),\ldots,y_s(1)$ are independent. So we either continue to have that $y_1(1),\ldots,y_s(1), y_{t}(1)$ are dependent for all $s+1\le t\le n$ or 
after rewriting (\ref{eq5}), we have an increase in $s$, without decreasing $r$ or $r+l$. 
In summary, we will have that the variables $(x(1),y(1))$ are prepared of type $(s_1,r_1,l_1)$ with $(s_1,r_1,l_1)\ge (s,r,l)$.

Let $E_{\overline Y}$ be the pullback of $E_Y$ on $\overline Y$. Then 
\begin{equation}\label{eq*}
\begin{array}{l}
\mbox{$E_{\overline Y}$ is supported on $Z(y_1(1)y_2(1)\cdots y_s(1))\subset \overline Y$}\\
\mbox{and $\overline Y\setminus Z(y_1(1)y_2(1)\cdots y_s(1))\rightarrow Y$ is an open embedding.}
\end{array}
\end{equation}

\begin{Lemma}\label{Lemma22} Suppose that $(x,y)$ are prepared of type $(s,r,l)$ and
$$
x_i=\prod_{j=1}^rx_j(1)^{a_{ij}}\mbox{ for }1\le i\le r
$$
is a GMT in $x_1,\ldots,x_r$. Then there exists a SGMT
$$
y_i=\prod_{j=1}^s y_j(1)^{b_{ij}}\mbox{ for }1\le i\le s
$$
such that the variables $(x(1),y(1))$ are prepared of type $(s_1,r_1,l_1)$ with $(s_1,r_1,l_1)\ge (s,r,l)$.
\end{Lemma}

\begin{proof} Let $\nu$ be the restriction of  $\nu_e$ to the quotient field $K$ of $\CC[y_1,\ldots,y_s]$, which contains $\CC[x_1,\ldots,x_r]$. The values $\nu(y_1),\ldots,\nu(y_s)$ are rationally independent and $\nu(x_1),\ldots,\nu(x_r)$ are rationally independent by Lemma \ref{Lemma25}. The inclusion $\CC[x_1,\ldots,x_r]\rightarrow \CC[y_1,\ldots,y_s]$ induces a dominant morphism $\AA^s\rightarrow \AA^r$ of nonsingular toric varieties. Let $\pi:Z\rightarrow \AA^r$ be a projective morphism of nonsingular toric varieties such that $x_1(1),\ldots,x_r(1)$ are regular parameters in $\mathcal O_{Z,p}$ where $p$ is the center of $\nu$ on $Z$. Let $J$ be a monomial ideal in $\CC[x_1,\ldots,x_r]$ whose blow up is $Z$. By principalization of ideals, there exists a projective morphism of toric varieties $\Lambda:W\rightarrow \AA^s$ which is a product of blow ups of nonsingular subvarieties, such that $J\mathcal O_W$ is locally principal, so that the rational map $W\dashrightarrow Z$ is a morphism. Let $q_1$ be the center of $\nu$ on $W$. Since $\nu(y_1),\ldots,\nu(y_s)$ are rationally independent and $\Lambda$ is toric, there exist 
regular parameters $\overline y_1,\ldots,\overline y_s$ in $\mathcal O_{W,q_1}$ and $b_{ij}\in\NN$ with ${\rm det}(b_{ij})=\pm 1$ such that 
$$
y_i=\prod_{j=1}^s\overline y_i^{b_{ij}}\mbox{ for }1\le i\le s.
$$
 $W$ is the blow up of a (monomial) ideal $H$ in $\CC[y_1,\ldots,y_s]$. Let $Y_1\rightarrow \tilde Y$ be the blow up of $H$ in a neighborhood of $e_{\tilde Y}$. Let $e_{Y_1}$ be the center of $e$ on $Y_1$. Then $\overline y_1,\ldots,\overline y_s,y_{s+1},\ldots,y_n$ are regular parameters in $\mathcal O_{Y_1,e_{Y_1}}^{\rm an}$, giving the conclusions of the lemma.
 \end{proof}

\begin{Lemma}\label{Lemma23} Suppose that $(x,y)$ are prepared of type $(s,r,l)$, $1\le \overline m\le l$ and
$$
x_i=\prod_{j=1}^rx_j(1)^{a_{ij}}\mbox{ for }1\le i\le r
$$
and
$$
x_{r+\overline m}=\prod_{j=1}^rx_j(1)^{a_{r+\overline m,j}}(x_{r+\overline m}(1)+\alpha)
$$
with $0\ne \alpha\in \CC$ is a GMT. Then there exists a SGMT
$$
y_i=\prod_{j=1}^sy_j(1)^{b_{ij}}\mbox{ for }1\le i\le s
$$
and
$$
y_{s+\overline m}=\prod_{j=1}^sy_j(1)^{b_{s+\overline m,j}}(y_{s+\overline m}(1)+\alpha)
$$
such that the variables $(x(1),y(1))$ are prepared of type $(s_1,r_1,l_1)$ with $(s_1,r_1,l_1)\ge (s,r,l)$.
\end{Lemma}

\begin{proof} Let 
$\overline x_1,\ldots,\overline x_n$ be the variables defined by (\ref{eq2}) which lead to the variables $x_1(1),\ldots,x_n(1)$ of the statement of Lemma \ref{Lemma23} by the analytic change of variables defined in Lemma \ref{Lemma6}.

Let $\nu$ be the restriction of $\nu_e$ to the quotient field $K$ of $\CC[y_1,\ldots,y_s,y_{s+\overline m}]$, which contains $\CC[x_1,\ldots,x_r,x_{r+\overline m}]$. Then $\nu(y_1),\ldots,\nu(y_s)$ are rationally independent by Lemma \ref{Lemma25} and
$\nu(y_{s+\overline m})=\nu(x_{r+\overline m})$ is rationally dependent on $\nu(x_1),\ldots,\nu(x_r)$, hence  $\nu(y_{s+\overline m})$ is rationally dependent on $\nu(y_1),\ldots,\nu(y_s)$. Let $\pi:Z\rightarrow \AA^{r+1}$ be a projective morphism of nonsingular toric varieties such that $\overline x_1,\ldots,\overline x_r,\overline x_{r+\overline m}$ are regular parameters in $\mathcal O_{Z,p}$ where $p$ is the center of $\nu$ on $Z$. We have that
\begin{equation}\label{eq51}
\begin{array}{lll}
x_i&=&\prod_{j=1}^r\overline x_j^{a_{ij}}(\overline x_{r+\overline m}+\overline \alpha)^{a_{i,r+1}}\mbox{ for }1\le i\le r\mbox{ and}\\
x_{r+\overline m}&=& \prod_{j=1}^r\overline x_j^{a_{r+1,j}}(\overline x_{r+\overline m}+\overline \alpha)^{a_{r+1,r+1}}
\end{array}
\end{equation}
where $0\ne\overline\alpha\in \CC$.

Let $J$ be a (monomial) ideal in $\CC[x_1,\ldots,x_r,x_{r+\overline m}]$ whose blow up is $Z$. By principalization of ideals, there exists a toric projective morphism $\Lambda:W\rightarrow \AA^{s+1}$  which is a product of blow ups of non singular varieties such that 
 $J\mathcal O_W$ is locally principal. Let $q_1$ be the center of $\nu$ on $W$. Since $\nu(y_1),\ldots,\nu(y_s)$ are rationally independent, and $\Lambda$ factors through $Z$, we have that $\mathcal O_{W,q_1}$ dominates $\mathcal O_{Z,p}$ and $\mathcal O_{W,q_1}$ has regular parameters $\overline y_1,\ldots,\overline y_s,\overline y_{s+\overline m}$ such that 
\begin{equation}\label{eq52}
\begin{array}{lll}
y_i&=&\prod_{j=1}^s\overline y_j^{b_{ij}}(\overline y_{s+\overline m}+\overline \beta)^{b_{i,s+1}}\mbox{ for }1\le i\le s\mbox{ and}\\
y_{s+\overline m}&=& \prod_{j=1}^s\overline y_i^{b_{s+1,j}}(\overline y_{s+\overline m}+\overline \beta)^{b_{s+1,s+1}}
\end{array}
\end{equation}
where $0\ne\overline\beta\in \CC$, $b_{ij}\in \NN$ and ${\rm Det}(b_{ij})=\pm 1$.

The variety $W$ is the blow up of a monomial ideal $H$ in $\CC[y_1,\ldots,y_s,y_{s+\overline m}]$. Let $Y_1\rightarrow \tilde Y$ be the blow up of $H$ in a neighborhood of $e_{\tilde Y}$. Let $e_{Y_1}$ be the center of $e$ on $Y_1$. Then 
$$
\overline y_1,\ldots,\overline y_s,y_{s+1},\ldots,y_{s+\overline m-1},\overline y_{s+\overline m},y_{s+\overline m+1},\ldots,y_n
$$
are regular parameters in $\mathcal O_{Y_1,e_{Y_1}}^{\rm an}$.

In $\mathcal O_{X_1,e_{X_1}}^{\rm an}$, we have the following relations between the variables $\overline x$ and $x(1)$. 
\begin{equation}\label{eq53}
\begin{array}{lll}
\overline x_i&=&(x_{r+\overline m}(1)+\alpha)^{\overline c\gamma_i}x_i(1)\mbox{ for }1\le i\le r\mbox{ and}\\
\overline x_{r+\overline m}&=&(x_{r+\overline m}(1)+\alpha)^{\overline c}-\overline \alpha
\end{array}
\end{equation}
with $\alpha^{\overline c}=\overline \alpha$ and
$$
(a_{ij})\left(\begin{array}{c} \gamma_1\\ \vdots\\ \gamma_r\\1\end{array}\right) =
\left(\begin{array}{c}0\\ \vdots\\ 0\\ \frac{1}{\overline c}\end{array}\right)
$$
with
$$
\overline c={\rm det}\left(\begin{array}{ccc}
a_{11}&\cdots& a_{1r}\\
&\vdots&\\
a_{r1}&\cdots&a_{rr}
\end{array}\right)
{\rm det}\left(\begin{array}{ccc}
a_{11}&\cdots&a_{1,r+1}\\
&\vdots&\\
a_{r+1,1}&\cdots& a_{r+1,r+1}
\end{array}\right).
$$

In $\mathcal O_{Y_1,e_{Y_1}}^{\rm an}$, we have the following relations between the variables $\overline y$ and $y(1)$ of the proof of Lemma \ref{Lemma6}.
\begin{equation}\label{eq54}
\begin{array}{lll}
\overline y_i&=&(y_{s+\overline m}(1)+\beta)^{\overline d\tau_i}y_i(1)\mbox{ for }1\le i\le s\mbox{ and}\\
\overline y_{s+\overline m}&=&(y_{s+\overline m}(1)+\beta)^{\overline d}-\overline \beta
\end{array}
\end{equation}
with $\beta^{\overline d}=\overline \beta$ and
$$
(b_{ij})\left(\begin{array}{c} \tau_1\\ \vdots\\ \tau_s\\1\end{array}\right) =
\left(\begin{array}{c}0\\ \vdots\\ 0\\ \frac{1}{\overline d}\end{array}\right)
$$
with
$$
\overline d={\rm det}\left(\begin{array}{ccc}
b_{11}&\cdots& b_{1s}\\
&\vdots&\\
b_{s1}&\cdots&b_{ss}
\end{array}\right)
{\rm det}\left(\begin{array}{ccc}
b_{11}&\cdots&b_{1,s+1}\\
&\vdots&\\
b_{s+1,1}&\cdots& b_{s+1,s+1}
\end{array}\right).
$$

We have expressions
$$
\overline x_i= x_1^{g_{i1}}\cdots x_r^{g_{ir}} x_{r+\overline m}^{g_{i,r+1}}\mbox{ for }1\le i\le r
$$
and
$$
\overline x_{r+\overline m}+\overline \alpha= x_1^{g_{r+1,1}}\cdots x_r^{g_{r+1,r}}x_{r+\overline m}^{g_{r+1,r+1}}
$$
where $(g_{ij})=(a_{ij})^{-1}$ and
$$
\overline y_i= y_1^{h_{i1}}\cdots y_s^{h_{is}} y_{s+\overline m}^{h_{i,s+1}}\mbox{ for }1\le i\le s
$$
and
$$
\overline y_{s+\overline m}+\overline \beta= y_1^{h_{s+1,1}}\cdots y_s^{h_{s+1,s}}y_{s+\overline m}^{h_{s+1,s+1}}
$$
where $(h_{ij})=(b_{ij})^{-1}$.

Substituting (\ref{eq5}), we have
$$
\overline x_i=y_1^{d_{i1}}\cdots y_s^{d_{is}}y_{s+\overline m}^{d_{i,s+1}}\mbox{ for }1\le i\le r
$$
and
$$
\overline x_{r+\overline m}+\overline\alpha=y_1^{d_{r+1,1}}\cdots y_s^{d_{r+1,s}}y_{s+\overline m}^{d_{r+1,s+1}}
$$
where 
$$
(d_{ik})=(a_{ij})^{-1}\left(\begin{array}{cc} (c_{jk})&0\\ 0& 1\end{array}\right).
$$
We have
\begin{equation}\label{eq55}
\begin{array}{lll}
\overline x_i&=& \overline y_1^{e_{i1}}\cdots \overline y_s^{e_{is}}(\overline y_{s+\overline m}+\overline \beta)^{e_{i,s+1}}\mbox{ for $1\le i\le r$ and}\\
\overline x_{r+\overline m}+\overline \alpha&=& \overline y_1^{e_{r+1,1}}\cdots \overline y_s^{e_{r+1,s}}(\overline y_{s+\overline m}+\overline\beta)^{e_{r+1,s+1}}
\end{array}
\end{equation}
where $(e_{ij})=(d_{ij})(h_{ij})^{-1}$. Since $\nu(\overline x_{r+\overline m}+\overline\alpha)=\nu(\overline y_{s+\overline m}+\overline\beta)=0$ and $\nu(\overline y_1),\ldots,\nu(\overline y_s)$ are rationally independent we have that
$$
0=e_{r+1,1}=\cdots = e_{r+1,s}.
$$
We then have that $e_{s+1,s+1}\ne 0$ since $\mbox{rank}(e_{ij})=r+1$. We have that $e_{ij}\ge 0$ for $1\le i\le r+1$ and $1\le  j\le s+1$ since $\Lambda$ factors through $Z$. We compute

$$
\begin{array}{lll}
(e_{ij})\left(\begin{array}{c} \tau_1\\ \vdots\\ \tau_s\\ 1\end{array}\right)
&=&(a_{ij})^{-1}\left(\begin{array}{cc} (c_{ij})&0\\ 0 & 1\end{array}\right)
\left(\begin{array}{c} 0\\ \vdots\\ 0\\ \frac{1}{\overline d}\end{array}\right)\\
&=& (a_{ij})^{-1}\left(\begin{array}{c} 0\\ \vdots\\ 0 \\ \frac{1}{\overline d}\end{array}\right)
=\frac{\overline c}{\overline d}\left(\begin{array}{c} \gamma_1\\ \vdots\\ \gamma_r\\ 1\end{array}\right).
\end{array}
$$
Substituting  (\ref{eq53}) and  (\ref{eq54}) into (\ref{eq55}), we obtain
$$
x_i(1)(x_{r+\overline m}(1)+\alpha)^{\overline c\gamma_i}
=y_1(1)^{e_{i1}}\cdots y_s(1)^{e_{s1}}(y_{s+\overline m}(1)+\beta)^{\overline c\gamma_i}\mbox{ for }1\le i\le r
$$
and
$$
(\overline x_{r+\overline m}(1)+\alpha)^{\overline c}=(\overline y_{s+\overline m}(1)+ \beta)^{\overline c}.
$$
We thus have an expression (after possibly replacing $\overline y_{s+\overline m}$ with its product times a root of unity) 
$$
x_i(1)=\prod_{j=1}^sy_j(1)^{e_{ij}}\mbox{ for }1\le i\le r
$$
and
$$
x_{r+\overline m}(1)=y_{s+\overline m}(1)
$$
giving the conclusions of the lemma.
\end{proof}


\begin{Lemma}\label{Lemma35} Suppose that $(x,y)$ are prepared of type $(s,r,l)$, $\overline m>l$ and we have an expression
$$
x_{r+\overline m}=y_1^{c_{r+1,1}}\cdots y_s^{c_{r+1,s}}u
$$
where $u\in \CC\{\{y_1,\ldots,y_n\}\}$ is a unit and
$$
x_i=\prod_{j=1}^rx_j(1)^{a_{ij}}\mbox{ for }1\le i\le r\mbox{ and}
$$
$$
x_{r+\overline m}=\prod_{j=1}^rx_j(1)^{a_j}(x_{r+\overline m}(1)+\alpha)\mbox{ with }0\ne\alpha\in\CC
$$
is a GMT in $x_1,\ldots,x_r,x_{r+\overline m}$. Then there exists a SGMT
$$
y_i=\prod_{j=1}^sy_j(1)^{b_{ij}}\mbox{ for }1\le i\le s
$$
in $y_1,\ldots,y_s$ such that the variables $(x(1),y(1))$ are prepared of type $(s_1,r_1,l_1)$ with $(s_1,r_1,l_1)\ge (s,r,l)$.
\end{Lemma}

\begin{proof} By Lemma \ref{Lemma6}, the GMT $(x)\rightarrow (x(1))$ is determined by a monoidal transform
$$
x_i=\left(\prod_{j=1}^r\overline x_i^{g_{ij}}\right)(\overline x_{r+\overline m}+\overline \alpha)^{g_{i,r+1}}\mbox{ for }1\le i\le r\mbox{ and}
$$
$$
x_{r+\overline m}=\left(\prod_{j=1}^r\overline x_i^{g_{r+1,j}}\right)(\overline x_{r+\overline m}+\overline\alpha)^{g_{r+1,r+1}}
$$
where ${\rm det}(g_{ij})=\pm 1$ and 
\begin{equation}\label{eq36}
\begin{array}{lll}
x_i(1)&=&(\overline x_{r+\overline m}+\overline \alpha)^{\lambda_i}\overline x_i\mbox{ for }1\le i\le r\mbox{ and}\\
x_{r+\overline m}(1)&=& (\overline x_{r+\overline m}+\overline \alpha)^{\lambda}-\overline\alpha^{\lambda},\,\,\alpha=\overline\alpha^{\lambda}
\end{array}
\end{equation}
for suitable $\lambda_i,\lambda\in \QQ$ (with $\lambda\ne 0$). Letting $(e_{ij})=(g_{ij})^{-1}$ and $(d_{ij})=(g_{ik})^{-1}(c_{kj})$, we have
$$
\overline x_i=\left(\prod_{j=1}^s y_j^{d_{ij}}\right) u^{e_{i,r+1}}\mbox{ for }1\le i\le r\mbox{ and}
$$
$$
\overline x_{r+\overline m}+\overline \alpha=\left(\prod_{j=1}^s y_j^{d_{r+1,j}}\right) u^{e_{r+1,r+1}}.
$$
The values $\nu_e(y_1),\ldots,\nu_e(y_s)$ are rationally independent by Lemma \ref{Lemma25}. Since 
$$
\nu_e(x_{r+\overline m}+\overline\alpha)=\nu_e(u)=0,
$$
 we have that $d_{r+1,j}=0$ for $1\le j\le s$. Thus by (\ref{eq36}),
 \begin{equation}\label{eq37}
 x_{r+\overline m}(1)=u^{\lambda e_{r+1,r+1}}-\alpha\in \CC\{\{y_1,\ldots,y_n\}\}.
 \end{equation}
Write
$$
\prod_{j=1}^sy_j^{d_{ij}}=\frac{M_i}{N_i}
$$
where $M_i,N_i$ are monomials in $y_1,\ldots,y_s$ for $1\le i\le r$. Let $K$ be the ideal $K=\prod_{i=1}^s(M_i,N_i)$ in $\CC\{\{y_1,\ldots,y_s\}\}$. By Lemma \ref{Lemma21}, there exists a (monomial) SGMT in $y_1,\ldots,y_s$
$$
y_i=\prod_{j=1}^sy_j(1)^{b_{ij}}\mbox{ for }1\le i\le s
$$
such that $K\mathcal O_{Y(1),e_{Y(1)}}^{\rm an}$ is a principal ideal. $y_1(1),\ldots,y_s(1)$ are independent by Lemma \ref{Lemma1}.
Since $\nu_e(M_i/N_i)=\nu_e(x_i)>0$ we have that $N_i$ divides $M_i$ in $\mathcal O_{Y(1),e_{Y(1)}}^{\rm an}$ for $1\le i\le s$ and so we have an expression
$$
\overline x_i=\left(\prod_{j=1}^sy_j(1)^{c_{ij}(1)}\right)u^{e_{i,r+1}}\mbox{ for }1\le i\le r
$$
with $c_{ij}(1)\in \NN$. Since $x_i(1)$ is necessarily a Laurent monomial in $y_1(1),\ldots,y_s(1)$ for $1\le i\le s$, comparing with (\ref{eq36}), we see that 
$$
x_i(1)=\prod_{j=1}^sy_j(1)^{c_{ij}(1)}\mbox{ for }1\le i\le r.
$$
Since $x_{r+\overline m}(1)\in \CC\{\{y_1(1),\ldots,y_n(1)\}\}$ by (\ref{eq37}), we have attained the conclusions of the lemma.

\end{proof}

Suppose that $(x,y)$ are prepared of type $(s,r,l)$.
We will perform sequences of transformations of the following 10 types for $1\le i\le 10$ each of which will be called a transformation of type i) from the variables $(x,y)$ to $(x(1),y(1))$. The variables $x(1)$ and $y(1)$ are respective regular parameters in $\mathcal O_{X(1),e_{X(1)}}^{\rm an}$ and $\mathcal O_{Y(1),e_{Y(1)}}^{\rm an}$ from the corresponding diagram of  quasi regular analytic maps
$$
\begin{array}{ccc}
Y(1)&\stackrel{\phi(1)}{\rightarrow}&X(1)\\
\downarrow&&\downarrow\\
\tilde Y&\stackrel{\tilde \phi}{\rightarrow}&\tilde X
\end{array}
$$
where $Y(1)\rightarrow \tilde Y\rightarrow Y\in e$ and $X(1)\rightarrow \tilde X\rightarrow X\in e$.
We have that $(x(1),y(1))$ is prepared of type $(s_1,r_1,l_1)$ with $(s_1,r_1,l_1)\ge (s,r,l)$ for all 10 types of transformations.
The fact that none of $s,r$ or $r+l$ can go down after a transformation follows from Lemmas \ref{Lemma1}, \ref{Lemma6} and \ref{Lemma10}. Existence of transformations of types 2) and 4) follow from Lemmas \ref{Lemma22} and  \ref{Lemma23}. A transformation of type 9) will be constructed in the proof of Proposition \ref{Theorem2} (using Lemma \ref{Lemma35}).

Transformations of types 1) to 4) are the most basic and are used most of the time. Transformations of types 1) - 6) and 1) - 8) are used in blocks, depending on the lemma or proposition. Transformations of type 3), 5) or 10) are often used to make  a Tschirnhaus transformation (Lemma \ref{Lemma9}). A transformation of type 8) is often used to make a change of variables, giving an increase in $r$. A transformation of type 9) is used at the end of the proof of Proposition \ref{Theorem2}.

\begin{enumerate}
\item[1)] A (necessarily monomial) SGMT in $y_1,\ldots,y_s$,
$$
y_i=\prod_{j=1}^sy_j(1)^{b_{ij}}\mbox{ for }1\le i\le s,
$$
with $\mbox{Det}(b_{ij})=\pm 1$.
\item[2)] A (necessarily monomial) SGMT in $x_1,\ldots,x_r$ followed by a (necessarily monomial) SGMT in $y_1,\ldots,y_s$,
$$
x_i=\prod_{j=1}^rx_j(1)^{a_{ij}}\mbox{ for }1\le i\le r
$$
and
$$
y_i=\prod_{j=1}^sy_j(1)^{b_{ij}}\mbox{ for }1\le i\le s
$$
with $\mbox{Det}(a_{ij})=\pm 1$ and $\mbox{Det}(b_{ij})=\pm 1$.
\item[3)] A change of variables $x_{r+\overline m}(1)=x_{r+\overline m}-\Phi$ for some $\overline m$ with $1\le \overline m\le l$ and
$\Phi\in \CC\{\{x_1,\ldots,x_{r+\overline m-1}\}\}$, followed by a change of variables
$y_{s+\overline m}(1)=y_{s+\overline m}-\Phi$.
\item[4)] A SGMT in $x_1,\ldots,x_r,x_{r+\overline m}$ followed by a SGMT in $y_1,\ldots,y_s, y_{s+\overline m}$ for some $\overline m$ with $1\le \overline m\le l$,
$$
x_i=\prod_{j=1}^rx_j(1)^{a_{ij}}\mbox{ for }1\le i\le r\mbox{ and }x_{r+\overline m}=\prod_{j=1}^rx_j(1)^{a_j}(x_{r+\overline m}(1)+\alpha)
$$
for some $0\ne \alpha\in \CC$, and
$$
y_i=\prod_{j=1}^sy_j(1)^{b_{ij}}\mbox{ for }1\le i\le s\mbox{ and }y_{s+\overline m}=\prod_{j=1}^sy_j(1)^{b_j}(y_{s+\overline m}(1)+\alpha)
$$
with $\mbox{Det}(a_{ij})\ne 0$ and $\mbox{Det}(b_{ij})\ne 0$ and $\prod_{j=1}^sy_j(1)^{b_j}=\prod_{j=1}^rx_j(1)^{a_j}$.
\item[5)] A change of variables
$y_{s+\overline m}(1)=F$ with $F\in \CC\{\{y_1,\ldots,y_{s+\overline m}\}\}$ and 
$$
\mbox{ord }F(0,\ldots,0,y_{s+\overline m})=1
$$
 for some $\overline m$ with $\overline m>l$.
\item[6)] A SGMT in $y_1,\ldots,y_s, y_{s+\overline m}$, for some $\overline m$ with $l+1\le \overline m\le n-s$.
\item[7)] An interchange of variables $y_{s+i}$ and $y_{s+\overline m}$ with $s+l<s+i<s+\overline m\le n$.
\item[8)] A change of variables, replacing $y_i$ with $y_i\gamma^{c_i}$ for $1\le i\le s$ for some unit $\gamma\in\CC\{\{ y_1,\ldots,y_n\}\}$ and $c_i\in\QQ$ such that the form (\ref{eq5}) is preserved.
\item[9)] 
 A SGMT in $x_1,\ldots,x_r,x_{r+\overline m}$ followed by a SGMT in $y_1,\ldots,y_s$  (supposing  that $\overline m>l$ and
$$
x_{r+\overline m}=y_1^{b_1}\cdots y_s^{b_s}u
$$
where $u\in \CC\{\{ y_1,\ldots,y_n\}\}$ is a unit),
$$
x_i=\prod_{j=1}^rx_j(1)^{a_{ij}}\mbox{ for }1\le i\le r\mbox{ and }x_{r+\overline m}=\prod_{j=1}^rx_j(1)^{a_{j}}(x_{r+\overline m}(1)+\alpha)
$$
for some $0\ne \alpha\in \CC$, and
$$
y_i=\prod_{j=1}^sy_j(1)^{b_{ij}}\mbox{ for }1\le i\le s
$$
with $\mbox{Det}(b_{ij})=\pm 1$ and $\mbox{Det}(a_{ij})\ne 0$ and $\prod_{j=1}^sy_j^{b_j}=\prod_{j=1}^rx_j(1)^{a_{j}}$.
\item[10)] A change of variables, replacing $x_{r+\overline m}$ with $x_{r+\overline m}-\Phi$ for some $l<\overline m\le m-r$
and $\Phi\in \CC\{\{x_1,\ldots,x_{r+\overline m-1}\}\}$.
\end{enumerate}

In the following, we will assume that $(s,r,l)$ is preserved by these transformations. If this does not hold, then we just start over again with the assumption of the higher $(s,r,l)$. As these numbers cannot increase indefinitely, we will eventually reach a situation where they remain stable under the above transformations.

A sequence of transformations 
$$
(x,y)\rightarrow (x(1),y(1)) \rightarrow \cdots \rightarrow (x(t-1),y(t-1))\rightarrow (x(t),y(t))
$$
will be called a sequence of transformations from $(x,y)$ to $(x(t),y(t))$.

Observe that a sequence of transformations (which are of types 1) - 10)) satisfy the condition (\ref{eq*}).

\section{A decomposition  of series}\label{DecSeries}

In this section, suppose that $(x,y)$ are prepared of type $(s,r,l)$. As commented after  (\ref{eq5}), we have a natural inclusion of formal power series rings
$$
\CC[[x_1,\ldots,x_{r+l}]]\subset \CC[[y_1,\ldots,y_{s+l}]].
$$

\begin{Definition}\label{Def500}
Suppose that $g\in k[[y_1,\ldots,y_n]]$. We will say that $g$ is algebraic over $x_1,\ldots,x_{r+l}$ if $g\in \CC[[y_1,\ldots,y_{s+l}]]$ and $g$ has an  expansion 
\begin{equation}\label{eq4}
g=\sum a_{i_1,\ldots,i_{s+l}}y_1^{i_1}\cdots y_s^{i_s}y_{s+1}^{i_{s+1}}\cdots y_{s+l}^{i_{s+l}}
\end{equation}
where $a_{i_1,\ldots,i_{s+l}}\in \CC$ is nonzero only if 
$$
\mbox{rank}\left(\begin{array}{ccc} 
c_{11}&\cdots& c_{1s}\\
&\vdots&\\
c_{r1}&\cdots&c_{rs}\\
i_1&\cdots&i_s
\end{array}\right)=r.
$$
\end{Definition}

Observe that the property that $g$ is algebraic over $x_1,\ldots,x_{r+l}$ is preserved by a transformation of type 8).

\begin{Lemma}\label{Lemma26} Suppose that $x_1^{b_1}\cdots x_r^{b_r}$ with $b_1,\ldots,b_r\in \ZZ$ is such that 
$\prod_{i=1}^r(y_1^{c_{i1}}\cdots y_s^{c_{is}})^{b_i}\in \CC[y_1,\ldots,y_s]$ is algebraic over $x_1,\ldots,x_r$. Then there exists a SGMT
$$
x_i=\prod_{j=1}^rx_j(1)^{a_{ij}}\mbox{ for }1\le i\le r
$$
such that 
$$
x_1^{b_1}\cdots x_r^{b_r}=x_1(1)^{b_1(1)}\cdots x_r(1)^{b_r(1)}
$$
with $b_i(1)\in \NN$ for all $i$.
\end{Lemma}

\begin{proof} Let $\nu$ be the restriction of $\nu_e$ to the quotient field of $\CC[y_1,\ldots,y_s]$. We have $\nu(x_1^{b_1}\cdots x_r^{b_r})\ge 0$. Write $x_1^{b_1}\cdots x_r^{b_r}=\frac{M_1}{M_2}$ where $M_1$ and $M_2$ are monomials in $x_1,\ldots,x_r$. We have that
$\nu(M_1)\ge \nu(M_2)$. By Lemma \ref{Lemma21}, there exists a monomial SGMT in $x_1,\ldots,x_r$ such that the ideal generated by $M_1$ and $M_2$ in $\mathcal O_{X(1),e_{X(1)}}^{\rm an}$ is principal. Since $\nu(M_1)\ge \nu(M_2)$, we have that $M_2$ divides $M_1$ in $\mathcal O_{X(1),e_{X(1)}}^{\rm an}$, giving the conclusions of the lemma.
\end{proof}

Suppose that $g\in \CC[[y_1,\ldots,y_{s+l}]]$. As on page 1540 of \cite{LMTE}, we have an expression
\begin{equation}\label{eq6}
g=\sum_{[\Lambda]\in \left(\ZZ^s/(\QQ^rC)\cap \ZZ^s\right)} h_{[\Lambda]}
\end{equation}
where
\begin{equation}\label{eq60}
h_{[\Lambda]}=\sum_{\alpha\in \NN^s\mid [\alpha]=[\Lambda]}g_{\alpha}y_1^{\alpha_1}\cdots y_s^{\alpha_s}
\end{equation}
with $g_{\alpha}\in \CC[[y_{s+1},\ldots,y_{s+l}]]$. 

If $g\in \CC\{\{y_1,\ldots,y_{s+l}\}\}$ then each $h_{[\Lambda]}\in \CC\{\{y_1,\ldots,y_{s+l}\}\}$ by the criterion of (\ref{eq42}).

\begin{Proposition}\label{Prop2}
Suppose that $\Lambda=(\lambda_1,\ldots,\lambda_s)\in \NN^s$ is fixed.
Then there exists a SGMT of type 2), $(x,y)\mapsto (x(1),y(1))$,  $w_1,\ldots,w_r\in \NN$ and $d\in\ZZ_{>0}$ such that 
\begin{equation}\label{eq64}
\delta_{[\Lambda]}:= \frac{h_{[\Lambda]}}{y_1^{\lambda_1}\cdots y_s^{\lambda_s}}x_1^{w_1}\cdots x_r^{w_r}\in \CC[[x_1(1)^{\frac{1}{d}},\ldots,x_r(1)^{\frac{1}{d}},x_{r+1}(1),\ldots,x_{r+l}(1)]].
\end{equation}
If $[\Lambda]=0$, we further have
$$
h_{[\Lambda]}\in \CC[[x_1(1)^{\frac{1}{d}},\ldots,x_r(1)^{\frac{1}{d}},x_{r+1}(1),\ldots,x_{r+l}(1)]].
$$
 If $g\in \CC\{\{y_1,\ldots,y_{s+l}\}\}$, then $\delta_{[\Lambda]}\in \CC\{\{x_1(1)^{\frac{1}{d}},\ldots,x_r(1)^{\frac{1}{d}},x_{r+1}(1),\ldots,x_{r+l}(1)\}\}$ by the criterion (\ref{eq42}).
\end{Proposition}

\begin{proof} Write $C=(C_1,\ldots,C_s)$ and let $\Phi:\QQ^r\rightarrow \QQ^s$ be defined by $\Phi(v)=vC$ for $v\in \QQ^r$. $\Phi$ is injective since $C$ has rank $r$. Let $G=\Phi^{-1}(\ZZ^s)$. For $\Lambda=(\lambda_1,\ldots,\lambda_s)\in \NN^s$, define
$$
P_{\Lambda}=\{v\in \QQ^r\mid vC_i+\lambda_i\ge 0\mbox{ for }1\le i\le s\}.
$$
For $\Lambda\in \NN^s$, we have
$$
h_{[\Lambda]}=y_1^{\lambda_1}\cdots y_s^{\lambda_s}\left(\sum_{v=(v_1,\ldots,v_r)\in G\cap P_{\Lambda}}x_1^{v_1}\cdots x_r^{v_r}g_v\right)
$$
where $g_v\in \CC[[x_{r+1},\ldots,x_{r+l}]]$ and we have reindexed the $g_{\alpha}=g_{vC+\Lambda}$ in (\ref{eq60}) as $g_v$. Let
$$
H=\{v\in \ZZ^r\mid vC_i\ge 0\mbox{ for }1\le i\le s\},
$$
$$
I=\{v\in G\mid vC_i\ge 0\mbox{ for }1\le i\le s\}
$$
and for $\Lambda=(\lambda_1,\ldots,\lambda_s)\in \NN^s$,
$$
M_{\Lambda}=\{v\in G\mid vC_i+\lambda_i\ge 0\mbox{ for }1\le i\le s\}.
$$
We have that $P_{\Lambda}$ is a rational polyhedral set in $\QQ^r$ whose associated cone is
$$
\sigma=\{v\in \QQ^r\mid vC_i=0\mbox{ for }1\le i\le s\}=\{0\}.
$$
Let $W=\QQ^r$. We have that $G$ is a lattice in $W$ and $P_{\Lambda}$ is strongly convex. Thus $M_{\Lambda}=P_{\Lambda}\cap G$ is a finitely generated module over the semigroup $I$ (cf. Theorem 7.1 \cite{CHR}). Let $\overline n=[G:\ZZ^r]$. We have that $\overline n x\in H$ for all $x\in I$. Gordan's Lemma (cf. Proposition 1, page 12 \cite{Fu}) implies that $H$ and $I$ are finitely generated semigroups. There exist
$w_1,\ldots,w_{\overline l}\in I$ which generate $I$ as a semigroup and there exist $\overline v_1,\ldots,\overline v_{\overline a}\in H$ which generated $H$ as a semigroup. Then the finite set 
$$
\{a_1w_1+\cdots a_{\overline l}w_{\overline l}\mid a_i\in \NN\mbox{ and }0\le a_i\le \overline n\mbox{ for }1\le i\le \overline l\}
$$
generates $I$  as an $H$-module. We thus have that $M_{\Lambda}$ is a finitely generated module over the semigroup $H$. Thus there exist
$\overline u_1,\ldots,\overline u_{\overline b}\in M_{\Lambda}$ such that if $v=(v_1,\ldots,v_r)\in M_{\Lambda}$, then
$$
v=\overline u_i+\sum_{j=1}^{\overline a}n_j\overline v_j
$$
for some $1\le i\le b$ and $n_1,\ldots, n_{\overline a}\in \NN$. Thus
$$
x_1^{v_1}\cdots x_r^{v_r}=x_1^{\overline u_{i,1}}\cdots  x_r^{\overline u_{i,r}}\prod_{j=1}^{\overline a}(x_1^{\overline v_{j,1}}\cdots x_r^{\overline v_{j,r}})^{n_j}
$$
where $\overline u_i=(\overline u_{i,1},\ldots,\overline u_{i,r})$ for $1\le i\le \overline b$ and $\overline v_j=(\overline v_{j,1},\ldots,\overline v_{j,r})$ 
for $1\le j\le \overline a$.
By Lemma \ref{Lemma26} and Lemma \ref{Lemma22}, there exists a transformation of type 2) such that for $1\le j\le \overline a$,
$$
x_1^{\overline v_{j,1}}\cdots \overline x_r^{\overline v_{j,r}}=x_1(1)^{\overline v(1)_{j,1}}\cdots x_r(1)^{\overline v(1)_{j,r}}
$$
with $(\overline v(1)_{j,1},\ldots,\overline v(1)_{j,r})\in \NN^r$ for $1\le j\le \overline a$. We then have expressions of all
$\Lambda=(\lambda_1,\ldots,\lambda_s)\in\NN^s$,  where $\overline u_1,\ldots,\overline u_{\overline b}\in \QQ^r$ depend only on $\Lambda$,
$$
h_{[\Lambda]}=y_1(1)^{\lambda_1(1)}\cdots y_s(1)^{\lambda_s(1)}\left[
\sum_{i=1}^{\overline b}x_1(1)^{\overline u_{i,1}(1)}\cdots x_r(1)^{\overline u_{i,r}(1)}g_i\right]
$$
where $g_i\in \CC[[x_{r+1}(1),\ldots,x_{r+l}(1)]]$, 
$$
\Lambda(1):= (\lambda_1(1),\ldots,\lambda_s(1))=\Lambda(b_{ij})
$$
and 
$$
\overline u(1)_i=(\overline u_{i,1}(1),\ldots,\overline u_{i,r}(1))=\overline u_i(a_{ij}).
$$
If $\Lambda=0$, we have  $M_{\Lambda}=I$ so that $x_1^{\overline u_{i,1}}\cdots \overline x_r^{\overline u_{i,r}}$ is a monomial in $y_1,\ldots,y_s$ for $1\le i\le \overline b$, so we can construct a transformation of type 2), $(x,y)\mapsto (x(1),y(1))$ so that we also have that the $\overline u_i(1)$ satisfy $\overline u_i(1)\in \QQ_{\ge 0}^r$ for $1\le i\le \overline b$. 

Now let $d$ be a common denominator of the coefficients of the $\overline u_i(1)$ for $1\le i\le \overline b$. If $[\Lambda]=0$, we have that
$$
h_{[\Lambda]}\in \CC[[x_1(1)^{\frac{1}{d}},\ldots,x_r(1)^{\frac{1}{d}}, x_{r+1}(1),\ldots,x_{r+l}(1)]].
$$
 If $[\Lambda]\ne 0$, we choose $w=(w_1,\ldots,w_r)\in \NN^r$ such that $w+\overline u_i\in \QQ_{\ge 0}^r$ for $1\le i\le \overline b$.  Then
 $$
 \frac{h_{[\Lambda]}}{y_1^{\lambda_1}\cdots y_s^{\lambda_s}}x_1^{w_1}\cdots x_r^{w_r}\in
 \CC[[x_1(1)^{\frac{1}{d}},\ldots,x_r(1)^{\frac{1}{d}}, x_{r+1}(1),\ldots,x_{r+l}(1)]].
$$

\end{proof}

\begin{Lemma}\label{Lemma11}
Suppose that $f\in \CC\{\{ x_1,\ldots,x_m\}\}\subset \CC[[y_1,\ldots,y_n]]$ is algebraic over $x_1,\ldots,x_{r+l}$. Then $f\in \CC\{\{ x_1,\ldots,x_{r+l}\}\}$.
\end{Lemma}

\begin{proof} By Proposition \ref{Prop2} and by the criterion of (\ref{eq42}), there exists a monomial GMT

\begin{equation}\label{eq41}
\begin{array}{lll}
x_1&=& x_1(1)^{a_{11}(1)}\cdots x_r(1)^{a_{1r}(1)}\\
&&\vdots\\
x_r&=& x_1(1)^{a_{r1}(1)}\cdots x_r(1)^{a_{rr}(1)}
\end{array}
\end{equation}
with $\mbox{Det}(a_{ij}(1))=\pm 1$ and $d\in \ZZ_+$ such that
$$
f\in \CC\{\{ x_1(1)^{\frac{1}{d}},\ldots,x_r(1)^{\frac{1}{d}},x_{r+1},\ldots,x_{r+l}\}\}.
$$
Let
$$
\begin{array}{lll}
g(z)&=&\prod_{i_1,\ldots,i_r=1}^d(z-f(\omega^{i_1}x_1(1)^{\frac{1}{d}},\ldots,\omega^{i_r}x_r(1)^{\frac{1}{d}},x_{r+1},\ldots,x_{r+l}))\\
&&\in \CC\{\{ x_1(1),\ldots,x_r(1),x_{r+1},\ldots,x_{r+l}\}\}[z]
\end{array}
$$
where $\omega$ is a primitive complex $d$-th root of unity. We have that $f$ is integral over $\CC\{\{ x_1(1),\ldots,x_r(1),x_{r+1},\ldots,x_{r+l}\}\}$ since $f$ is a root of $g(z)=0$. But 
$$
f\in \CC\{\{ x_1(1),\ldots,x_r(1),x_{r+1},\ldots,x_{m}\}\}
$$
 and $\CC\{\{ x_1(1),\ldots,x_r(1),x_{r+1},\ldots,x_{r+l}\}\}$ is integrally closed in 
 $$
 \CC\{\{ x_1(1),\ldots,x_r(1),x_{r+1},\ldots,x_{m}\}\}
 $$
  so $f\in \CC\{\{ x_1(1),\ldots,x_r(1),x_{r+1},\ldots,x_{r+l}\}\}$. Substituting (\ref{eq41}) into the series expansion of $f$ in terms of $x_1,\ldots,x_m$ we obtain that $f\in \CC\{\{ x_1,\ldots,x_{r+l}\}\}$.
 
\end{proof}

\begin{Lemma}\label{Lemma7} Suppose that $g\in \CC[[y_1,\ldots,y_{s+l}]]$ has an expression $g=\sum h_{[\Lambda]}$ and one of the transformations 1) - 4) are performed.
Then $g\in \CC[[y_1(1),\ldots,y_{s+l}(1)]]$ and if $g=\sum h'_{[\Lambda']}$ is the decomposition in terms of the variables $y_1(1),\ldots,y_{s+l}(1)$
and $x_1(1),\ldots,x_{r+l}(1)$, then 
\begin{equation}\label{eq70}
h_{[\Lambda]}=h'_{[\Lambda \overline B]}
\end{equation}
where 
$$
\overline B=\left(\begin{array}{ccc} b_{11}&\cdots&b_{1s}\\
&\vdots&\\
b_{s1}&\cdots&b_{ss}\end{array}\right)
$$
with $b_{ij}$ defined as in the definitions of types 1), 2) and  4)  (and with $\overline B$ being the identity matrix for a transformation of type   3).

In particular, if a transformation of type 1) -  10) is performed, then $f\in \CC[[y_1,\ldots,y_n]]$ is algebraic over $x_1,\ldots,x_{r+l}$ if and only if $f$ is algebraic over $x_1(1),\ldots,x_{r+l}(1)$.
\end{Lemma}

\begin{proof} We will prove (\ref{eq70}) in the case of a transformation of type 4). The other cases are simpler. With the notation of (\ref{eq60}), we have expansions
$$
g_{\alpha}=\sum_i(y_1(1)^{b_1}\cdots y_s(1)^{b_s})^ig_{\alpha,i}
$$
with $g_{\alpha,i}\in \CC[[y_{s+1}(1),\ldots,y_{s+l}(1)]]$ so
$$
\begin{array}{lll}
h_{[\Lambda]}&=&\sum_{[\alpha]=[\Lambda]}\prod_{j=1}^s(y_1(1)^{b_{j1}}\cdots y_s(1)^{b_{js}})^{\alpha_j}(\sum_i(y_1(1)^{b_1}\cdots y_s(1)^{b_s})^ig_{\alpha,i})\\
&=& \sum_{\overline\alpha}y_1(1)^{\overline\alpha_1}\cdots y_s^{\overline \alpha_s}\left(\sum_i(y_1(1)^{b_1}\cdots y_s(1)^{b_s})^ig_{\alpha,i}\right)
\end{array}
$$
where 
\begin{equation}\label{eq61}
\alpha \overline B=\overline\alpha
\end{equation}
with $\overline \alpha=(\overline\alpha_1,\ldots,\overline \alpha_s)$. Write
$$
\overline A=\left(\begin{array}{ccc}
a_{11}&\cdots&a_{1r}\\
&\vdots&\\
a_{r1}&\cdots& a_{rr}
\end{array}\right)
$$
and 
$$
C(1)=\left(\begin{array}{ccc}
c_{11}(1)&\cdots& c_{1s}(1)\\
&\vdots&\\
c_{r1}(1)&\cdots&c_{rs}(1)
\end{array}\right).
$$
We showed in the  proof of Lemma \ref{Lemma23} (where $(e_{ij})$ is defined) that
$$
A(e_{ij})=\left(\begin{array}{cc} C&0\\0&1\end{array}\right)B.
$$
We have that 
$$
A=\left(\begin{array}{cc}
\overline A&*\\
*&*\end{array}\right),\,\,
B=\left(\begin{array}{cc}
\overline B&*\\
*&*\end{array}\right),\,\,
(e_{ij})=\left(\begin{array}{cc} C(1)&*\\
0&*\end{array}\right).
$$
We obtain that
\begin{equation}\label{eq63}
\overline A C(1)=C\overline B.
\end{equation}
From $y_1(1)^{b_1}\cdots y_s(1)^{b_s}=x_1(1)^{a_1}\cdots x_r(1)^{a_r}$ we obtain
$$
(a_1,\ldots,a_r)C(1)=(b_1,\ldots,b_s)
$$
and so 
$$
(b_1,\ldots,b_s)\in \QQ^rC(1)\cap \ZZ^s.
$$
Since $\overline A$ and $\overline B$ are invertible with integral coefficients, we have from (\ref{eq63}) that for $\alpha,\beta\in \ZZ^s$, $\alpha-\beta\in \QQ^rC\cap \ZZ^s$ if and only if $\alpha\overline B-\beta \overline B\in \QQ^rC(1)\cap \ZZ^s$, from which we obtain (\ref{eq70}).

\end{proof}

\section{Monomialization}\label{Mon}

\begin{Lemma}\label{Lemma51} Suppose that the variables $(x,y)$ are prepared of type $(s,r,l)$ and there exists $t$ with $r<t\le r+ l$ such that
$x_1,\ldots,x_r,x_t$ are independent. Then there exists a transformation of type 6) with $\overline m=t-r$, possibly followed by a tranformation of type 8) $(x,y)\rightarrow (x(1),y(1))$ such that $(x(1),y(1))$ are prepared of type $(s_1,r_1,l_1)$ with $(s_1,r_1,l_2)>(s,r,l)$.
\end{Lemma}

\begin{proof} Without loss of generality, we may assume that $t=r+1$. Since $y_1,\ldots,y_s$ are independent and $y_1,\ldots,y_s,y_{s+1}$ are dependent, there exists by Lemmas \ref{Lemma6} and \ref{Lemma50} a SGMT $(y)\rightarrow (y(1))$ (a transformation of type 6) with $\overline m=t-r=1$) defined by
$$
y_i=\prod_{j=1}^sy_j(1)^{b_{ij}}\mbox{ for $1\le j\le s$ and}
$$
$$
y_{s+1}=\prod_{j=1}^sy_j(1)^{b_j}(y_{s+1}(1)+\alpha)\mbox{ with }\alpha\ne 0.
$$
This gives us an expression
$$
x_i=\prod_{j=1}^s y_j(1)^{c_{ij}(1)}\mbox{ for $1\le i\le r$ and}
$$
$$
x_{r+1}=\prod_{j=1}^s y_j(1)^{b_j}(y_{s+1}(1)+\alpha).
$$
If $s_1>s$ we are done. Otherwise, we must have that
$$
\mbox{rank}\left(\begin{array}{ccc}
c_{11}(1)&\cdots&c_{1s}(1)\\
&\vdots&\\
c_{r1}(1)&\cdots&c_{rs}(1)\\
b_1&\cdots&b_s
\end{array}\right)=r+1
$$
since $x_1,\ldots,x_{r+1}$ are independent. Thus after making a change of variables in $y_1,\ldots,y_s$ (a transformation of type 8)) with 
$\gamma=(y_{s+1}(1)+\alpha)$) we obtain an increase  $r_1>r$ (and $(s_1,r_1,l_1)>(s,r,l)$).
\end{proof}

\begin{Lemma}\label{Prop1} Suppose that $(x,y)$ are prepared of type $(s,r,l)$ and $g\in \CC\{\{x_1,\ldots,x_{r+l}\}\}$. Then either there exists a sequence of transformations $(x,y)\rightarrow (x(1),y(1))$ such that $(x(1),y(1))$ are prepared of type $(s_1,r_1,l_1)$ with $(s_1,r_1,l_1)>(s,r,l)$ or there exists a sequence of transformations of the types 2) - 4) $(x,y)\rightarrow (x(1),y(1))$ such that 
$(x(1),y(1))$ are prepared of type $(s_1,r_1,l_1)$ with $(s_1,r_1,l_1)=(s,r,l)$
and we have an expression
$$
g=x_1(1)^{d_1}\cdots x_r(1)^{d_r}u
$$
with $u\in \CC\{\{x_1(1),\ldots, x_{r+l}(1)\}\}$ a unit.
\end{Lemma}

\begin{proof} In the course of the proof, we may assume that all transformations do not lead to an increase in $(s,r,l)$.
We will establish the lemma by induction on $t$ with $g\in \CC\{\{ x_1,\ldots,x_t\}\}$ for $r\le t\le r+l$. 
We will establish the lemma then with the further  restriction that all transformations of types 3) and 4) have $\overline m\le t-r$
and we will obtain $u\in \CC\{\{ x_1(1),\ldots,x_t(1)\}\}$.

We first prove the lemma for $t=r$, so suppose $g\in \CC\{\{ x_1,\ldots,x_r\}\}$. Expand
$$
g=\sum a_{i_1,\ldots,i_r}x_1^{i_1}\cdots x_r^{i_r}\mbox{ with }a_{i_1,\ldots,i_r}\in \CC.
$$
Let $I$ be the ideal 
$$
I=(x_1^{i_1}\cdots x_r^{i_r}\mid a_{i_1,\ldots,i_r}\ne 0).
$$
The ideal $I$ is generated by $x_1^{i_1(1)}\cdots x_r^{i_r(1)},\ldots,x_1^{i_1(k)}\cdots x_r^{i_r(k)}$
for some $i_1(1),\ldots,i_r(k)$ with $k\in \ZZ_{>0}$. By performing a transformation of type 2) $(x,y)\rightarrow (x(1),y(1))$ we  may principalize the ideal $I$ (by Lemma \ref{Lemma21}). Suppose that 
$x_1(1)^{a_1}\cdots x_r(1)^{a_r}$ is a generator of $I\mathcal O_{X(1),e_{X(1)}}^{\rm an}$. Then since $x_1,\ldots,x_r$ are independent, we have that $g=x_1(1)^{a_1}\cdots x_r(1)^{a_r}u$ where $u\in\CC\{\{ x_1(1),\ldots,x_r(1)\}\}$ is a unit, obtaining the conclusions of the lemma when $t=r$.

Now suppose that $l+r\ge t>r$, $g\in \CC\{\{ x_1,\ldots,x_t\}\}$ and the lemma is true in $\CC\{\{ x_1,\ldots,x_{t-1}\}\}$. We may then assume that
$g\in \CC\{\{ x_1,\ldots,x_t\}\}\setminus \CC\{\{ x_1,\ldots,x_{t-1}\}\}$.
Expand
$$
g=\sum_{i=0}^{\infty}\sigma_i x_t^i\mbox{ with }\sigma_i\in \CC\{\{ x_1,\ldots,x_{t-1}\}\}.
$$
Suppose that $\sigma_0,\ldots,\sigma_k$ generate the ideal $I=(\sigma_i\mid i\in \NN)$. By induction on $t$, there exists a sequence of transformations of types 2) - 4) $(x,y)\rightarrow (x(1),y(1))$ (with $\overline m\le t-r-1$ in transformations of types 3) and 4)) such that for $0\le i\le k$, either $\sigma_i=0$ or 
$$
\sigma_i=x_1(1)^{a_1^i}\cdots x_r(1)^{a_r^i}\overline u_i
$$
for some $a_j^i\in \NN$ and  unit $\overline u_i\in \CC\{\{x_1(1),\ldots,x_{t-1}(1)\}\}$. Then after a transformation of type 2) (which we incorporate into $(x,y)\rightarrow (x(1),y(1))$), we obtain (by Lemma \ref{Lemma21}) that $I\mathcal O_{X(1),e_{X(1)}}^{\rm an}$ is principal and generated by $x_1(1)^{a_1^i}\cdots x_r(1)^{a_r^i}$ for some $i$. Then we have an expression 
$$
g=x_1(1)^{a_1}\cdots x_r(1)^{a_r}F
$$
where $F\in \CC\{\{x_1(1),\ldots,x_t(1)\}\}$ and $h:=\mbox{ord }F(0,\ldots,0,x_t(1))<\infty$. If $h=0$ we have the conclusions of the lemma, so suppose that $h>0$. By Lemma \ref{Lemma9}, there exists a change of variables in $x_t(1)$ (inducing a transformation of type 3) with $\overline m=t-r$) such that $F$ has an expression
\begin{equation}\label{eq40}
F=\tau_0x_t(1)^h+\tau_2x_t(1)^{h-2}+\cdots+\tau_h
\end{equation}
with $\tau_0\in\CC\{\{ x_1(1),\ldots,x_t(1)\}\}$ a unit and $\tau_i\in \CC\{\{ x_1(1),\ldots,x_{t-1}(1)\}\}$ for $2\le i\le h$. By induction on $t$, we can perform a sequence of transformations of types 2) - 4) $(x(1),y(1))\rightarrow (x(2),y(2))$ (with $\overline m\le t-r-1$ in transformations of types 3) and 4)) such that for
$2\le i\le h$, 
$$
\tau_i=x_1(2)^{a_1^i}\cdots x_r(2)^{a_r^i}\overline\tau_i
$$
where $\overline \tau_i\in \CC\{\{ x_1(2),\ldots,x_{t-1}(2)\}\}$ is either zero or a unit series. We can assume by Lemma \ref{Lemma51} that $x_1(2),\ldots,x_r(2),x_t(1)$ are dependent. Now perform by Lemma \ref{Lemma23} a transformation of type 4)
$(x(2),y(2))\rightarrow (x(3),y(3))$ with $\overline m=t-r$ and substitute into (\ref{eq40}) to get an expression
$$
F=\tau_0x_1(3)^{b_1^0}\cdots x_r(3)^{b_r^0}(x_t(3)+\alpha)^h+\overline \tau_2x_1(3)^{b_1^2}\cdots x_r(3)^{b_r^2}(x_t(3)+\alpha)^{h-2}
+\cdots +\overline \tau_hx_1(3)^{b_1^h}\cdots x_r(3)^{b_r^h}
$$
with $0\ne \alpha\in \CC$. Now perform a transformation of type 2) (which we incorporate into $(x(2),y(2))\rightarrow (x(3),y(3)))$ to principalize the ideal 
$$
I=(x_1(3)^{b_1^i}\cdots x_r(3)^{b_r^i}\mid i=0\mbox{ or }\overline \tau_i\ne 0).
$$
We then have an expression
$$
g=x_1(3)^{\overline a_1}\cdots\overline x_r(3)^{\overline a_r}\overline F
$$
where $\mbox{ord }\overline F(0,\ldots,0,x_t(3))<h$. By induction on $h$, we  eventually reach the conclusions of the lemma for $g\in \CC\{\{ x_1,\ldots,x_t\}\}$. The lemma now follows from induction on $t$.
\end{proof}

\begin{Lemma}\label{Lemma2} Suppose that $(x,y)$ are prepared of type $(s,r,l)$ and $g\in \CC\{\{y_1,\ldots,y_{s+l}\}\}$. Then either there exists a sequence of transformations $(x,y)\rightarrow (x(1),y(1))$ such that $(x(1),y(1))$ are prepared of type $(s_1,r_1,l_1)$ with $(s_1,r_1,l_1)>(s,r,l)$ or there exists a sequence of transformations of the types 1) - 4) $(x,y)\rightarrow (x(1),y(1))$ such that 
$(x(1),y(1))$ are prepared of type $(s_1,r_1,l_1)$ with  $(s_1,r_1,l_1)=(s,r,l)$ and we have an expression
$$
g=y_1(1)^{d_1}\cdots y_s(1)^{d_s}u
$$
with $u\in \CC\{\{y_1(1),\ldots, y_{s+l}(1)\}\}$ a unit.
\end{Lemma}

\begin{proof} We will perform a sequence of transformations  which we may assume  do not lead to an increase in $(s,r,l)$.

Let $g$ have the expression (\ref{eq6}). Let $J$ be the ideal in $\mathcal O_{\tilde Y,e_{\tilde Y}}^{\rm an}$ defined by
$$
J=(h_{[\Lambda]}\mid [\Lambda] \in \ZZ^s/(\QQ^rC)\cap \ZZ^s).
$$
$J$ is generated by $h_{[\Lambda_1]},\ldots,h_{[\Lambda_t]}$ for some $[\Lambda_1],\ldots,[\Lambda_t]$.
After performing a transformation of type 2) $(x,y)\rightarrow (x(1),y(1))$ we obtain  expressions 
$$
\delta_{[\Lambda_i]}:= \frac{h_{[\Lambda_i]}}{y_1^{\lambda_1^i}\cdots y_s^{\lambda_s^i}}x_1^{w_1^i}\cdots x_r^{w_r^i}\in \CC\{\{x_1(1)^{\frac{1}{d}},\ldots,x_r(1)^{\frac{1}{d}},x_{r+1}(1),\ldots,x_{r+l}(1)\}\}.
$$
for $1\le i\le t$ of the form of  (\ref{eq64}) by Proposition \ref{Prop2}. We may choose the $w_1^{i},\ldots, w_r^i\in \NN$ so that 
$$
\frac{x_1^{w_1^i}\cdots x_r^{w_r^i}}{y_1^{\lambda_1^i}\cdots y_s^{\lambda_s^i}}\in \CC\{\{y_1,\ldots,y_s\}\}.
$$
 Let $\omega$ be a complex primitive $d$-th root of unity, and for $1\le j\le t$, let
$$
\epsilon_{[\Lambda_j]}=\prod_{i_1,\ldots,i_r=1}^d\delta_{[\Lambda_j]}(\omega^{i_1}x_1(1)^{\frac{1}{d}},\ldots,\omega^{i_r}x_r(1)^{\frac{1}{d}},x_{r+1}(1),\ldots,x_{r+l}(1))
\in \CC\{\{x_1(1),\ldots,x_{r+l}(1)\}\}.
$$
Let 
$$
f=\prod_{i=1}^t \epsilon_{[\Lambda_i]}.
$$
By Lemma \ref{Prop1}, there exists a sequence of transformations of types 2) - 4) $(x(1),y(1))\rightarrow (x(2),y(2))$ such that $f=x_1(2)^{m_1}\cdots x_r(2)^{m_r}u$
where $u\in\CC\{\{ x_1(2),\ldots,x_{r+l}(2)\}\}$ is a unit series. Thus each $\epsilon_{[\Lambda_i]}$ has such a form,  so
$$
\epsilon_{[\Lambda_i]}=x_1(2)^{m_1^i}\cdots x_r(2)^{m_r^i}u_i
$$
for $1\le i\le t$ where $u_i\in \CC\{\{x_1(2),\ldots,x_{r+l}(2)\}\}$ is a unit. 

Let $K$ be the quotient field of $R=\CC\{\{ y_1(2),\ldots,y_{s+l}(2)\}\}$. We have
$$
\chi_{[\Lambda_i]}:=\frac{\epsilon_{[\Lambda_i]}}{\delta_{[\Lambda_i]}}\in K
$$
 for $1\le i\le t$. We also have 
 $$
 \chi_{[\Lambda_i]}\in \CC\{\{x_1(2)^{\frac{1}{d}},\ldots,x_r(2)^{\frac{1}{d}},x_{r+1}(2),\ldots,x_{r+l}(2)\}\},
 $$
as we have only performed transformations of types 2) - 4). So $\chi_{[\Lambda_i]}$ is integral over $\CC\{\{x_1(2),\ldots,x_{r+l}(2)\}\}$ and thus $\chi_{[\Lambda_i]}$ is integral over $R$. Since $R$ is a regular local ring it is normal so $\chi_{[\Lambda_i]}\in R$. Thus $\delta_{[\Lambda_i]}$ divides $\epsilon_{[\Lambda_i]}$ in $R$ and so there are expressions
$$
\delta_{[\Lambda_i]}=y_1(2)^{e_1^i}\cdots y_s(2)^{e_s^i}v_i
$$
for $1\le i\le t$ where $v_i\in\CC\{\{y_1(2),\ldots,y_{s+l}(2)\}\}$ are unit series and thus 
$$
h_{[\Lambda_i]}=y_1(2)^{m_1^i}\cdots y_s(2)^{m_s^i}\overline u_i
$$
for $1\le i\le t$ where $\overline u_i\in \CC\{\{y_1(2),\ldots,y_{s+l}(2)\}\}$ are unit series. 
Now perform a transformation of type 1) to principalize the ideal
$J\mathcal O_{Y(2),e_{Y(2)}}^{\rm an}=(y_1(2)^{m_1^i}\cdots y_s(2)^{m_s^i}\mid 1\le i\le t)$. Then we have the desired conclusion for $g$ by (\ref{eq70}) in Lemma \ref{Lemma7}.
\end{proof}

\begin{Lemma}\label{Lemma3} Suppose that $(x,y)$ are prepared of type $(s,r,l)$ and $g\in \CC\{\{y_1,\ldots,y_{s+l}\}\}$. Then either there exists a sequence of transformations $(x,y)\rightarrow (x(1),y(1))$ such that $(x(1),y(1))$ are prepared of type $(s_1,r_1,l_1)$ with $(s_1,r_1,l_1)>(s,r,l)$ or there exists a sequence of transformations of the types 1) - 4) and 8) $(x,y)\rightarrow (x(1),y(1))$ such that 
$(x(1),y(1))$  are prepared of type $(s_1,r_1,l_1)$ with $(s_1,r_1,l_1)=(s,r,l)$ 
 and
either
$g$ is algebraic over $x_1(1),\ldots,x_{r+l}(1)$ or 
$$
g=P+y_1(1)^{d_1}\cdots y_s(1)^{d_s}
$$
with $P\in \CC\{\{y_1(1),\ldots,y_{s+l}(1)\}\}$ algebraic over $x_1(1),\ldots,x_{r+l}(1)$ and $y_1(1)^{d_1}\cdots y_s(1)^{d_s}$  not algebraic over $x_1(1),\ldots,x_r(1)$.
\end{Lemma}

\begin{proof}  We will perform a sequence of transformations which we may assume do not lead to an increase in $(s,r,l)$. Let $g$ have the expression (\ref{eq6}) and let $g'=g-h_{[0]}\in \CC\{\{y_1,\ldots,y_{s+l}\}\}$. By Lemma \ref{Lemma2}, there exists a sequence of transformations of type 1) - 4) $(x,y) \rightarrow (x(1),y(1))$ so that $g'=y_1(1)^{d_1}\cdots y_s(1)^{d_s}u$ with $u\in \CC\{\{y_1(1),\ldots,y_{s+l}(1)\}\}$ a unit. 
By Proposition \ref{Prop2}, after possibly performing another transformation of type 2), we also obtain that $h_{[0]}\in \CC\{\{x_1(1)^{\frac{1}{d}},\ldots,x_r(1)^{\frac{1}{d}},x_{r+1}(1),\ldots,x_{r+l}(1)\}\}$. 
Since $h_{[0]}\in \CC\{\{y_1(1),\ldots,y_{s+l}(1)\}\}$, by Lemma \ref{Lemma7}, we have that $h_{[0]}$ is algebraic over $x_1(1),\ldots,x_{r+l}(1)$. 
We also have by Lemma \ref{Lemma7} that 
$$
{\rm rank}\left(\begin{array}{ccc}
c_{11}(1)&\cdots&c_{1s}(1)\\
\vdots&&\vdots\\
c_{r1}(1)&\cdots&c_{rs}(1)\\
d_1&\cdots&d_s
\end{array}\right)=r+1.
$$
Thus there exist $e_1,\ldots,e_s\in \QQ$ such that 
$$
e_1c_{i1}(1)+\cdots+e_sc_{is}(1)=0\mbox{ for }1\le i\le r
$$
and
$$
e_1d_1+\cdots+e_sd_s=-1,
$$
and so, making a change of variables, replacing $y_i(1)$ with $y_i(1)u^{e_i}$ for $1\le i\le s$, we have a transformation of type 8) which gives the conclusions of the lemma.

\end{proof}

\begin{Proposition}\label{Theorem502} Suppose that $(x,y)$ are prepared of type $(s,r,l)$. Then either there exists a sequence of transformations $(x,y)\rightarrow (x(1),y(1))$ such that $(x(1),y(1))$ are prepared of type $(s_1,r_1,l_1)$ with $(s_1,r_1,l_1)>(s,r,l)$ or the induced homomorphism
$$
\alpha:\CC[[x_1,\ldots,x_{r+l},x_{r+l+1}]]\rightarrow \CC[[y_1,\ldots,y_n]]
$$
is an injection.
\end{Proposition}

\begin{proof} Set $z=\alpha(x_{r+l+1})$ and suppose that there exists a nonzero series $G\in \CC[[x_1,\ldots,x_{r+l+1}]]$ such that $\alpha(G)=0$. Expand $G$ as
$$
G=\sum_{i=0}^{\infty}a_i(x_1,\ldots,x_{r+l})x_{r+l+1}^i
$$
with $a_i(x_1,\ldots,x_{r+l})\in \CC[[x_1,\ldots,x_{r+l}]]$ for all $i$. We have
$\alpha(a_i)\in \CC[[y_1,\ldots,y_{s+l}]]$ for all $i$ and 
\begin{equation}\label{eq501}
0=\alpha(G)=\sum_i\alpha(a_i)z^i=0
\end{equation}
in $\CC[[y_1,\ldots,y_n]]$. 

 Let $A=\CC[[y_1,\ldots,y_{s+l}]]$ and $A[[t]]$ be a power series ring in one variable. Let
 $$
 f(t)=\sum\alpha(a_i)t^i\in A[[t]].
 $$
 $f(t)$ is nonzero since $\alpha(a_i)$ is nonzero whenever $a_i$ is nonzero.

Suppose that $z\not\in \CC[[y_1,\ldots,y_{s+l}]]$. We will derive a contradiction. Expand 
$$
z=\sum b_{i_{s+l+1},\ldots,i_n}y_{s+l+1}^{i_{s+l+1}}\cdots y_n^{i_n}
$$
 in $\CC[[y_1,\ldots,y_n]]$ with $b_{i_{s+l+1},\ldots,i_n}\in \CC[[y_1,\ldots,y_{s+l}]]$. Since $z$ is in the maximal ideal of $\CC[[y_1,\ldots,y_n]]$, we have that $b_{0,\ldots,0}$ is in the maximal ideal of $\CC[[y_1,\ldots,y_{s+l}]]$. Thus the map $g(t)\mapsto g(t+b_{0,\ldots,0})$ is an isomorphism of $A[[t]]$. 
 Let $\overline f(t)=f(t+b_{0,\ldots,0})$. We have that $\overline f(t)\ne 0$. Let $\overline z=z-b_{0,\ldots,0}$. We have that $\overline f(\overline z)=0$. Let $(j_{s+l+1},\ldots,j_n)$ be the minimum in the lex order of
 $$
 \{(i_{s+l+1},\ldots,i_n)\mid b_{i_{s+l+1},\ldots,i_n}\ne 0\mbox{ and }(i_{s+l+1},\ldots,i_n)\ne (0,\ldots,0)\}.
 $$
 Then $\overline f(\overline z)$ has a nonzero $\lambda(j_{s+l+1},\ldots,j_n)$ term, where $\lambda$ is the smallest positive exponent of $t$ such that $\overline f(t)$ has a nonzero $t^{\lambda}$ term. This contradiction shows that $z\in \CC[[y_1,\ldots,y_{s+l}]]$.
 Thus 
$$
z\in \CC[[y_1,\ldots,y_{s+l}]]\cap\CC\{\{y_1,\ldots,y_n\}\}=\CC\{\{y_1,\ldots,y_{s+l}\}\}.
$$
Suppose that $z$ is not algebraic over $x_1,\ldots,x_{r+l}$ (Definition \ref{Def500}). 
Let
$$
z=\sum_{[\Lambda]\in (\ZZ^s/(\QQ^rC)\cap \ZZ^s)}h_{[\Lambda]}
$$
be the decomposition of (\ref{eq6}). Then $z\ne h_{[0]}$ since we are assuming that $z$ is not algebraic over $x_1,\ldots,x_{r+l}$. Since $z$ is in the maximal ideal of $\CC[[y_1,\ldots,y_n]]$ we have that $h_{[0]}$ is in the maximal ideal of $\CC[[y_1,\ldots,y_{s+l}]]$. Thus the map $g(t)\mapsto g(t+h_{[0]})$ is an isomorphism of $A[[t]]$. Let $\tilde f(t)=f(t+h_{[0]})$. We have that $\tilde f(t)\ne 0$. Further, all coefficients of $\tilde f(t)$ are algebraic over $x_1,\ldots,x_{r+l}$.

Let $\tilde z=z-h_{[0]}\in\CC\{\{y_1,\ldots,y_{s+l}\}\}$. We have that $\tilde f(\tilde z)=0$. By Lemma \ref{Lemma2}, there either exists a sequence of transformations $(x,y)\rightarrow (x(1),y(1))$ such that  $(x(1),y(1))$ are prepared of type $(s_1,r_1,l_1)$ with $(s_1,r_1,l_1)>(s,r,l)$ or there exists a sequence of transformations $(x,y)\rightarrow (x(1),y(1))$ of types 1) - 4) such that we have an expression 
$$
\tilde z=y_1(1)^{d_1}\cdots y_s(1)^{d_s}u
$$
where $u\in \CC\{\{y_1(1),\ldots,y_{s+l}(1)\}\}$ is a unit. We have that
$$
\mbox{rank}\left(\begin{array}{ccc}
c_{11}(1)&\cdots&c_{1s}(1)\\
\vdots&&\vdots\\
c_{r1}(1)&\cdots&c_{rs}(1)\\
d_1&\cdots&d_s
\end{array}\right)
=r+1
$$
by Lemma \ref{Lemma7}. We now perform a transformation of type 8), replacing $y_i(1)$ with $y_i(1)u^{\lambda_i}$ for some $\lambda_i\in \QQ$ for $1\le i\le s$ to obtain that
$\tilde z=y_1(1)^{d_1}\cdots y_s(1)^{d_s}$. We have that $\tilde f(t)\in A_1[[t]]$ where $A_1=\CC\{\{y_1(1),\ldots,y_{s+l}(1)\}\}$ and all coefficients $e_i$  of 
$$
\tilde f(t)=\sum e_it^i
$$
 are algebraic over $x_1(1),\ldots,x_{r+l}(1)$ by Lemma \ref{Lemma1}. From the expansion
$$
0=\tilde f(\tilde z)=\sum_{i=0}^{\infty}e_i(y_1(1)^{d_1}\cdots y_s(1)^{d_s})^i
$$
we see that this is the expansion of type (\ref{eq6}) of $\tilde f(\tilde z)=0$, so that $e_i(y_1(1)^{d_1}\cdots y_s(1)^{d_s})^i=0$ for all $i$, which implies that $e_i=0$ for all $i$ so that $\tilde f(t)=0$, giving a contradiction,
 so $z$ is algebraic over $x_1,\ldots,x_{r+l}$. By Lemma \ref{Lemma11}, identifying $z$ with $x_{r+l+1}$ by the inclusion $\CC\{\{x_1,\ldots,x_n\}\}\subset \CC\{\{y_1,\ldots,y_n\}\}$, we have that
$x_{r+l+1}\in \CC\{\{x_1,\ldots,x_{r+l}\}\}$, a contradiction. Thus $\alpha$ is injective.
\end{proof}

\begin{Lemma}\label{Lemma4} Suppose that $(x,y)$ are prepared of type $(s,r,l)$ and $g\in \CC\{\{y_1,\ldots,y_{t}\}\}$ with $s+l\le t\le n$. Then either there exists a sequence of transformations $(x,y)\rightarrow (x(1),y(1))$ such that $(x(1),y(1))$ are prepared of type $(s_1,r_1,l_1)$ with $(s_1,r_1,l_1)>(s,r,l)$ or there exists a sequence of transformations of the types 1) - 6)  $(x,y)\rightarrow (x(1),y(1))$ (with $l<\overline m\le t-s$ in transformations of type 5) - 6))
such that $(x(1),y(1))$  are prepared of type $(s_1,r_1,l_1)$ with  $(s_1,r_1,l_1)=(s,r,l)$  and 
$$
g=y_1(1)^{d_1}\cdots y_s(1)^{d_s}u
$$
with $u\in \CC\{\{y_1(1),\ldots, y_{t}(1)\}\}$ a unit.
\end{Lemma}

\begin{proof} We will perform a sequence of transformations  which we may assume  do not lead to an increase in $(s,r,l)$.
The proof is by induction on $t$ with $s+l\le t\le n$, with $g\in \CC\{\{x_1,\ldots,x_t\}\}$. The case $t=s+l$ is proven in Lemma \ref{Lemma2}.
Thus we may assume that $t> s+l$. Write
$$
g=\sum \sigma_iy_t^i
$$
where $\sigma_i\in \CC\{\{ y_1,\ldots,y_{t-1}\}\}$.
Let $I=(\sigma_i\mid i\ge 0)$. There exist  $\sigma_{0},\ldots,\sigma_{k}$ which generate $I$. by induction, there exist a sequence of transformations of the types 1) - 6) $(x,y)\rightarrow (x(1),y(1))$ (with $l<\overline m\le t-1-s$ whenever a transformation of type 5) or 6) is performed) such that 
$$
\sigma_{j}=y_1(1)^{i_1(j)}\cdots y_s(1)^{i_s(j)}\overline u_{j}
$$
for $0\le j\le k$ where $\overline u_{j}\in\CC\{\{ y_1(1),\ldots,y_{t-1}(1)\}\}$ is a unit or zero. Now perform a transformation of type 1) (which we incorporate into $(x,y)\rightarrow (x(1),y(1))$) to make $I$ principal. Then we have an expression
$g=y_1(1)^{m_1}\cdots y_s(1)^{m_s}\overline g$ where $h=\mbox{ord}(\overline g(0,\ldots,0,y_t(1))<\infty$. If $h=0$ we are done. We will now proceed by induction on $h$.
 By Lemma \ref{Lemma9}, we can perform a transformation of type 5), replacing $y_t(1)$ with $y_t(1)-\Phi$ for an appropriate $\Phi\in \CC\{\{y_1(1),\ldots,y_{t-1}(1)\}\}$, to obtain  an expression 
\begin{equation}\label{eq8}
\overline g=\tau_0y_t(1)^h+\tau_1y_t(1)^{h-2}+\cdots+\tau_h
\end{equation}
with $\tau_0\in \CC\{\{y_1(1),\ldots,y_t(1)\}\}$ a unit series and $\tau_i\in \CC\{\{y_1(1),\ldots,y_{t-1}(1)\}\}$ for $1\le i\le h$.
By induction on $t$, we may construct a sequence of transformations of type 1) - 6) $(x(1),y(1))\rightarrow (x(2),y(2))$ (with $\overline m\le t-1-s$ whenever a transformation of type 5) or 6) is performed) such that for $2\le i\le h$, whenever $\tau_i$ is nonzero, it has an expression
$$
\tau_i=y_1(2)^{j_1^i}\cdots y_s(2)^{j_s^i}\overline u_i
$$ 
where $\overline u_i\in\CC\{\{ y_1(2),\ldots,y_{t-1}(2)\}\}$  is a unit series. Since $y_t(2)$ is dependent on $y_1(2),\ldots,y_s(2)$, there exists a transformation of  type 6) $(x(2),y(2))\rightarrow (x(3),y(3))$ with $\overline m=t$, which we perform. Substituting into (\ref{eq8}), we obtain
$$
\overline g=\tau_0 y_1(3)^{b_1^0}\cdots y_s(3)^{b_s^0}(y_t(3)+\alpha)^h+
y_1(3)^{b_1^2}\cdots y_s(3)^{b_s^2}\overline u_2(y_t(3)+\alpha)^{h-2}+\cdots
+y_1(3)^{b_1^h}\cdots y_s(3)^{b_s^h}\overline u_h
$$
(with $0\ne \alpha\in \CC$).
Now perform a transformation of type 1) (which we incorporate into $(x(2),y(2))\rightarrow (x(3),y(3)))$ to principalize the ideal
$$
J=(y_1(3)^{b_1^0}\cdots y_s(3)^{b_s^0},y_1(3)^{b_1^2}\cdots y_s(3)^{b_s^2}\overline u_2,\ldots,y_1(3)^{b_1^h}\cdots y_s(3)^{b_s^h}\overline u_h),
$$
giving us that $\overline g=y_1(3)^{d_1}\cdots y_s(3)^{d_s}\tilde g$ with $\mbox{ord}(\tilde g(0,\ldots,y_t(3)))<h$. By induction on $h$, we obtain the conclusions of the lemma.

\end{proof}

\begin{Lemma}\label{Lemma5} Suppose that $(x,y)$ are prepared of type $(s,r,l)$ and 
$$
g\in \CC\{\{y_1,\ldots,y_{t}\}\}\setminus \CC\{\{y_1,\ldots,y_{s+l}\}\}
$$
 with $s+l< t\le n$. Then either there exists a sequence of transformations $(x,y)\rightarrow (x(1),y(1))$ such that $(x(1),y(1))$ are prepared of type $(s_1,r_1,l_1)$ with $(s_1,r_1,l_1)>(s,r,l)$ or there exists a sequence of transformations of the types 1) - 7) $(x,y)\rightarrow (x(1),y(1))$ (with $\overline m\le t-s$ in transformations of types 5) - 7)) such that $(x(1),y(1))$  are prepared of type $(s_1,r_1,l_1)$ with  $(s_1,r_1,l_1)=(s,r,l)$  and 
$$
g=P+y_1(1)^{d_1}\cdots y_s(1)^{d_s}y_t
$$
with $P\in\CC\{\{y_1(1),\ldots,y_{s+l}(1)\}\}$. 
\end{Lemma}

\begin{proof} We will perform a sequence of transformations  which we may assume  do not lead to an increase in $(s,r,l)$.
Write $g=\sum_{i\ge 0} \sigma_iy_t^i$ with $\sigma_i\in \CC\{\{y_1,\ldots,y_t\}\}$. Let $I$ be the ideal $I=(\sigma_i\mid i>0)$. 
Suppose that $I$ is generated by $\sigma_1,\ldots,\sigma_k$. 
By Lemma \ref{Lemma4}, there exist a sequence of transformations of types 1) - 6) $(x,y)\rightarrow (x(1),y(1))$ (with $\overline m\le t-s-1$ if a tranformation of type 5) or 6) is performed) such that for $1\le j\le k$,
$$
\sigma_j=y_1(1)^{i_1^j}\cdots y_s(1)^{i_s^j}u_j
$$
with $u_j\in\CC\{\{ y_1(1),\ldots,y_{t-1}(1)\}\}$ a unit (or zero). By induction on $t$ in Lemma \ref{Lemma5}, there exists a sequence of transformations of types 1) - 7) $(x(1),y(1))\rightarrow (x(2),y(2))$ (with $\overline m\le t-s-1$ if a transformation of type 5), 6)  or 7) is performed) such that 
\begin{equation}\label{eq10}
\sigma_0=P_0+y_1(2)^{a_1}\cdots y_s(2)^{a_s}y_{t-1}(2)
\end{equation}
or 
\begin{equation}\label{eq11}
\sigma_0=P_0
\end{equation}
with $P_0\in \CC\{\{y_1(2),\ldots,y_{s+l}(2)\}\}$. Case (\ref{eq10}) can only occur if $t>s+l+1$.

Let $J$ be the ideal $I\mathcal O_{Y(2),e_{Y(2)}}^{\rm an}+(y_1(2)^{a_1}\cdots y_s(2)^{a_s})$ if (\ref{eq10}) holds and $J=I\mathcal O_{Y(2),e_{Y(2)}}^{\rm an}$ if (\ref{eq11}) holds. $J$ is generated by monomials in $y_1(2),\ldots,y_s(2)$. There exists a transformation of type 1) $(x(2),y(2))\rightarrow (x(3),y(3))$ such that $J\mathcal O_{X(3),e_{X(3)}}^{\rm an}$ is principal by Lemma \ref{Lemma21}, so
$$
g=P_0+\sum_{i>0}\sigma_iy_t(2)^i+y_1(2)^{a_1}\cdots y_s(2)^{a_s}y_{t-1}(2)\overline u
=P_0+y_1(3)^{d_1}\cdots y_s(3)^{d_s}\overline g
$$
where $\overline u$ is zero or 1, 
 $P_0\in \CC\{\{y_1(3),\ldots,y_{s+l}(3)\}\}$ and $\overline g\in \CC\{\{y_1(3),\ldots,y_t(3)\}\}$ is not divisible by $y_1(3),\ldots,y_s(3)$. If $\mbox{ord }\overline g(0,\ldots,0,y_{t-1}(3),0)=1$ we set $y_t(3)=\overline g$ and $y_{t-1}(3)=y_t(3)$ (a composition of transformations of type 7) and 5)) to get the conclusions of Lemma \ref{Lemma5}. Otherwise, we have  
$$
0<\mbox{ord }\overline g(0,\ldots,0,y_t(3))<\infty.
$$

Now suppose that 
\begin{equation}\label{eq18}
g=P+y_1(3)^{d_1}\cdots y_s(3)^{d_s}F
\end{equation}
 where $P\in \CC\{\{y_1(3),\ldots,y_{s+l}(3)\}\}$, $F\in \CC\{\{y_1(3),\ldots,y_t(3)\}\}$ is such that the power series expansion of $y_1(3)^{d_1}\cdots y_s(3)^{d_s}F$ has no monomials in $y_1(3),\cdots,y_{s+l}(3)$; that is, 
 $$
 F(y_1(3),\ldots,y_{s+l}(3),0,\ldots,0)=0,
 $$
 $y_i(3)\not\,\mid F$ for $1\le i\le s$ and 
 $$
 0<h:=\mbox{ord }F(0,\ldots,0,y_t(3))<\infty.
 $$
  If $h=1$, we can set $y_t(3)=F$ (a transformation of type 5)) to get the conclusions of Lemma \ref{Lemma5} for $g$.

Suppose that $h>1$. 
 By Lemma \ref{Lemma9}, we can make a change of variables, replacing $y_t(3)$ with  $y_t(3)-\Phi$ for an appropriate
 $\Phi\in \CC\{\{y_1(3),\ldots,y_{t-1}(3)\}\}$ (a transformation of type 5)) to get  an expression 
\begin{equation}\label{eq12}
F=\tau_0y_t(3)^h+\tau_2y_t(3)^{h-2}+\cdots+\tau_h
\end{equation}
where $\tau_0\in \CC\{\{y_1(3),\ldots,y_t(3)\}\}$ is a unit and $\tau_i\in \CC\{\{y_1(3),\ldots,y_{t-1}(3)\}\}$ for $2\le i\le h$.
By  Lemma \ref{Lemma4}, there exists a sequence of transformations of types 1) - 6) $(x(3),y(3))\rightarrow (x(4),y(4))$ (with $\overline m<t-s$ for transformations of types 5) - 6)) such that for $2\le i\le h-1$,
$$
\tau_i=y_1(4)^{j_1^i}\cdots y_s(4)^{j_s^i}\overline u_i
$$
with  $\overline u_i\in\CC\{\{ y_1(4),\ldots,y_{t-1}(4)\}\}$ either   a unit or zero.
By induction on $t$ in Lemma \ref{Lemma5}, there exists a sequence of transformations of types 1) - 7) $(x(4),y(4))\rightarrow (x(5),y(5))$ (with $\overline m<t-s$ for transformations of types 5) - 7)) such that we further have that 
\begin{equation}\label{eq13}
\tau_h=P_0+y_1(5)^{c_1}\cdots y_s(5)^{c_s}y_{t-1}\overline u
\end{equation}
where $\overline u$ is zero or 1 and $P_0\in \CC\{\{y_1(5),\ldots,y_{s+l}(5)\}\}$. Since $y_t(5)$ is dependent on $y_1(5),\ldots,y_s(5)$, there exists a transformation of  type 6) $(x(5),y(5))\rightarrow (x(6),y(6))$ with $\overline m=t-s$.
Perform it and substitute into (\ref{eq12}) to get 
$$
\begin{array}{lll}
F&=&\tau_0y_1(6)^{b_1^0}\cdots y_s(6)^{b_s^0}(y_t(6)+\alpha)^h+y_1(6)^{b_1^2}\cdots y_s(6)^{b_s^2}\overline u_2(y_t(6)+\alpha)^{h-2}
+\cdots\\
&&+y_1(6)^{d_1}\cdots y_s(6)^{d_s}y_{t-1}(6)\overline u + P_0
\end{array}
$$
Now perform a transformation of type 1) $(x(6),y(6))\rightarrow (x(7),y(7))$ to principalize the ideal
$$
K=(y_1(6)^{b_1^0}\cdots y_s(6)^{b_s^0})+(y_1(6)^{b_1^i}\cdots y_s(6)^{b_s^i}\mid \overline u_i\ne 0)+(\overline uy_1(1)^{d_1}\cdots y_s(1)^{d_s}).
$$
We obtain an expression
$$
g=P_1+y_1(7)^{e_1}\cdots y_s(7)^{e_s}\overline F
$$
where 
$$
P_1=P+y_1(3)^{d_1}\cdots y_s(3)^{d_s}F(y_1(7),\ldots,y_{s+l}(7),0,\ldots,0)\in\CC\{\{y_1(7),\ldots,y_{s+l}(7)\}\}
$$
and
$$
y_1(7)^{e_1}\cdots y_s(7)^{e_s}\overline F=y_1(3)^{d_1}\cdots y_s(3)^{d_s}(F-F(y_1(7),\ldots,y_s(7),0,\ldots,0))
$$
is such that $y_i(7)\not\,\mid \overline F$ for $1\le i\le s$.
We either have $\mbox{ord } \overline F(0,\ldots,0,y_{t-1}(7),0)=1$ or $1\le \mbox{ord } \overline F(0,\ldots,0,y_t(7))<h$. In the first case, set $y_t(7)=F$ and $y_{t-1}(7)=y_t(7)$ (a composition of transformations of type 7) and  5)) to get the conclusions of Lemma \ref{Lemma5}. Otherwise we have a reduction in $h$ in (\ref{eq18}). By induction in $h$ we will eventually get the conclusions of Lemma \ref{Lemma5}.
\end{proof}

\begin{Proposition}\label{Theorem1}
Suppose that $(x,y)$ are prepared of type $(s,r,l)$  with $r+l<m$. Then either there exists a sequence of transformations $(x,y)\rightarrow (x(1),y(1))$ such that $(x(1),y(1))$ are prepared of type $(s_1,r_1,l_1)$ with $(s_1,r_1,l_1)>(s,r,l)$ or there exists a sequence of transformations of  types  1) - 8) $(x,y)\rightarrow (x(1),y(1))$ such that $(x(1),y(1))$  are prepared of type $(s_1,r_1,l_1)$ with  $(s_1,r_1,l_1)=(s,r,l)$  and we have an expression
\begin{equation}\label{eq15}
x_{r+l+1}(1)=P+y_1(1)^{d_1}\cdots y_s(1)^{d_s}
\end{equation}
with $P \in \CC\{\{y_1(1),\ldots,y_{s+l}(1)\}\}$ algebraic over $x_1(1),\ldots,x_{r+l}(1)$ and 
$y_1(1)^{d_1}\cdots y_s(1)^{d_s}$ not algebraic over $x_1(1),\ldots,x_r(1)$
or we have an expression
\begin{equation}\label{eq16}
x_{r+l+1}(1)=P+y_1(1)^{d_1}\cdots y_s(1)^{d_s}y_{s+l+1}(1)
\end{equation}
with $P\in \CC\{\{y_1(1),\ldots,y_{s+l}(1)\}\}$ algebraic over $x_1(1),\ldots,x_{r+l}(1)$.
\end{Proposition}

\begin{proof} 
We will construct a sequence of transformations  such that either we obtain an increase in $(s,r,l)$, or we obtain the conclusions of Proposition \ref{Theorem1}. We may thus assume that all transformations  in the course of our proof do not give an increase in $(s,r,l)$.

We have that $x_{r+l+1}$ is not algebraic over $x_1,\ldots,x_{r+l}$ by Lemma \ref{Lemma11}.

 First suppose that $x_{r+l+1}\in \CC\{\{y_1,\ldots,y_{s+l}\}\}$. Then there exists a sequence of transformations of types 1) - 4) and 8) such that the conclusions of Lemma \ref{Lemma3} hold, giving an expression (\ref{eq15}) of the conclusions of Proposition \ref{Theorem1}, since $x_{r+l+1}$ is not algebraic over $x_1(1),\ldots,x_{r+l}(1)$ by Lemma \ref{Lemma7}.
 
 Now suppose that $x_{r+l+1}\not\in \CC\{\{y_1,\ldots,y_{s+l}\}\}$. Then by Lemma \ref{Lemma5}, there exists a sequence of transformations of types 1) - 7) $(x,y)\rightarrow (x(1),y(1))$  such that we have an expression
 \begin{equation}\label{eq17}
 x_{r+l+1}(1)=\tilde P+y_1(1)^{a_1}\cdots y_s(1)^{a_s}y_{s+l+1}(1)
 \end{equation}
 with $\tilde P\in \CC\{\{y_1(1),\ldots,y_{s+l}(1)\}\}$.
 Then by Lemma \ref{Lemma3}, there exists a sequence of transformations 1) - 4) and 8) $(x(1),y(1))\rightarrow (x(2),y(2))$  such that we have an expression (\ref{eq17}) with
 $$
 \tilde P=P'+y_1(2)^{b_1}\cdots y_s(2)^{b_s}\overline u
 $$
 where $P'$ is algebraic over $x_1(2),\ldots,x_{r+l}(2)$ and $y_1(2)^{b_1}\cdots y_s(2)^{b_s}$ is not algebraic over $x_1(2),\ldots,x_r(2)$ and $\overline u$ is 0 or 1. If $\overline u=0$ we have achieved the conclusions of (\ref{eq16}) of Proposition \ref{Theorem1}, so assume that $\overline u=1$. Now (by Lemma \ref{Lemma21}) perform a transformation of type 1) $(x(2),y(2))\rightarrow (x(3),y(3))$ to principalize the ideal
 $$
 L=(y_1(2)^{b_1}\cdots y_s(2)^{b_s}, y_1(1)^{a_1}\cdots y_s(1)^{a_s}).
 $$
 If $y_1(2)^{b_1}\cdots y_s(2)^{b_s}$ divides $y_1(1)^{a_1}\cdots y_s(1)^{a_s}$ (in $\mathcal O_{X(3),e_{X(3)}}^{\rm an}$), since we have the condition that $y_1(2)^{b_1}\cdots y_s(2)^{b_s}$ is not algebraic over $x_1(3),\ldots,x_s(3)$ from Lemma \ref{Lemma7}, we can change variables, multiplying the $y_i$ by units for $1\le i\le s$ to get an expression (\ref{eq15}) of the conclusions of Proposition \ref{Theorem1} (a transformation of type 8)). If $y_1(2)^{b_1}\cdots y_s(2)^{b_s}$ does not divide $y_1(1)^{a_1}\cdots y_s(1)^{a_s}$ in $\mathcal O_{X(3),e_{X(3)}}^{\rm an}$ (so that 
 $y_1(1)^{a_1}\cdots y_s(1)^{a_s}$ properly divides $y_1(2)^{b_1}\cdots y_s(2)^{b_s}$ in $\mathcal O_{X(3),e_{X(3)}}^{\rm an}$) we have an expression
 $$
 x_{r+l+1}(3)=P+y_1(3)^{\overline a_1}\cdots y_s(3)^{\overline a_s}F
 $$
 with $F\in \CC\{\{y_1(3),\ldots,y_s(3),y_{s+l+1}(3)\}\}$
 such that  $\mbox{ord }F(0,\ldots,0,y_{s+l+1}(3))=1$. Replacing $y_{s+l+1}(3)$ with $F$ (a transformation of type 5)) we get an expression of the form (\ref{eq16}) of the conclusions of Proposition \ref{Theorem1}.

\end{proof}

\begin{Proposition}\label{Theorem2}
Suppose that $(x,y)$ are prepared of type $(s,r,l)$  with $r+l<m$. Then there exists a sequence of transformations of  types  1) - 10) $(x,y)\rightarrow (x(1),y(1))$ such that $(x(1),y(1))$  are prepared of type $(s_1,r_1,l_1)$ with $(s_1,r_1,l_1)>(s,r,l)$ 
\end{Proposition}

\begin{proof}  We may assume that all transformations of type 1) - 10) in the course of our proof do not give an increase in $(s,r,l)$; otherwise we have obtained the conclusions of the theorem and we can terminate our algorithms. By Proposition \ref{Theorem1},
there exists a sequence of transformations of types 1) - 8) $(x,y)\rightarrow (x(0),y(0))$ such that we have an expression (for $i=0$)
\begin{equation}\label{eq15*}
x_{r+l+1}(i)=P+y_1(i)^{d_1}\cdots y_s(i)^{d_s}
\end{equation}
with $P$ algebraic over $x_1(i),\ldots,x_{r+l}(i)$ and $y_1(i)^{d_1}\cdots y_s(i)^{d_s}$ not algebraic over $x_1(i),\ldots,x_r(i)$
or we have an expression
\begin{equation}\label{eq16*}
x_{r+l+1}(i)=P+y_1(i)^{d_1}\cdots y_s(i)^{d_s}y_{s+l+1}(i)
\end{equation}
with $P$ algebraic over $x_1(i),\ldots,x_{r+l}(i)$.

We will perform  sequences of transformations $(x,y)\rightarrow (x(i),y(i))$ in the course of this proof which preserve the respective expressions (\ref{eq15*}) or (\ref{eq16*}).

We will now construct a function $g$ (in equation (\ref{eq21})) using transformations which preserve the respective form (\ref{eq15*}) or (\ref{eq16*}). The function $g$ (or its strict transform) will play a major role in the proof.

The decomposition (\ref{eq6}) of $P$ is $P=h_{[0]}$ since $P$ is algebraic over $x_1(1),\ldots,x_{r+l}(1)$. There exists a transformation of type 2) $(x(0),y(0))\rightarrow (x(1),y(1))$ such that 
$$
P\in \CC\{\{x_1(1)^{\frac{1}{d}},\ldots,x_r(1)^{\frac{1}{d}},x_{r+1}(1),\ldots,x_{r+l}(1)\}\}
$$
 for some $d$ by Proposition \ref{Prop2}. 
Let $\omega$ be a primitive $d$-th root of unity in $\CC$. Let
$$
S_{i_1,\ldots,i_r}=P(\omega^{i_1}x_1(1)^{\frac{1}{d}},\ldots,\omega^{i_r}x_r(1)^{\frac{1}{d}},x_{r+1}(1),\ldots,x_{r+l}(1))
$$
for $1\le i_1,\ldots,i_r\le d$.  We have that
$$
S_{i_1,\ldots,i_r}\in \CC\{\{y_1(1)^{\frac{1}{d}},\ldots,y_s(1)^{\frac{1}{d}},y_{s+1}(1),\ldots,y_{s+l}(1)\}\}
$$ 
for all $i_1,\ldots,i_r$ since
$$
x_i(1)^{\frac{1}{d}}=\prod_{j=1}^s(y_j(1)^{\frac{1}{d}})^{c_{ij}(1)}\mbox{ for }1\le i\le r.
$$
Since $P\in \CC\{\{y_1(1),\ldots,y_{s+l}(1)\}\}$, we have that $S_{i_1,\ldots,i_r}\in \CC\{\{y_1(1),\ldots,y_{s+l}(1)\}$ for all $i_1,\ldots,i_r$. Further, $S_{i_1,\ldots,i_r}$ is algebraic over $x_1(1),\ldots,x_{r+l}(1)$ for all $i_1,\ldots,i_r$ since $P$ is.
Let
$$
R=\prod_{i_1,\ldots,i_r=1}^dS_{i_1,\ldots,i_r}\in \CC\{\{x_1(1),\ldots,x_{r+l}(1)\}\}.
$$
By Lemma \ref{Prop1}, there exists a sequence of transformations of types 2) - 4) $(x(1),y(1))\rightarrow (x(2),y(2))$ such that 
$$
R=x_1(2)^{m_1}\cdots x_r(2)^{m_r}u
$$
where $u\in \CC\{\{ x_1(2),\ldots,x_{r+l}(2)\}\}$ is a unit. Now $P$ divides $R$ in $\CC\{\{y_1(2),\ldots,y_{r+l}(2)\}\}$, so 
we have that 
\begin{equation}\label{eq20}
P=y_1(2)^{m_1}\cdots y_s(2)^{m_s}\tilde u
\end{equation}
where $\tilde u\in \CC\{\{y_1(2),\ldots,y_{s+l}(2)\}\}$ is a unit and by Lemma \ref{Lemma7} and since $P$ is algebraic over $x_1(2),\ldots,x_{r+l}(2)$, we have that $y_1(2)^{m_1}\cdots y_s(2)^{m_s}$ is algebraic over $x_1(2),\ldots,x_r(2)$.
Set
\begin{equation}\label{eq21}
g=\prod_{i_1,\ldots,i_r=1}^d(x_{r+l+1}(2)-S_{i_1,\ldots,i_r})\in \CC\{\{x_1(2),\ldots,x_{r+l+1}(2)\}\}.
\end{equation}
Let
\begin{equation}\label{eq29}
t=\mbox{ord }g(0,\ldots,0,x_{r+l+1}(2)).
\end{equation}
We have that $0<t\le d^r$.

The proof now proceedes by induction on $t$. We will make a series of transformations which will either give an increase in $(s,r,l)$ (establishing the proposition) or preserve the respective form (\ref{eq15*}) or (\ref{eq16*}) with a reduction in $t$ (which will remain positive) in the strict transform of $g$.  We first perform transformations which preserve the respective forms (\ref{eq15*}) or (\ref{eq16*}) and preserve $t=\mbox{ord }g(0,\ldots,0,x_{r+l+1}(i))$ which put $g$ into the  good form of  equation (\ref{eq23}) (in case (\ref{eq15*})) or in the good form of equation (\ref{eq24}) (in case (\ref{eq16*})).

 Set 
$$
Q_{i_1,\ldots,i_r}=P-S_{i_1,\ldots,i_r}
$$
which are algebraic over $x_1(2),\ldots,x_{r+l}(2)$.
By the argument leading to (\ref{eq20}), we can construct a sequence of transformations of types 2) - 4) $(x(2),y(2))\rightarrow (x(3),y(3))$ which preserve the expressions (\ref{eq20}), (\ref{eq21}), $t$ in equation (\ref{eq29}) and the expression (\ref{eq15*}) or (\ref{eq16*}) (in the variables $x(3)$ and $y(3)$) such that for all $I=(i_1,\ldots,i_r)$, 
\begin{equation}\label{eq22}
Q_I=y_1(3)^{n_1^I}\cdots y_s(3)^{n_s^I}u_I
\end{equation}
where $u_I\in \CC\{\{y_1(3),\ldots,y_{s+l}(3)\}\}$ are units and $y_1(3)^{n_1^I}\cdots y_s(3)^{n_s^I}$ are algebraic over $x_1(3),\ldots,x_r(3)$.
After a transformation of type 1) $(x(3),y(3))\rightarrow (x(4),y(4))$, we can principalize the ideals $(y_1(3)^{n_1^I}\cdots y_s(3)^{n_s^I},y_1(0)^{d_1}\cdots y_s(0)^{d_s})$ for all $I$ (by Lemma \ref{Lemma21}), giving us
the  possibilities 
$$
x_{r+l+1}(4)-S_{i_1,\ldots,i_r}=y_1(3)^{n_1^I}\cdots y_s(3)^{n_s^I}\overline u_I
$$
where $\overline u_I\in \CC\{\{y_1(4),\ldots,y_{s+l}(4)\}\}$ is a unit and $y_1(3)^{n_1^I}\cdots y_s(3)^{n_s^I}$ is algebraic over $x_1(4),\ldots,x_r(4)$
or
$$
x_{r+l+1}(4)-S_{i_1,\ldots,i_r}=y_1(0)^{d_1}\cdots y_s(0)^{d_s}\overline u_I
$$
where $\overline u_I\in \CC\{\{y_1(4),\ldots,y_{s+l}(4)\}\}$ is a unit and $y_1(0)^{d_1}\cdots y_s(0)^{d_s}$ is  not algebraic over $x_1(4),\ldots,x_r(4)$ if (\ref{eq15*}) holds
and giving us the  possibilities 
$$
x_{r+l+1}(4)-S_{i_1,\ldots,i_r}=y_1(3)^{n_1^I}\cdots y_s(3)^{n_s^I}G^I
$$
where $G^I\in \CC\{\{y_1(4),\ldots,y_{s+l+1}(4)\}\}$ is a unit and $y_1(3)^{n_1^I}\cdots y_s(3)^{n_s^I}$ is algebraic over $x_1(4),\ldots,x_r(4)$
or
$$
x_{r+l+1}(4)-S_{i_1,\ldots,i_r}=y_1(4)^{m_1^I}\cdots y_s(4)^{m_s^I}G^I
$$
where $G^I\in \CC\{\{y_1(4),\ldots,y_{s+l+1}(4)\}\}$ satisfies $\mbox{ord }G^I(0,\ldots,0,y_{r+l+1}(4))=1$ if (\ref{eq16*}) holds.
We have that
$$
x_{r+l+1}(4)-P=y_1(0)^{d_1}\cdots y_s(0)^{d_s}
$$
in case (\ref{eq15*}) and
$$
x_{r+l+1}(4)-P=y_1(0)^{d_1}\cdots y_s(0)^{d_s}y_{s+l+1}(4)
$$
in case (\ref{eq16*}). We thus have
\begin{equation}\label{eq23}
g=y_1(4)^{m_1}\cdots y_s(4)^{m_s}u
\end{equation}
where $u\in \CC\{\{y_1(4),\ldots, y_{s+l}(4)\}\}$ is a unit and $y_1(4)^{m_1}\cdots y_s(4)^{m_s}$ is not algebraic over $x_1(4),\ldots,x_r(4)$
in case (\ref{eq15*}) and 
\begin{equation}\label{eq24}
g=y_1(4)^{m_1}\cdots y_s(4)^{m_s}y_{s+l+1}(4)\prod_{I\ne (d,\ldots,d)} G^I
\end{equation}
where for all $I$, $G^I\in \CC\{\{y_1(4),\ldots, y_{s+l+1}(4)\}\}$ satisfies $\mbox{ord }G^I(0,\ldots,0,y_{s+l+1}(4))=1$ or 0 in case (\ref{eq16*}).

We now consider a special case, when $g$ has an expression of the form (\ref{eq25}) below, and show that after a few transformations we obtain the conclusions of the proposition.

Suppose that there exists $\Phi\in \CC\{\{x_1(4),\ldots,x_{r+l}(4)\}\}$ such that 
\begin{equation}\label{eq25}
g=\tilde u(x_{r+l+1}(4)-\Phi)^{\lambda}
\end{equation}
where $\lambda\in \ZZ_{>0}$ and $\tilde u\in \CC\{\{x_1(4),\ldots,x_{r+l+1}(4)\}\}$ is a unit series. 
Setting $P'=P-\Phi$, we have an expression
$$
x_{r+l+1}(4)-\Phi=P'+Q
$$
where $Q:=x_{r+l+1}(4)-P$ has the expression
$$
Q=\left\{
\begin{array}{ll}
y_1(0)^{d_1}\cdots y_s(0)^{d_s}\mbox{ of (\ref{eq15*}) or}\\
y_1(0)^{d_1}\cdots y_s(0)^{d_s}y_{s+l+1}(4)\mbox{ of (\ref{eq16*})}
\end{array}\right.
$$
and $P'$ is algebraic over $x_1(4),\ldots,x_{r+l}(4)$.
By Lemma \ref{Lemma2}, there exists a sequence of transformations of  types 1) - 4) $(x(4),y(4))\rightarrow (x(5),y(5))$ such that
$$
P'=y_1(5)^{a_1}\cdots y_s(5)^{a_s}u'
$$
where $u'\in\CC\{\{y_1(5),\ldots,y_{s+l}(5)\}\}$ is a unit series. We have that  $y_1(5)^{a_1}\cdots y_s(5)^{a_s}$ is algebraic over $x_1(5),\ldots,x_r(5)$ by Lemma \ref{Lemma7}. By Lemma \ref{Lemma21},
after a transformation of type 1), $(x(5),y(5))\rightarrow (x(6),y(6))$, we have  that in the case when (\ref{eq15*}) holds,
\begin{equation}\label{eq26}
x_{r+l+1}(6)-\Phi=y_1(6)^{n_1}\cdots y_s(6)^{n_s}\hat u
\end{equation}
with $\hat u\in \CC\{\{y_1(6),\ldots,y_{s+l+1}(6)\}\}$ a unit 
and in the case when (\ref{eq16*}) holds, we have
\begin{equation}\label{eq27}
x_{r+l+1}(6)-\Phi=
\left\{\begin{array}{ll}
y_1(6)^{n_1}\cdots y_s(6)^{n_s}\hat u&\mbox{ with $\hat u\in \CC\{\{y_1(6),\ldots,y_{s+l+1}(6)\}\}$ a unit}\\
&\mbox{ and $y_1(6)^{n_1}\cdots y_s(6)^{n_s}$ algebraic over }\\
&\mbox{ $x_1(6),\ldots,x_{r+l}(6)$, or}\\
y_1(6)^{n_1}\cdots y_s(6)^{n_s}F&\mbox{ with $F\in \CC\{\{y_1(6),\ldots,y_{s+l+1}(6)\}\}$}\\
&\mbox{ such that }\mbox{ord }F(0,\ldots,0,y_{s+l+1}(6)) = 1.
\end{array}\right.
\end{equation}

If Case (\ref{eq15*}) holds, we have from comparison of the equations (\ref{eq26}), (\ref{eq23}) and (\ref{eq25}) that
$$
y_1(4)^{m_1}\cdots y_s(4)^{m_s}=(y_1(6)^{n_1}\cdots y_s(6)^{n_s})^{\lambda}
$$
where
$y_1(4)^{m_1}\cdots y_s(4)^{m_s}$ is not algebraic over $x_1(6),\ldots,x_{r+l}(6)$. Thus $y_1(6)^{n_1}\cdots y_s(6)^{n_s}$ is also not algebraic over $x_1(6),\ldots,x_{r+l}(6)$. Making a change of variables replacing $x_{r+l+1}(6)$ with $x_{r+l+1}(6)-\Phi$ and $y_1(6),\ldots,y_s(6)$ with their products by appropriate units in $\CC\{\{y_1(6),\ldots,y_{s+l}(6)\}\}$ (transformations of types 10) and 8)), we get 
$$
x_{r+l+1}(6)=y_1(6)^{n_1}\cdots y_s(6)^{n_s}
$$
with $y_1(6)^{n_1}\cdots y_s(6)^{n_s}$  not algebraic over $x_1(6),\ldots,x_r(6)$
obtaining an increase in $r$ (and $(s,r,l)$), and so we have achieved the conclusions of Proposition \ref{Theorem2}.

If case (\ref{eq16*}) holds, then (\ref{eq24}), (\ref{eq27}) and (\ref{eq25}) hold, so we have that
$$
x_{r+l+1}(6)-\Phi=y_1(6)^{n_1}\cdots y_s(6)^{n_s}F
$$
where $F\in \CC\{\{y_1(6),\ldots,y_{s+l+1}(6)\}\}$ satisfies $\mbox{ord }F(0,\ldots,0,y_{s+l+1}(6))=1$. Then making changes of variables, replacing $y_{n+l+1}(6)$ with $F$ and
$x_{r+1+1}(6)$ with $x_{r+l+1}(6)-\Phi$ (transformations of types 5) and 10)), we have
$$
x_{r+l+1}(6)=y_1(6)^{n_1}\cdots y_s(6)^{n_s}y_{s+l+1}(6).
$$
If $y_1(6),\ldots,y_s(6),y_{s+l+1}(6)$ are independent, we have an increase in $s$ (and $(s,r,l)$). 
Otherwise, we perform a SGMT in $y_1(6),\ldots,y_s(6),y_{s+l+1}(6)$ giving a transformation of type 6) $(x(6),y(6))\rightarrow (x(7),y(7))$ such that
$$
x_{r+l+1}(7)=y_1(7)^{b_1}\cdots y_s(7)^{b_s}(y_{s+l+1}(7)+\alpha)
$$
for some $0\ne\alpha\in \CC$. If $y_1(7)^{b_1}\cdots y_s(7)^{b_s}$ is not algebraic over $x_1(7),\ldots,x_{r+l}(7)$, then we can make a change of variables in $y_1(7),\ldots,y_s(7)$, (a transformation of type 8) $(x(7),y(7))\rightarrow (x(8),y(8)))$, giving an expression
$$
x_{r+l+1}(8)=y_1(8)^{b_1}\cdots y_s(8)^{b_s},
$$
thus giving an increase in $r$ (and $(s,r,l)$).  If $y_1(7)^{b_1}\cdots y_s(7)^{b_s}$ is algebraic over $x_1(7),\ldots,x_{r+l}(7)$, then
$\nu_e(x_{r+l+1}(7))$ is rationally dependent on $\nu_e(x_1(7)),\ldots,\nu_e(x_{r+l}(7))$, and so  
$$
x_1(7),\ldots,x_{r+l}(7), x_{r+l+1}(7)
$$
 are dependent  by Lemma \ref{Lemma25}. Thus by Lemma \ref{Lemma6}, there exists a SGMT $(x(7))\rightarrow (x(8))$ defined by
$$
x_i(7)=\prod_{j=1}^rx_j(8)^{a_{ij}(8)}\mbox{ for $1\le i\le r$ and}
$$
$$
x_{r+l+1}(7)=\left(\prod_{j=1}^rx_j(8)^{a_{r+1,j}(8)}\right)(x_{r+l+1}(8)+\beta)
$$
with $0\ne\beta\in\CC$.

By Lemma \ref{Lemma35}, we can extend the SGMT $(x(7))\rightarrow (x(8))$ to a transformation $(x(7),y(7))\rightarrow (x(8),y(8))$ of type 9) (where $(y(7))\rightarrow (y(8))$ is a SGMT in $y_1(7),\ldots,y_s(7)$). We have
$$
\left(\prod_{j=1}^rx_j(8)^{a_{r+1,j}(8)}\right)(x_{r+l+1}(8)+\beta)
=\left(\prod_{j=1}^sy_j(8)^{b_j(8)}\right)(y_{r+l+1}(8)+\alpha),
$$
with $\alpha,\beta\ne 0$. Then
$$
\sum_{j=1}^ra_{r+1,j}(8)\nu_e(x_j(8))=\sum_{j=1}^sb_j(8)\nu_e(y_j(8)).
$$
The values $\nu_e(y_1(8)),\ldots,\nu_e(y_s(8))$ are rationally independent by Lemma \ref{Lemma25}, so 
$$
(a_{r+1,1},\ldots,a_{r+1,r})
\left(\begin{array}{ccc}
c_{11}(8)&\cdots&c_{1s}(8)\\
&\vdots&\\
c_{r1}(8)&\cdots&c_{rs}(8)
\end{array}\right)
=(b_1(8),\ldots,b_s(8)).
$$
Thus
$$
\prod_{j=1}^rx_j(8)^{a_{r+1.j}(8)}=\prod_{j=1}^sy_j(8)^{b_j(8)}
$$
and $\alpha=\beta$, so $x_{r+l+1}(8)=y_{r+l+1}(8)$, giving an increase in $r+l$ (and $(s,r,l)$).

 In all cases, we have reached the conclusions of Proposition \ref{Theorem2} (under the assumption that (\ref{eq25}) holds).

Now suppose that an expression (\ref{eq25}) does not hold. Then $t>1$ in (\ref{eq29}) (by the implicit function theorem).
Now we will use a Tschirnhaus transformation to put $g$ into a good form, and perform a sequence of transformations that preserve the respective forms (\ref{eq15*}) or (\ref{eq16*}) and lead to a decrease 
$$
0<\mbox{ord }\overline g(0,\ldots,0,x_{r+l+1}(i))<t
$$
where $\overline g$ is the strict transform of $g$. The conclusions of the proposition will then follow by induction on $t$.

  By Lemma \ref{Lemma9}, we can make a change of variables, replacing $x_{r+l+1}(4)$ with $x_{r+l+1}(4)-\Phi$ for some $\Phi\in \CC\{\{x_1(4),\ldots,x_{r+l}(4)\}\}$ (a transformation of type 10)) to get an expression
\begin{equation}\label{eq28}
g=\tau_0x_{r+l+1}(4)^t+\tau_2x_{r+l+1}(4)^{t-2}+\cdots+\tau_t
\end{equation}
where $\tau_0\in \CC\{\{x_1(4),\ldots,x_{r+l+1}(4)\}\}$ is a unit and $\tau_i\in \CC\{\{x_1(4),\ldots,x_{r+l}(4)\}\}$.
If all $\tau_i=0$ for $i\ge 2$ then we are in case (\ref{eq25}), so we may suppose that some $\tau_i\ne0$ with $i\ge 2$.

By Lemma \ref{Prop1}, there exists a sequence of transformations of types 1) - 4) $(x(4),y(4))\rightarrow (x(5),y(5))$ making
$$
\tau_i=x_1(5)^{a_1^i}\cdots x_s(5)^{a_s^i}\overline u_i 
$$
for $2\le i$, where $\overline u_i\in \CC\{\{x_1(5),\ldots,x_{s+l}(5)\}\}$ is either a unit or zero. The forms of equations (\ref{eq15*}) and (\ref{eq23}) or of
(\ref{eq16*}) and (\ref{eq24}) (in the variables $x(5)$ and $y(5)$) are preserved by these transformations.

Now apply the argument following (\ref{eq25}) to $x_{r+l+1}(5)$ (in the place  of $x_{r+l+1}(4)-\Phi$ in (\ref{eq25})) to construct a sequence of  transformations of types 1) - 4) $(x(5),y(5))\rightarrow (x(6),y(6))$ to get
in the case when (\ref{eq15*}) holds,
\begin{equation}\label{eq34}
x_{r+l+1}(6)=y_1(6)^{n_1}\cdots y_s(6)^{n_s}\hat u
\end{equation}
with $\hat u\in \CC\{\{y_1(6),\ldots,y_{s+l}(6)\}\}$ a unit 
and in the case when (\ref{eq16*}) holds, we have
\begin{equation}\label{eq30}
x_{r+l+1}(6)=
\left\{\begin{array}{ll}
y_1(6)^{n_1}\cdots y_s(6)^{n_s}\hat u&\mbox{ with $\hat u\in \CC\{\{y_1(6),\ldots,y_{s+l+1}(6)\}\}$ a unit}\\
&\mbox{ and $y_1(6)^{n_1}\cdots y_s(6)^{n_s}$ algebraic over $x_1(6),\ldots,x_{r+l}(6)$, or}\\
y_1(6)^{n_1}\cdots y_s(6)^{n_s}F&\mbox{ with $F\in \CC\{\{y_1(6),\ldots,y_{s+l+1}(6)\}\}$}\\
&\mbox{such that $\mbox{ord }F(0,\ldots,0,y_{s+l+1}(6)) = 1$}
\end{array}\right.
\end{equation}

Suppose that (\ref{eq15*}) and (\ref{eq34}) hold and $y_1(6)^{n_1}\cdots y_s(6)^{n_s}$ is not algebraic over $x_1(6),\ldots,x_{r+l}(6)$. Then 
after a transformation of type 8) we have an expression
$$
x_{r+l+1}(6)=y_1(6)^{n_1}\cdots y_s(6)^{n_s}
$$
giving us  an increase in $r$ (and $(s,r,l)$) in (\ref{eq5}), so we have obtained the conclusions of Proposition \ref{Theorem2}. 

Suppose that (\ref{eq16*}) and (\ref{eq30}) hold, and we have that
$x_{r+l+1}(6)=y_1(6)^{n_1}\cdots y_s(6)^{n_s}F$ with 
$$
\mbox{ord }F(0,\ldots,0,y_{n+l+1}(6))=1.
$$
 Then replacing $y_{s+l+1}(6)$ with $F$ (a transformation of type 5)), we have  relations (\ref{eq5}) with 
$$
x_{r+l+1}(6)=y_1(6)^{n_1}\cdots y_s(6)^{n_s}y_{s+l+1}(6).
$$
If $y_1(6),\ldots,y_s(6),y_{s+l+1}(6)$ are independent, we have an increase in $s$ (and in $(s,r,l)$), and we have achieved the conclusions of Proposition \ref{Theorem2}, so we may suppose that 
$$
y_1(6),\ldots,y_s(6),y_{s+l+1}(6)
$$
 are dependent. If $x_1(6),\ldots,x_r(6),x_{r+l+1}(6)$ are independent, then we perform a transformation of type 6) $(x(6),y(6))\rightarrow (x(7),y(7))$ (with $\overline m=l+1$) to get 
$$
x_{r+l+1}(7)=y_1(7)^{n_1}\cdots y_s(7)^{n_s}(y_{s+l+1}(7)+\alpha)
$$
with $0\ne \alpha\in \CC$. Since $x_1(7),\ldots,x_r(7),x_{r+l+1}(7)$ are independent (and so 
$$
\nu_e(x_1(7)),\ldots,\nu_e(x_r(7)),\ldots,\nu_e(x_{r+l+1}(7))
$$
 are rationally independent), we must have that 
$y_1(7)^{n_1}\cdots y_s(7)^{n_s}$ is not algebraic over $x_1(7),\ldots,x_r(7)$.

Thus after a change of variables, multiplying $y_i(7)$ by an appropriate unit for $1\le i\le s$ (a transformation of type 8)), we obtain an expression (\ref{eq5}), with an increase in $r$ (and $(s,r,l)$). 

The remaining case in (\ref{eq34}) and (\ref{eq30}) is when we have an expression
\begin{equation}\label{eq31}
x_{r+l+1}(6)=y_1(6)^{\overline m_1}\cdots y_s(6)^{\overline m_s}\hat u
\end{equation}
where $\hat u\in \CC\{\{y_1(6),\ldots,y_{s+l+1}(6)\}\}$ is a unit and $y_1(6)^{\overline m_1}\cdots y_s(6)^{\overline m_s}$ is algebraic over $x_1(6),\ldots,x_{r+l}(6)$. We will presume that this case holds.

From (\ref{eq31}), we see that $\nu_e(x_{r+l+1}(6))$ is rationally dependent on $\nu_e(x_1(6)),\ldots,\nu_e(x_r(6))$, so by Lemma \ref{Lemma25},  $x_1(6),\ldots,x_r(6), x_{r+l+1}(6)$ are dependent. Thus  there exists by Lemma \ref{Lemma6} a SGMT
\begin{equation}\label{eq32}
\begin{array}{lll}
x_1(6)&=& x_1(7)^{a_{11}(7)}\cdots x_r(7)^{a_{1r}(7)}\\
&\vdots&\\
x_r(6)&=& x_1(7)^{a_{r1}(7)}\cdots x_r(7)^{a_{rr}(7)}\\
x_{r+l+1}(6)&=& x_1(7)^{a_{1}(7)}\cdots x_r(7)^{a_r(7)}(x_{r+l+1}(7)+\alpha)
\end{array}
\end{equation}
with $0\ne \alpha\in \CC$. Substituting into (\ref{eq28}) and performing a (monomial) SGMT in $x_1(7),\ldots,x_r(7)$ (which we incorporate into $x(6)\rightarrow x(7)$) we obtain an expression
$$
g=x_1(7)^{b_1}\cdots x_s(7)^{b_s}\overline g
$$
where
\begin{equation}\label{eq33}
\mbox{ord }\overline g(0,\ldots,0,x_{r+l+1}(7))<t.
\end{equation}

By Lemma \ref{Lemma35}, we can extend the SGMT $(x(6))\rightarrow (x(7))$ to a transformation $(x(6),y(6))\rightarrow (x(7),y(7))$ of type 9) (where $(y(6))\rightarrow (y(7))$ is a SGMT in $y_1(6),\ldots,y_s(6)$).

Writing $\overline g=x_1(7)^{-b_1}\cdots x_s(7)^{-b_s}g$, we see from (\ref{eq23}) or (\ref{eq24}) that $\overline g$ is not a unit in $\CC\{\{y_1(7),\ldots,y_{s+l+1}(7)\}\}$. Thus
$$
\mbox{ord }\overline g(0,\ldots,0,x_{r+l+1}(7))>0.
$$
Now $x_{r+l+1}(7)$ continues to have a form (\ref{eq15*}) or (\ref{eq16*}), and
 $\overline g$ has a form (\ref{eq23}) (if (\ref{eq15*}) holds) or a form (\ref{eq24}) (if (\ref{eq16*}) holds), in terms of the variables $x(7)$, $y(7)$. Thus we are in the situation after (\ref{eq24}) (replacing $g$ with $\overline g$), but by (\ref{eq33}), we have a reduction of $t$ in (\ref{eq29}). By induction in $t$, continuing to run the algorithm following (\ref{eq24}), we must eventually obtain the conclusions of Proposition \ref{Theorem2}.

\end{proof}

\begin{Proposition}\label{Theorem3} Suppose that $\phi:Y\rightarrow X$ is a  morphism of  complex analytic manifolds, $E_Y$ is a simple normal crossings divisor on $Y$ and $e$ is an \'etoile over $Y$. Suppose that $\phi$ is quasi regular with respect to $e$. Then $\phi$ is regular at $e_Y$ and there exists a commutative diagram
$$
\begin{array}{rcl}
Y_e&\stackrel{\phi_e}{\rightarrow}&X_e\\
\pi_e\downarrow&&\downarrow\lambda_e\\
Y&\stackrel{\phi}{\rightarrow}&X
\end{array}
$$
of regular analytic morphisms such that the vertical arrows are products of local blow ups of nonsingular analytic subvarieties,  $Y_e\rightarrow Y\in e$ and $\phi_e$ is a monomial morphism for a toroidal structure $O_e$ on $Y_e$ at $p$.  Further, we have that $\pi_e^*(E_Y)$ is an effective divisor supported on  $O_e$  and the restriction of $\pi_e$ to $Y_e\setminus O_e$ is an open embedding.
\end{Proposition}

\begin{proof} Let $x_1,\ldots,x_m$ be regular parameters in $\mathcal O_{X,e_X}^{\rm an}$ and $y_1,\ldots,y_n$ be regular parameters in $\mathcal O_{Y,e_Y}^{\rm an}$ such that $E_Y$ is supported on the analytic set $Z(y_1y_2\cdots y_n)$ (in a neighborhood of $e_Y$ in $Y$). After reindexing the $y_i$ we may assume that $s\ge 1$ is such that $y_1,\ldots,y_s$ are independent and $y_1,\ldots,y_s,y_i$ are dependent for all $i$ with $s+1\le i\le n$. 
After performing SGMT of type 6) for $1\le \overline m\le n-s$, we may assume that $E_Y$ is supported on $Z(y_1y_2\cdots y_s)$. 
Then $(x,y)$ are prepared of type $(s',0,0)$ with $s'\ge s$. By successive application of Proposition \ref{Theorem2}, we construct a sequence of transformations $(x,y)\rightarrow (x',y')$ such that $r'+l'=m$, giving the conclusions of the theorem. 

The fact that $\phi_e$ is regular at $e_{Y_e}$ follows from the rank theorem (page 134 \cite{L}) and the inequality (\ref{Gineq}) applied to the monomial morphism $\phi_e$. Thus $\phi$ is regular at $e_Y$ as $\pi_{e}$
and $\lambda_e$  are products of local blowups, so that they are open embeddings away from nowhere dense closed analytic subspaces.
\end{proof}

We isolate as a corollary one of the conclusions of Proposition \ref{Theorem3}.

\begin{Corollary}\label{RegQR} Suppose that $\phi:Y\rightarrow X$ is a morphism of connected complex analytic manifolds, $e$ is an \'etoile on $Y$ and $\phi$ is quasi regular with respect to $e$. Then $\phi$ is regular.
\end{Corollary}

Corollary \ref{RegQR} can also be deduced from the local flattening theorem of Hironaka, Lejeune and Teissier \cite{HLT} and Hironaka \cite{H3}, as is shown in \cite{WLM}.

\begin{Theorem}\label{TheoremB} Suppose that $\phi:Y\rightarrow X$ is a morphism of reduced complex analytic spaces, $A$ is a closed  analytic  subspace of $Y$  and $e$ is an \'etoile over $Y$. Then there exists a commutative diagram of complex analytic morphisms
$$
\begin{array}{ccc}
Y_{e}&\stackrel{\phi_{e}}{\rightarrow}&X_e\\
\beta\downarrow &&\downarrow \alpha\\
Y&\stackrel{\phi}{\rightarrow}& X
\end{array}
$$
such that $\beta\in e$, the morphisms $\alpha$ and $\beta$ are finite products of local blow ups of nonsingular analytic sub varieties, $Y_e$ and $ X_e$ are nonsingular analytic  spaces and $\phi_e$ is a  monomial  analytic morphism for a toroidal structure $O_e$ on $Y_e$ at $e_{Y_e}$ such that the restriction $(Y_e\setminus O_e) \rightarrow Y$ is an open embedding.  
There exists a nowhere dense closed analytic subspace $F_e$ of $X_e$ such that $X_e\setminus F_e\rightarrow X$ is an open embedding and $\phi_e^{-1}(F_e)$ is nowhere dense in $Y_e$.
Further, either the preimage  of $A$ in $Y_e$ is equal to $Y_e$, or $\mathcal I_A\mathcal O_{Y_e}=\mathcal O_{Y_e}(-G)$ where $\mathcal I_A$ is the ideal sheaf in $\mathcal O_Y^{\rm an}$ of the analytic subspace $A$  of $Y$ and $G$ is an effective divisor which is supported on $O_e$. \end{Theorem}

\begin{proof}
The proof  follows by first applying Proposition \ref{TheoremA} above to get a  morphism of smooth analytic spaces $Y_1\rightarrow X_1$, with closed analytic sub manifold $Z$ of $X_1$ such that $\phi_1(Y_1)\subset Z$ and if $\overline\phi_1:Y_1\rightarrow Z$ is the induced map, then $\overline\phi_1$ is quasi regular with respect to $e$. We may thus replace $X_1$ with $Z$ in the remainder of the proof, and assume that $\phi_1:Y_1\rightarrow X_1$ is quasi regular with respect to $e$ in the remainder of the proof. Either $\mathcal I_A\mathcal O_{Y_1}$ is the zero ideal sheaf, which holds if and only if the preimage of $A$ in $Y_1$ is $Y_1$, or $\mathcal I_A\mathcal O_{Y_1}$ is a nonzero ideal sheaf.
Let $B$ be a proper closed analytic subspace of $Y_1$ such that $Y_1\setminus B\rightarrow Y$ is an open embedding.  
Then applying  principalization of ideals and embedded resolution of singularities by blowing up nonsingular sub varieties to $\mathcal I_B$ if $\mathcal I_A= (0)$ and to
$\mathcal I_A \mathcal I_{B}$ if $\mathcal I_A\ne (0)$,
 we construct $Y_2\rightarrow X_1$ 
such that either $\mathcal I_A\mathcal O_{Y_2}=(0)$ and $\mathcal I_B\mathcal O_{Y_2}=\mathcal O_{Y_2}(-G)$ or $\mathcal I_A\mathcal  I_{B}\mathcal O_{Y_2}=\mathcal O_Y(-G)$ where $G$ is a simple normal crossings divisor on $Y_2$ which satisfies  the assumptions  of Proposition  \ref{Theorem3}, and $Y_2\setminus G\rightarrow Y$ is an open embedding. We then apply Proposition \ref{Theorem3}  to $Y_2\rightarrow X_1$ to obtain a monomial morphism at the center of $e$, satisfying the conclusions of the theorem.
\end{proof}

We obtain Theorem \ref{TheoremB*} of the introduction as an immediate consequence of Theorem \ref{TheoremB}.

\begin{Theorem}\label{TheoremC} Suppose that $\phi:Y\rightarrow X$ is a morphism of reduced  complex  analytic spaces, $A$ is a closed  analytic  subspace of $Y$  and $p\in Y$. Then there exists a finite number $t$ of commutative diagrams of complex  analytic morphisms
$$
\begin{array}{ccc}
Y_{i}&\stackrel{\phi_{i}}{\rightarrow}&X_i\\
\beta_i\downarrow &&\downarrow \alpha_i\\
Y&\stackrel{\phi}{\rightarrow}& X
\end{array}
$$
for $1\le i \le t$ 
such that each $\beta_i$ and $\alpha_i$ is a finite product of local blow ups of nonsingular analytic sub varieties, $Y_i$ and $ X_i$ are smooth analytic  spaces and $\phi_i$ is a  monomial  analytic morphism for a toroidal structure $O_i$ on $Y_i$. 
Either the preimage of $A$ in  $Y_i$ is $Y_i$ or $\mathcal I_A\mathcal O_{Y_i}=\mathcal O_{Y_i}(-G_i)$ where $\mathcal I_A$ is the ideal sheaf in $\mathcal O_Y^{\rm an}$ of the analytic subspace $A$  of $Y$, $G_i$ is an effective divisor which is supported on $O_i$, and has the further property that the restriction $(Y_i\setminus O_i) \rightarrow Y$ is an open embedding. Further, there exist compact subsets $K_i$ of $Y_i$ such that $\cup_{i=1}^t\beta_i(K_i)$ is a compact neighborhood of $p$ in $Y$. 
There exist nowhere dense closed analytic subspaces $F_i$ of $X_i$ such that $X_i\setminus F_i\rightarrow X$ are open embeddings and $\phi_i^{-1}(F_i)$ is nowhere dense in $Y_i$.
\end{Theorem}

\begin{proof}
Let $\mathcal E_Y$ be the vo\^ute \'etoil\'ee over $Y$, with canonical map $P_Y:\mathcal E_Y\rightarrow Y$ defined by $P_Y(e)=e_Y$. We summarized in Section \ref{Pre}  properties of $\mathcal E_Y$ which we require in this proof. By Theorem \ref{TheoremB}, for each $e\in \mathcal E_Y$ we have a commutative diagram
$$
\begin{array}{rcl}
Y_e&\stackrel{\phi_e}{\rightarrow}&X_e\\
\pi_e\downarrow&&\downarrow\\
Y&\stackrel{\phi}{\rightarrow}&X
\end{array}
$$
such that $\phi_e$ is monomial at $e_{Y_e}$ and satisfies the other conditions of the conclusions of Theorem \ref{TheoremB}. Let $V_e$ be an open relatively compact neighborhood of $e_{Y_e}$ in $Y_e$. Let $\overline\pi_e:V_e\rightarrow Y$ be the induced maps. Let $K$ be a compact neighborhood of $p$ in $Y$ and $K'=P_Y^{-1}(K)$. The set $K'$ is compact since $P_Y$ is proper (Theorem 3.4 \cite{Hcar}). The sets 
$\mathcal E_{\overline \pi_e}$ (see equation (\ref{eqopen})) give an open cover of $K'$, so there is a finite subcover, which we reindex as $\mathcal E_{\overline \pi_{e_1}},\ldots,\mathcal E_{\overline \pi_{e_t}}$. For $1\le i\le t$, let $K_i$ be the closure of $V_{e_i}$ in $Y_{e_i}$ which is compact.  Since $P_Y$ is surjective and continuous, we have inclusions of compact sets 
$$
p\in K\subset \cup_{i=1}^t \pi_{e_i}(K_i) 
$$
giving the conclusions of the theorem.
\end{proof}

We obtain Theorem \ref{TheoremC*}  of the introduction as an immediate consequence of Theorem \ref{TheoremC}.

\section{Monomialization of real analytic maps}\label{Real}

In this section we prove local monomialization theorems for real analytic morphisms. 
We use the method of complexifications of real analytic spaces developed in Section 1 of \cite{H3}.

\begin{Remark}\label{RemarkR1}
Resolution of singularities of a germ of a complex analytic space $(X,x)$, which has a natural  auto conjugation, can be accomplished by blowing up smooth analytic sub varieties which are preserved by the auto conjugation. This follows by applying the basic theorem of resolution of singularities in \cite{H0} (or \cite{BM1}) to the spectrum of the invariant analytic local ring $\mbox{Spec}((\mathcal O_{X,x}^{\rm an})^{\sigma})$ by the action of the auto conjugation $\sigma$ of $X$ and then extending to  $\mbox{Spec}(\mathcal O_{X,x}^{\rm an})$. We also need the fact that a principalization of a sheaf of ideals  which is invariant under $\sigma$ can be obtained by blowing up smooth analytic sub varieties which are preserved by the auto conjugation  (this also follows from \cite{H0}).
\end{Remark}

\begin{Lemma}\label{LemmaR2}  Suppose that $Y$ is a smooth connected real analytic variety with a complexification $\tilde Y$ which is a smooth connected complex variety. Suppose that $Z\subset Y$ is  a closed real analytic subspace of $Y$ such that its complexification $\tilde Z\subset \tilde Y$ is a nowhere dense closed complex analytic subspace of $\tilde Y$. Then $Z$ is nowhere dense in $Y$ (in the Euclidean topology).
\end{Lemma}

The necessity that $Y$ be smooth in the lemma can be seen from consideration of the Whitney Umbrella $x^2-zy^2=0$.

\begin{proof} Since $\tilde Y$ and $Y$ are manifolds, for all $p\in Y$, the topological dimension of $Y$ at $p$, T-$\dim_pY$ (Remarks 5.16 and 5.17 \cite{H3} and Section \ref{Pre}), is equal to the  dimension $\dim_p \tilde Y$ of $\tilde Y$ (Section \ref{Pre}), which is equal to $\dim \mathcal O^{\rm an}_{\tilde Y,p}$. Since $\tilde Z$ is a nowhere dense analytic subspace of the manifold $\tilde Y$, we have that
$$
\dim_p\tilde Z=\dim \mathcal O_{\tilde Z,p}^{\rm an}<\dim \mathcal O_{\tilde Y,p}^{\rm an}=\dim_p\tilde Y.
$$
Since $Z=\tilde Z\cap Y$, we have that
T-$\dim_pZ\le \dim_p \tilde Z$ for all $p\in Z$. Since $Z$ is closed in $Y$, we have that  $Z$ is nowhere dense in $Y$.
\end{proof}

\begin{Lemma}\label{LemmaReg} Suppose that $\phi:Y\rightarrow X$ is a morphism of connected smooth real analytic varieties and $\tilde\phi:\tilde Y\rightarrow\tilde X$ is a complexification of $\phi$ (with $\tilde Y$ and $\tilde X$ smooth). Then $\phi$ is  regular if and only if $\tilde\phi$ is regular.
\end{Lemma}

\begin{proof} Suppose that $\tilde\phi$ is regular. Let $n=\dim \tilde X$. Then the closed analytic subspace
$$
\tilde Z=\{\tilde q\in \tilde Y\mid {\rm rank}(d\tilde\phi_{\tilde q})<n\}
$$
is a proper subset of $\tilde Y$. Suppose that $\phi:Y\rightarrow X$ is not regular. Then
$$
Y=\{q\in Y\mid {\rm rank}(d\phi_q)<n\}=\tilde Z\cap Y,
$$
a contradiction to Lemma \ref{LemmaR2}.  

A simpler argument shows that if $\phi$ is regular then $\tilde\phi$ is regular.
\end{proof}

\begin{Proposition}\label{TheoremA'} Suppose that $\phi:Y\rightarrow X$ is a morphism of reduced real analytic spaces with complexification $\tilde\phi:\tilde Y\rightarrow \tilde X$, such that there are  auto conjugations $\sigma:\tilde X\rightarrow \tilde X$ and $\tau:\tilde Y\rightarrow \tilde Y$ such that $\tilde \phi\tau=\sigma\tilde\phi$. Let $e\in \mathcal E_{\tilde Y}$ be an \'etoile over $\tilde Y$. Then there exists a commutative diagram of morphisms 
$$
\begin{array}{rcl}
\tilde Y_e&\stackrel{\tilde \phi_e}{\rightarrow}&\tilde X_e\\
\tilde\delta\downarrow&&\downarrow\tilde\gamma\\
\tilde Y&\stackrel{\tilde \phi}{\rightarrow}&\tilde X
\end{array}
$$
such that $\tilde\delta\in e$,  $\tilde Y_e$ and $\tilde X_e$ are smooth analytic spaces, 
there exists a closed analytic sub manifold $\tilde Z_e$ of $\tilde X_e$ such that $\tilde \phi_e(\tilde Y_e)\subset \tilde Z_e$ and the induced analytic map $\tilde \phi_e:\tilde Y_e\rightarrow \tilde Z_e$ is  regular. 
Further, there exists a nowhere dense closed analytic subspace $\tilde F_e$ of $\tilde X_e$ such that $\tilde X_e\setminus \tilde F_e\rightarrow \tilde X$ is an open embedding and $\tilde\phi_e^{-1}(\tilde F_e)$ is nowhere dense in $\tilde Y_e$.

There exist auto conjugations $\sigma_e:\tilde X_e\rightarrow \tilde X_e$ and $\tau_e:\tilde Y_e\rightarrow \tilde Y_e$ which are compatible with the diagram. We  have that 
$\sigma_e(\tilde Z_e)=\tilde Z_e$ and $\sigma_e(\tilde F_e)=\tilde F_e$.

Further, we have a factorization of $\tilde\delta$ as
$$
\tilde Y_e=W_s\stackrel{\beta_{s-1}}{\rightarrow} W_{s-1}\rightarrow \cdots\rightarrow W_1\stackrel{\beta_0}{\rightarrow} W_0=\tilde Y
$$
where each $\beta_i$ is a local blow up $(U_i,E_i,\beta_i)$ where $E_i$ is a smooth sub variety of $U_i$ and there are auto conjugations $\tau_i:W_i\rightarrow W_i$  such that 
$\beta_i\tau_{i+1}=\tau_i\beta_i$ for all $i$.
We have that $\tau_0=\tau$ and  $\tau_s=\tau_e$ and   $\tau_i(U_i)=U_i$ and $\tau_i(E_i)=E_i$ for all $i$. Further, either the center $e_{W_i}$ of $e$ on $W_i$ is a real point $(\tau_i(e_{W_i})=e_{W_i})$ or $\tau_i(e_{W_i})\ne  e_{W_i}$ and $U_i$ is the disjoint union of two open subsets $S_i$ and $\tau_i(S_i)$ which are respective open neighborhoods of $e_{W_i}$ and $\tau_i(e_{W_i})$.  

We also have a factorization of $\tilde \gamma$ by 
$$
\tilde X_e=Z_r\stackrel{\alpha_{r-1}}{\rightarrow} Z_{r-1}\rightarrow \cdots \rightarrow Z_1\stackrel{\alpha_0}{\rightarrow}Z_0=\tilde X
$$
where each $\alpha_i$ is a local blow up $(V_i,H_i,\alpha_i)$ where $H_i$ is a smooth sub variety of $V_i$, and there are auto conjugations $\sigma_i:Z_i\rightarrow Z_i$ such that $\alpha_i\sigma_{i+1}=\sigma_i\alpha_i$, $\sigma_i(V_i)=V_i$ and $\sigma_i(H_i)=H_i$. Further either $q_i:=\alpha_i\cdots\alpha_{r-1}\tilde \phi_e(e_{\tilde Y_e})$ is a real point ($\sigma_i(q_i)=q_i$) or $\sigma_i(q_i)\ne q_i$ and $V_i$ is the disjoint union of two open subsets $T_i$ and $\sigma_i(T_i)$ which are respective open neighborhoods of $q_i$ and $\sigma_i(q_i)$.
We have that $\sigma_0=\sigma$ and $\sigma_r=\sigma_e$.

\end{Proposition}

We obtain the conclusions of  Proposition \ref{TheoremA'}, by modifying the proof of Proposition \ref{TheoremA}, using Corollary \ref{RegQR} and Remark \ref{RemarkR1}.

\begin{Proposition}\label{Theorem5} 
 Suppose that $\phi:Y\rightarrow X$ is a regular morphism of real analytic manifolds, $E_Y$ is a SNC divisor on $Y$ with complexification $\tilde\phi:\tilde Y\rightarrow \tilde X$ where $\tilde Y$ and $\tilde X$  are complex analytic manifolds and complexification $E_{\tilde Y}$ of $E_Y$ which is a SNC divisor on $\tilde Y$.
 Let $e$ be an \'etoile over $\tilde Y$. Then there exists a commutative diagram
$$
\begin{array}{rcl}
\tilde Y_e&\stackrel{\tilde \phi_e}{\rightarrow}&\tilde X_e\\
\tilde\pi_e\downarrow&&\downarrow\tilde\lambda_e\\
\tilde Y&\stackrel{\tilde \phi}{\rightarrow}&\tilde X
\end{array}
$$
of regular complex analytic morphisms such that the vertical arrows are products of local blow ups of nonsingular analytic subvarieties such that  
$\tilde Y_e\rightarrow \tilde Y\in e$.

The vertical arrows have factorizations by sequences of local blow ups
\begin{equation}\label{N}
\begin{array}{rcccccl}
\tilde Y_e=W_t&\stackrel{\beta_t}{\rightarrow}& \cdots &\rightarrow &W_1&\stackrel{\beta_1}{\rightarrow}& W_0=\tilde Y\\
\downarrow &&&&\downarrow&&\downarrow\\
\tilde X_e=V_t&\stackrel{\alpha_t}{\rightarrow}& \cdots& \rightarrow &V_1&\stackrel{\alpha_1}{\rightarrow}& V_0=\tilde X
\end{array}
\end{equation}
with complex auto conjugations of the $W_i$ and $V_i$ which are compatible with the above diagram, so that taking the invariants of these auto conjugations, we have an induced diagram of regular real analytic morphisms
$$
\begin{array}{rcl}
Y_e&\stackrel{\phi_e}{\rightarrow}& X_e\\
\pi_e\downarrow&&\downarrow\lambda_e\\
Y&\stackrel{\phi}{\rightarrow}& X
\end{array}
$$
 such that the vertical arrows are products of local blow ups of nonsingular real analytic subvarieties. The auto conjugations of $W_i$ induce auto conjugations of the preimage of $\tilde E_Y$ on $W_i$.

Either all $e_{W_i}$ (and $e_{V_i}$) are real points in the diagram (\ref{N}), and $\phi_e$ and $\tilde\phi_e$ are monomial morphisms for toroidal structures $ O_e$ on $Y_e$ with complexification $\tilde O_e$ on $\tilde Y_e$ or 
 $e_{\tilde Y_e}$ is not a real point, and $Y_e$ is the empty set.
 
 Further, we have that $\tilde\pi_e^*(E_{\tilde Y})$ is an effective divisor supported on $\tilde O_e$ and the restriction of $\tilde\pi_e$ to $\tilde Y_e\setminus \tilde O_e$ is an open embedding.  Also, $\pi^*(E_Y)$ is an effective divisor supported on $O_e$, and the restriction of $\pi_e$ to $Y_e\setminus O_e$ is an open embedding.
\end{Proposition}

\begin{proof} 
We inductively construct the diagram (\ref{N}) of local blow ups as in the proof of Proposition \ref{Theorem3}, with the following differences. 
If after construction of the local blow up $W_i\rightarrow W_{i-1}$ we find that $e_{W_i}$ is not a real point then we take a neighborhood $U$ of $e_{W_i}$ which contains no real points and set $W_{i+1}$ to be the (disjoint) union of $U$ and $\sigma(U)$ where $\sigma$ is the auto conjugation of $W_i$. We then terminate the algorithm, setting $\tilde Y_e=W_{i+1}$ and $\tilde X_e=V_i$. 

In our inductive construction of (\ref{N}), as long as $e_{W_j}$ are real points for $j\le i$, the sequences of local blow ups in (\ref{N}) are complexifications of sequences of real local blow ups of nonsingular real analytic subvarieties. This follows from the algorithms of Proposition \ref{Theorem3}, as we then work within the rings 
$$
\begin{array}{ccc}
\RR\{\{x_1,\ldots,x_m\}\}&\rightarrow &\RR\{\{y_1,\ldots,y_n\}\}\\
\downarrow&&\downarrow\\
\RR[[x_1,\ldots,x_m]]&\rightarrow &\RR[[y_1,\ldots,y_n]]
\end{array}
$$
instead of in the corresponding complexifications of these rings. 

The only modification which needs to be made in the algorithm (since we assume all centers of $e$ are real) is that a little more care is needed when taking roots of unit series. For instance, in Lemma \ref{Lemma10}, we must insist that the constant term of the unit $\gamma$ is positive. This leads to the introduction of factors of $\pm1$ in the equations of Lemmas \ref{Lemma6}, \ref{Lemma23} and \ref{Lemma35}. To preserve the monomial form (\ref{eq7}), we may have to  replace some of the $y_j(1)$ with their negatives $-y_j(1)$ and some of the $x_i(1)$ with their negatives $-x_i(1)$. We also need the conclusions of Lemma \ref{LemmaReg}.
\end{proof}

\begin{Proposition}\label{TheoremB'} Suppose that $\phi:Y\rightarrow X$ is a morphism of reduced real  analytic spaces and $A\subset Y$ is a closed analytic subspace of $Y$, with complexification $\tilde\phi:\tilde Y\rightarrow\tilde X$ of $\phi$ and complexification $\tilde A\subset \tilde Y$ of $A$. Let $e$ be an \'etoile over $\tilde Y$. Then there exists a commutative diagram
$$
\begin{array}{rcl}
\tilde Y_e&\stackrel{\tilde\phi_e}{\rightarrow}&\tilde X_e\\
\tilde \beta\downarrow&&\downarrow\tilde \alpha\\
\tilde Y&\stackrel{\tilde\phi}{\rightarrow}&\tilde X
\end{array}
$$
of complex analytic morphisms such that $\tilde Y_e$ and $\tilde X_e$ are smooth analytic spaces, $\tilde\beta\in e$ and
we have a factorization of $\tilde \beta$ as
$$
\tilde Y_e=W_s\stackrel{\beta_{s-1}}{\rightarrow} W_{s-1}\rightarrow \cdots\rightarrow W_1\stackrel{\beta_0}{\rightarrow} W_0=\tilde Y
$$
where each $\beta_i$ is a local blow up $(U_i,E_i,\beta_i)$ where $E_i$ is a smooth sub variety of $U_i$ and there are auto conjugations $\tau_i:W_i\rightarrow W_i$  such that 
$\beta_i\tau_{i+1}=\tau_i\beta_i$ for all $i$.
We have that $\tau_0=\tau$ and  $\tau_s=\tau_e$ and   $\tau_i(U_i)=U_i$ and $\tau_i(E_i)=E_i$ for all $i$. Further, either the center $e_{W_i}$ of $e$ on $W_i$ is a real point $(\tau_i(e_{W_i})=e_{W_i})$ or $\tau_i(e_{W_i})\ne  e_{W_i}$ and $U_i$ is the disjoint union of two open subsets $S_i$ and $\tau_i(S_i)$ which are respective open neighborhoods of $e_{W_i}$ and $\tau_i(e_{W_i})$.  

We also have a factorization of $\tilde \alpha$ by 
$$
\tilde X_e=Z_r\stackrel{\alpha_{r-1}}{\rightarrow} Z_{r-1}\rightarrow \cdots \rightarrow Z_1\stackrel{\alpha_0}{\rightarrow}Z_0=\tilde X
$$
where each $\alpha_i$ is  a local blow up $(V_i,H_i,\alpha_i)$ where $H_i$ is a smooth sub variety of $V_i$, and there are auto conjugations $\sigma_i:Z_i\rightarrow Z_i$ such that $\alpha_i\sigma_{i+1}=\sigma_i\alpha_i$, $\sigma_i(V_i)=V_i$ and $\sigma_i(H_i)=H_i$. Further either $q_i:=\alpha_i\cdots\alpha_{r-1}\tilde \phi_e(e_{\tilde Y_e})$ is a real point ($\sigma_i(q_i)=q_i$) or $\sigma_i(q_i)\ne q_i$ and $V_i$ is the disjoint union of two open subsets $T_i$ and $\sigma_i(T_i)$ which are respective open neighborhoods of $q_i$ and $\sigma_i(q_i)$.
We have that $\sigma_0=\sigma$ and $\sigma_r=\sigma_e$. Further, there exists a nowhere dense closed analytic subspace $\tilde F_e$ of $\tilde X_e$ such that $\sigma_e(\tilde F_e)=\tilde F_e$, $\tilde X_e\setminus \tilde F_e\rightarrow \tilde X$ is an open embedding and $\tilde\phi_e^{-1}(\tilde F_e)$ is nowhere dense in $\tilde Y_e$.

Taking the invariants of these auto conjugations, we have an induced diagram of real analytic morphisms
$$
\begin{array}{rcl}
Y_e&\stackrel{\phi_e}{\rightarrow}& X_e\\
\beta\downarrow&&\downarrow\alpha\\
Y&\stackrel{\phi}{\rightarrow}& X
\end{array}
$$
such that the vertical arrows are products of local blow ups of nonsingular real analytic subvarieties. 
Either all $e_{W_i}$ (and $e_{V_i}$) are real points  and $\phi_e$ and $\tilde\phi_e$ are monomial morphisms for toroidal structures $ O_e$ on $Y_e$ at $e_{\tilde Y_e}$ with complexification $\tilde O_e$ on $\tilde Y_e$ or 
 $e_{\tilde Y_e}$ is not a real point, and $Y_e$ is the empty set.
 
 Further, either the preimage of $\tilde A$ in $\tilde Y_e$ is equal to $\tilde Y_e$ or $\mathcal I_{\tilde A}\mathcal O_{\tilde Y_e}^{\rm an}=\mathcal O_{\tilde Y_e}^{\rm an}(-G)$ where $\mathcal I_{\tilde A}$ is the ideal sheaf in $\mathcal O_{\tilde Y}^{\rm an}$ of the analytic subspace $\tilde A$ of $\tilde Y$, $\tilde G$ is an effective divisor which is supported on $\tilde O_e$ and $\tilde Y_e\setminus \tilde O_e\rightarrow \tilde Y$ is an open embedding. 
 We have that $\tau_e(G)=G$.
 
 Suppose that $e_{\tilde Y_e}$ is real. Then $F_e=\tilde F_e\cap X_e$ is nowhere dense in $X_e$, $X_e\setminus F_e\rightarrow X$ is an open embedding and $\phi_e^{-1}(F_e)$ is nowhere dense in $Y_e$.

\end{Proposition}

We obtain Proposition \ref{TheoremB'} by arguing as in the proof of Theorem \ref{TheoremB}, using Proposition \ref{TheoremA'},  Remark \ref{RemarkR1}, Proposition \ref{Theorem5}, Lemma \ref{LemmaReg} and Lemma \ref{LemmaR2}.

We have the following theorem, which generalizes Theorem \ref{TheoremC} to a real analytic morphism from a real analytic manifold.

\begin{Theorem}\label{TheoremC'} Suppose that $Y$ is a real analytic manifold, $X$ is a reduced real analytic space, $\phi:Y\rightarrow X$ is a real analytic morphism, $A$ is a closed analytic  subspace of $Y$  and $p\in Y$. Then there exists a finite number $t$ of commutative diagrams of real  analytic morphisms
$$
\begin{array}{ccc}
Y_{i}&\stackrel{\phi_{i}}{\rightarrow}&X_i\\
\beta_i\downarrow &&\downarrow \alpha_i\\
Y&\stackrel{\phi}{\rightarrow}& X
\end{array}
$$
for $1\le i \le t$ 
such that each $\beta_i$  and $\alpha_i$ is a finite product of local blow ups of nonsingular analytic sub varieties, $Y_i$ and $ X_i$ are smooth analytic  spaces and $\phi_i$ is a monomial  analytic morphism for a toroidal structure $O_i$ on $Y_i$. Either the preimage of $A$ in $Y_i$ is $Y_i$, or $\mathcal I_A\mathcal O_{Y_i}=\mathcal O_{Y_i}(-G_i)$ where $\mathcal I_A$ is the ideal sheaf in $\mathcal O_Y^{\rm an}$ of the analytic subspace $A$  of $Y$, $G_i$ is an effective divisor which is supported on $O_i$, and has the further property that the restriction $(Y_i\setminus O_i) \rightarrow Y$ is an open embedding. Further, there exist compact subsets $K_i$ of $Y_i$ such that $\cup_{i=1}^t\beta_i(K_i)$ is a compact neighborhood of $p$ in $Y$.
There exist nowhere dense closed analytic subspaces $F_i$ of $X_i$ such that $X_i\setminus F_i\rightarrow X$ are open embeddings and $\phi_i^{-1}(F_i)$ is nowhere dense in $Y_i$.
\end{Theorem}

\begin{proof}
Let $\tilde\phi:\tilde Y\rightarrow \tilde X$ be a complexification of $\phi$ such that $\tilde Y$ is nonsingular. 

Suppose that 
$e\in \mathcal E_{\tilde Y}$ (the vo\^ute \'etoil\'ee and the notation used in this proof are reviewed in Section \ref{Pre}). Then we may construct a diagram  satisfying 
 the conclusions of Proposition \ref{TheoremB'}
 $$
 \begin{array}{rcl}
 \tilde Y_e&\stackrel{\tilde\phi_e}{\rightarrow}&\tilde X_e\\
 \tilde\beta_e\downarrow&&\downarrow \tilde\alpha_e\\
 \tilde Y&\stackrel{\tilde\phi}{\rightarrow}&\tilde X
 \end{array}
 $$
 with real part
 $$
 \begin{array}{rcl}
 Y_e&\stackrel{\phi_e}{\rightarrow}&X_e\\
 \beta_e\downarrow&&\downarrow\alpha_e\\
 Y&\stackrel{\phi}{\rightarrow}&X.
 \end{array}
 $$ 
 (We can have $Y_e=\emptyset$).
 
 Let $\tilde C_e$ be an open relatively compact neighborhood of $e_{\tilde Y_e}$ in $\tilde Y_e$ on which the auto conjugation acts.  Let $\overline\beta_e:\tilde C_e\rightarrow \tilde Y$ be the induced map. 
 
 Let $K$ be a compact neighborhood of $p$ in $\tilde Y$ and $K'=P_{\tilde Y}^{-1}(K)$. The set $K'$ is compact since $P_{\tilde Y}$ is proper (Theorem 3.4 \cite{Hcar}). The open sets $\mathcal E_{\overline \beta_e}$ for $e\in K'$  (defined in equation (\ref{eqopen})) give an open cover of $K'$, so there is a finite subcover, which we index as
$\mathcal E_{\overline \beta_{e_1}},\ldots,\mathcal E_{\overline \beta_{e_t}}$. Let $K_i$ be the closure of $\tilde C_{e_i}$ in $\tilde Y_{e_i}$ which is compact.  Since $P_{\tilde Y}$ is surjective and continuous, we have inclusions of compact sets $p\in K\subset \cup_{i=1}^t\tilde\beta_{e_i}(K_i)$.
Since $\tilde Y$ is nonsingular and each $\tilde\beta_{e_i}$ is a (finite) product of local blow ups of proper sub varieties, if $\tilde H_{e_i}$ is the union of the
preimages on $\tilde Y_{e_i}$ of these centers, then $\tilde H_{e_i}$ is a nowhere dense closed analytic subspace of $\tilde Y_{e_i}$ and $\tilde\beta_{e_i}$ is an open embedding of $\tilde Y_{e_i}\setminus \tilde H_{e_i}$ into $\tilde Y$. 

Suppose that $q\in Y$. Then ${\rm T}$-$\dim_qY=\dim_q\tilde Y$ since $Y$ is a manifold (Section \ref{Pre} and Section 5 of \cite{H3}). Suppose $i$ satisfies $1\le i\le t$. The set $\tilde H_{e_i}\cap K_i$ is compact and $\tilde\beta_{e_i}(\tilde H_{e_i}\cap K_i)$ is compact. Let $M_i=\tilde\beta_{e_i}(\tilde H_{e_i}\cap K_i)\cap Y$. Suppose $q\in M_i$. Then
$$
\dim_q\tilde\beta_{e_i}(\tilde H_{e_i}\cap K_i)<\dim \tilde Y
$$
by Theorem 1, page 254 \cite{L} or Corollary 1, page 255 \cite{L}. Thus
$$
{\rm T}\mbox{-}\dim_qM_i\le \dim_q\tilde\beta_{e_i}(\tilde H_{e_i}\cap K_i)<\dim_q\tilde Y={\rm T}\mbox{-}\dim_qY.
$$
Since $M_i$ is compact, we have that $M_i$ is nowhere dense in $Y$.

Let $K^*=K\cap Y$ which is a compact neighborhood of $p$ in $Y$. Let $p'\in K^*\setminus \cup_{i=1}^t\tilde \beta_{e_i}(\tilde H_{e_i}\cap K_i)$. Then there exist $i$ and $e\in \mathcal E_{\overline \beta_{e_i}}$ such that $e_{\tilde Y}=p'$ and $p_i=e_{\tilde Y_{e_i}}\in K_i\setminus \tilde H_{e_i}\subset \tilde Y_{e_i}$. Since $p_i\not\in 
\tilde H_{e_i}$, $\tilde \beta_{e_i}$ is an open embedding near $p_i$, and since $p'$ is real, $p_i\in Y_{e_i}$ is real. Thus $p'\in \beta_{e_i}(K_i\cap Y_{e_i})$. We thus have that the set
$K^*\setminus \cup_{i=1}^t\tilde \beta_{e_i}(\tilde H_{e_i}\cap K_i)$, which we have shown is dense in $K^*$, is contained in the compact set $\cup_{i=1}^t\beta_{e_i}(K_i\cap Y_{e_i})$. Thus its closure $K^*$ is contained in $\cup_{i=1}^t\beta_{e_i}(K_i\cap Y_{e_i})$, giving the conclusions of Theorem \ref{TheoremC'}.
\end{proof}
\vskip .2truein
Theorems \ref{TheoremC'*} and \ref{TheoremD*}  of the introduction follow from Theorem \ref{TheoremC'}.

\end{document}